\documentclass[11pt]{book}
\usepackage[a4paper,width=145mm,top=30mm,bottom=30mm]{geometry}
\usepackage{hyperref}
\usepackage{amsfonts,amsmath,amssymb,bm,cmll}
\usepackage{enumitem}
\usepackage{bussproofs}
\usepackage{multirow}
\usepackage{wasysym}
\usepackage{amsthm}
\usepackage{thmtools}
\usepackage{verbatim}
\usepackage{stmaryrd}
\usepackage{setspace}
\usepackage[figuresright]{rotating}
\usepackage{float}
\restylefloat{table}
\usepackage{tikz}
\usetikzlibrary{matrix,arrows,decorations.pathmorphing,backgrounds,decorations.markings}
\usepackage{graphicx}
\usepackage{caption}
\usepackage{subcaption}
\usepackage{afterpage}

\newcommand{\op}{\textup{op}}
\newcommand{\true}{\textup{true}}
\newcommand{\false}{\textup{false}}
\newcommand{\ob}[1]{\textup{Ob}(#1)}
\newcommand{\Hom}{\textup{Hom}}
\newcommand{\myand}{\,\&\,}
\newcommand{\bline}{\begin{spacing}{0.5} \end{spacing}}
\newcommand{\db}[1]{{\llbracket #1 \rrbracket}}
\newcommand{\lcs}[1]{\mathbf{#1}}
\newcommand{\ind}[1]{{\textup{Ind}(#1)}}
\newcommand{\funct}[2]{{\textup{Fun}(#1,#2)}}
\newcommand{\unit}[1]{{\eta_{#1}}}
\newcommand{\counit}[1]{{\varepsilon_{#1}}}
\newcommand{\counitinv}[1]{{\varepsilon_{#1}^\prime}}

\newcommand{\wfig}{0.32\textwidth}
\newcommand{\wfigg}{0.375\textwidth}
\newcommand{\csize}{1.5pt}

\tikzset{->-/.style={decoration={
  markings,
  mark=at position #1 with {\arrow{>}}},postaction={decorate}}}
\tikzset{my grid/.style={very thin,black!70!white}}
\tikzset{my cood/.style={black}}
\tikzset{my arrow/.style={-> }}
\tikzset{my mapsto/.style={|-> }}
\tikzset{my coordarrow/.style={->}}
\tikzset{my frame/.style={very thick}}

\makeatletter
\newcommand{\dotminus}{\mathbin{\text{\@dotminus}}}
\newcommand{\@dotminus}{%
  \ooalign{\hidewidth\raise1ex\hbox{.}\hidewidth\cr$\m@th-$\cr}%
}
\makeatother

\numberwithin{equation}{chapter}

\declaretheorem[style=plain,numberwithin=chapter,name=Theorem]{thm}
\declaretheorem[style=plain,sibling=thm,name=Lemma]{lem}
\declaretheorem[style=plain,sibling=thm,name=Proposition]{prop}
\declaretheorem[style=plain,sibling=thm,name=Corollary]{cor}

\declaretheorem[style=definition,qed=$\blacksquare$,sibling=thm,name=Definition]{defn}
\declaretheorem[style=definition,qed=$\blacksquare$,sibling=thm,name=Example]{ex}

\title{A Categorical Approach to L-Convexity}
\author{Soichiro Fujii}
\begin{document}
\frontmatter
\begin{titlepage}
    \begin{center}
        \vspace*{3cm}
 
        \Huge
        A Categorical Approach to L-Convexity
        
        \vspace{9cm}
        \LARGE{Soichiro Fujii}
 
        \vfill
 \LARGE{Bachelor's Thesis\\ \vspace{1cm}
 February, 2014\\
 Mathematical Information Engineering Course\\
 Department of Mathematical Engineering and Information Physics\\
 Faculty of Engineering, The University of Tokyo}
 
    \end{center}
\end{titlepage}

\chapter*{Abstract}
We investigate an enriched-categorical approach to a
field of discrete mathematics.
The main result is a duality theorem
between a class of enriched categories
(called $\overline{\mathbb{Z}}$- or 
$\overline{\mathbb{R}}$-categories)
and that of what we call 
($\overline{\mathbb{Z}}$- or 
$\overline{\mathbb{R}}$-) extended L-convex sets.
We introduce extended L-convex sets as
variants of certain discrete structures
called L-convex sets and L-convex polyhedra, 
studied in the field of discrete convex analysis.
We also introduce homomorphisms between
extended L-convex sets.
The theorem claims that there is a one to one 
correspondence 
(up to isomorphism) between two classes.
The thesis also contains an introductory
chapter on enriched categories and
no categorical knowledge is assumed.

\tableofcontents
\chapter*{Notation}
\section*{Proof Trees}
Throughout the thesis, we use a
proof-tree-like notation as a visually intuitive way
of reasoning. 
The tree
\begin{prooftree}
	\alwaysDoubleLine
	\Axiom$P(x,y,\dots)\fCenter$
	\UnaryInf$Q(x,y,\dots)\fCenter$
\end{prooftree}
where $P$ and $Q$ are formulas
possibly with some free variables
($x,y,\dots$), represents the assertion
\[
	P(x,y,\dots) \iff Q(x,y,\dots)	
\]
for all possible values of $x,y,\dots$ in some
appropriate domains. 

For example, the supremum of a subset $S$ of a
complete lattice $(L,\sqsubseteq)$ is the unique
element $a\in L$ with the following property:
\begin{prooftree}
	\def\fCenter{\ \sqsubseteq\ }
	\alwaysDoubleLine
	\Axiom$s\fCenter x\quad(\forall s\in S)$
	\UnaryInf$a\fCenter x$
\end{prooftree}
which is to be read as 
\[
	x\text{ is an upper bound of }S\iff a\sqsubseteq x	
\]
for all $x\in L$, or equivalently ``$a$ is the least element
among upper bounds of $S$.''

Higher trees also appear:
\begin{prooftree}
	\alwaysDoubleLine
	\Axiom$P(x,y,\dots)\fCenter$
	\UnaryInf$Q(x,y,\dots)\fCenter$
	\UnaryInf$R(x,y,\dots)\fCenter$
\end{prooftree}
Note that any two formulas, especially the top
and the bottom ones, are equivalent in a 
higher tree as well.

\section*{$\lambda$-Notation}
In order to define maps succinctly without giving 
a name,
we exploit the $\lambda$-notation 
borrowed from (typed) lambda calculi.
An example of definitions of maps via the
$\lambda$-notation is the following:
\[
\lambda x\in \mathbb{R}.\,x^2+2x.
\]
This expression denotes a function which takes 
an element $x$ of $\mathbb{R}$ and returns the 
value $x^2+2x$.
Therefore if we pass an argument, say, 3,
we have:
\[
(\lambda x\in \mathbb{R}.\,x^2+2x)(3)=3^2+2\cdot3=15.
\]

\section*{Conventions}
To highlight definitions,
terms being defined (either explicitly or implicitly)
are written in the
\textbf{boldface} font;
the \textit{italic} font is not used for this purpose, 
and is reserved for emphases.

The ends of definitions and examples are marked by
$\blacksquare$, and the ends of proofs
by $\qedsymbol$.

\mainmatter

\chapter{Introduction}
\section{Background and Our Results}
The study presented in this thesis is first motivated and
inspired by a recent paper by Simon Willerton
\cite{Wil13}.
In this paper, he recognizes that a construction
called \textit{directed tight span}, introduced 
independently by \cite{HK12} and \cite{KKO12},
is an instance of a more general construction 
known as the \textit{Isbell completion}.
In fact, the directed tight span is a 
result within the theory of metric spaces,
whereas the Isbell completion belongs to a
discipline called enriched category theory. 
Why such a link could exist?
This is due to an influential 1973 paper \cite{Law73} of
William Lawvere; he remarks that, among others,
\textit{enriched categories specialize to metric spaces}.

Let us be more precise.
Enriched categories are abstract entities that 
generalize (ordinary) categories studied in
category theory.
A general definition of an enriched category involves
a parameter $\mathcal{V}$ called the 
\textit{enriching category},
itself a category (with some additional structures).
Therefore, if we restrict our attention to
those enriched categories with a specific 
enriching category, we obtain a theory to which
the theory of enriched categories specializes.
For example, if we let $\mathcal{V}=\textbf{Set}$, 
where $\textbf{Set}$ is the category of sets and
maps, then it turns out that 
enriched categories with the enriching category
$\textbf{Set}$ 
(called $\textbf{Set}$-categories)
are nothing but categories;
thus enriched category theory specializes to
category theory.
What Lawvere observed is, if we set 
$\mathcal{V}=\overline{\mathbb{R}}_+$,
a poset of nonnegative real numbers together with
an additional element $\infty$ 
(for those not familiar with categories:
posets are a special kind of categories),
then $\overline{\mathbb{R}}_+$-categories are a 
little generalized metric spaces
(in a precise statement, every metric space is an
$\overline{\mathbb{R}}_+$-category).
Therefore he established a link between
the theory of enriched categories and that of
metric spaces.
Now it is amazing that many notions 
of metric spaces,
e.g., that of nonexpansive maps, sup-distances on
function spaces, the Fr\'{e}chet embeddings of
metric spaces, and directed tight spans,
already appear quite canonically in  
enriched category theory.

The main objective of this thesis 
is to present an enriched-categorical 
approach to another seemingly unrelated field: 
\textit{discrete convex analysis}.
The theory of discrete convex analysis 
(see \cite{Mur03} for details)
can be seen as a discrete (or $\mathbb{Z}$) version of 
convex analysis (based on $\mathbb{R}$),
transporting various notions of convex analysis
to a discrete setting.
The notion of convex sets ($\subseteq\mathbb{R}^n$), 
for example, thus has its
discrete counterparts, notably that of
\textit{L-convex sets} (named after lattices) and 
\textit{M-convex sets} (after matroids),
which are certain subsets of $\mathbb{Z}^n$.
Our main result concerns L-convex sets; 
rather, their variant what we call 
$\overline{\mathbb{Z}}$-\textit{extended L-convex sets}.
By $\overline{\mathbb{Z}}$ we mean (among several other 
things) a poset of integers together with two additional 
elements $-\infty$ and $\infty$.
We introduce $\overline{\mathbb{Z}}$-extended L-convex 
sets roughly as certain 
subsets of $\overline{\mathbb{Z}}^n$,
or more generally, subsets of $\overline{\mathbb{Z}}^V$ 
where $V$ is some (possibly infinite) set.
The axioms we impose for $\overline{\mathbb{Z}}$-extended 
L-convex sets are similar (but definitely not equal) to 
that for L-convex sets, so we regard 
$\overline{\mathbb{Z}}$-extended L-convex sets as 
analogs of L-convex sets.
On the other hand, the poset $\overline{\mathbb{Z}}$
has a natural structure to be an enriching category,
and we can consider $\overline{\mathbb{Z}}$-categories
and the theory of them.
Recall the observation of Lawvere that 
$\overline{\mathbb{R}}_+$-categories are 
like metric spaces.
Because the structures of $\overline{\mathbb{R}}_+$
and $\overline{\mathbb{Z}}$ are similar, 
one can say that $\overline{\mathbb{Z}}$-categories 
are like metric spaces as well, 
although the analogy is even weaker.

Our main result is a correspondence (or duality) between 
$\overline{\mathbb{Z}}$-categories and 
$\overline{\mathbb{Z}}$-extended L-convex sets.
The correspondence is established at three levels:
first for individual objects,
second for maps between objects,
and finally for what we call canonical orderings
(usually called \textit{natural transformations})
between maps.
At the first level, we present a construction
that makes a $\overline{\mathbb{Z}}$-extended 
L-convex set out of a $\overline{\mathbb{Z}}$-category,
and another one that performs the reverse. 
These two constructions are \textit{inverses} in the 
sense that if we start from a 
$\overline{\mathbb{Z}}$-category and apply the first
construction to get a $\overline{\mathbb{Z}}$-extended 
L-convex set, and then apply the second to
obtain another $\overline{\mathbb{Z}}$-category,
then the resulting one is isomorphic to the 
$\overline{\mathbb{Z}}$-category we started from;
and likewise if we start from
a $\overline{\mathbb{Z}}$-extended 
L-convex set.
Such a result is in fact already established 
for L-convex sets and distance
functions satisfying the triangle inequality,
with essentially the same technical contents;
see \cite[Section 5.3]{Mur03}.
However, our result is new in its
formulation of the constructions.
It turns out that the set 
$\mathbb{Z}\cup\{-\infty,\infty\}$, which is the 
underlying set of the poset $\overline{\mathbb{Z}}$,
can naturally be seen
both as a $\overline{\mathbb{Z}}$-category and
as a $\overline{\mathbb{Z}}$-extended L-convex set;
in either case we denote the resulting
entity by $\overline{\mathbb{Z}}$.
Our constructions explicitly involve the 
\textit{function space constructions 
with codomain} $\overline{\mathbb{Z}}$ (either as 
a $\overline{\mathbb{Z}}$-category or 
as a $\overline{\mathbb{Z}}$-extended L-convex set)
as key steps.
Therefore in our formulation,
maps play a crucial role.
Let us discuss them next;
the discussion also leads to the second level of the 
correspondence.

In enriched category theory, there is an established
notion of  
maps between enriched (say, $\mathcal{V}$-)
categories, called 
$\mathcal{V}$-\textit{functors}.
Thus we adopt $\overline{\mathbb{Z}}$-functors as 
the members of the class of maps between 
$\overline{\mathbb{Z}}$-categories we consider, which 
are basically nonexpansive 
(distance-nonincreasing) maps.
For $\overline{\mathbb{Z}}$-extended L-convex sets, 
we adopt what we call \textit{homomorphisms} between
them as natural maps.
The introduction of the notion of 
homomorphisms is among the contribution of the thesis.
Its definition is closely tied to our definition 
of $\overline{\mathbb{Z}}$-extended L-convex sets
and seems to be natural.
The duality at the second level states that 
there is a bijection between 
the set of $\overline{\mathbb{Z}}$-functors
with specified domain and codomain
$\overline{\mathbb{Z}}$-categories and that of 
homomorphisms with the corresponding
domain and codomain $\overline{\mathbb{Z}}$-extended
L-convex sets, where the correspondence is 
that built in the first level.
However, one noteworthy point is 
that the \textit{directions} of maps reverse;
that is, what corresponds to the domain 
$\overline{\mathbb{Z}}$-extended
L-convex set is the codomain 
$\overline{\mathbb{Z}}$-category,
and the codomain 
$\overline{\mathbb{Z}}$-extended
L-convex set the domain 
$\overline{\mathbb{Z}}$-category.
This is the reason why we call the whole correspondence
a \textit{duality} as well.

The canonical orderings are certain preorder relations
on the sets of $\overline{\mathbb{Z}}$-functors or
homomorphisms with specified domain and codomain.
They are specialization of $\mathcal{V}$-natural
transformations between $\mathcal{V}$-functors in 
enriched category theory (one can interpret 
homomorphisms as a special kind of 
$\overline{\mathbb{Z}}$-functors).
We show that, as a third level of the duality,
the correspondence of maps at the second level
respects the canonical orderings.

In fact, there is no difficulty to develop 
an entirely parallel story by replacing
$\mathbb{Z}$ by $\mathbb{R}$; in this case,
what correspond to our $\overline{\mathbb{R}}$-extended 
L-convex sets turn out to be 
\textit{L-convex polyhedra}, which are also studied in 
discrete convex analysis.
Therefore in the later chapters of the thesis,
we use the symbol $\mathbb{K}$ to denote 
either $\mathbb{Z}$ or $\mathbb{R}$
and discuss $\overline{\mathbb{K}}$-categories and 
$\overline{\mathbb{K}}$-extended L-convex sets,
treating both cases simultaneously.

\section{Chapter Overview}
In Chapter 2, we develop the theory of enriched 
categories, in a simplified form.
Although the general theory of enriched categories
normally requires acquaintance with basic
(ordinary) category theory
(as presented in \cite{Mac98}), we avoid those points
where such knowledge is compulsory by restricting
our interest to the cases where enriching categories
are posets.
Therefore, the exposition is intended to be so
introductory as to be
readable without any previous categorical experience.
One remarkable point of (enriched) category theory is 
that many general abstract notions specialize to 
ones which are well-known inside some particular
branch of mathematics.
We hope that, with preordered sets and 
(somewhat generalized variants of)
metric spaces as running examples, the chapter
serves to convey the reader this fascinating
aspect of categories.  

Chapter 3 contains our main contribution.
First we introduce the notions of 
$\overline{\mathbb{K}}$-extended L-convex sets 
and homomorphisms between them.
To indicate the underlying categorical viewpoints
(in this case, not in the level of (enriched)
$\overline{\mathbb{K}}$\textit{-categories},
but in that of the (ordinary, or 2-) \textit{category of 
$\overline{\mathbb{K}}$-categories} or the 
\textit{category of $\overline{\mathbb{K}}$-extended
L-convex sets}),
we occasionally use the terminology of 
category theory without giving definitions. 
However, all statements are translated into 
elementary terms and therefore the reader can
entirely skip these parts.
We present the duality theorem,
first for individual objects,
next for maps between them,
and finally for canonical orderings between maps.

Finally, in Chapter 4 we summarize the results of 
the thesis and briefly indicate ways to 
future research.

\chapter{Poset-Enriched Category Theory}
In this chapter, we develop the theory of
enriched categories, under a crucial assumption
that the enriching category is actually a
\textit{poset}.
This assumption entirely removes the burden
of checking the commutativity of (usually large) diagrams
in the enriching category, simply because every
diagram in a poset commutes.
Since all enriching categories we will encounter
in this thesis are posets,
such a simplified theory suffices.
The organization of the chapter roughly follows that of 
\cite{Kel82}, which we also recommend as a text
on general enriched category theory.

\section{Enriching Posets}
As one needs the notion of fields to define
vector spaces, in order to define enriched categories,
we need the notion of \textit{enriching categories},
often denoted by the symbol $\mathcal{V}$.
The fertility of the resulting theory of enriched
categories over a particular enriching category
$\mathcal{V}$ depends on how \textit{nice} $\mathcal{V}$ is, and it turns out that
in order to gain a fully
fruitful theory, $\mathcal{V}$ should be a
\textit{bicomplete symmetric monoidal closed category}.
As promised above, we treat only the cases where
$\mathcal{V}$ is a poset; when this is the case, the
definition of bicomplete symmetric monoidal closed
categories, rather, 
\textit{bicomplete symmetric monoidal closed posets} = 
\textit{symmetric monoidal closed complete lattices},
is given successively as follows:

\begin{defn} 
A \textbf{symmetric monoidal poset}
(\textbf{SM-P})
$\mathcal{V}$ is a triple
$(\mathcal{V}_0,\otimes,e)$
where
\begin{itemize}
	\item
	$\mathcal{V}_0=(\mathcal{V}_0,\sqsubseteq)$ 
	is a poset called the \textbf{underlying poset};
	\item
	$\otimes\colon\mathcal{V}_0\times\mathcal{V}_0\longrightarrow \mathcal{V}_0$ is a binary operation on
	$\mathcal{V}_0$
	called the \textbf{tensor product};
	\item
	$e$ is an element of $\mathcal{V}_0$ called the 
	\textbf{unit element};
\end{itemize}
such that the following axioms hold:
\begin{description}[font=\normalfont]
	\item[$\qquad$(Monotonicity of $\otimes$)]
	$x\sqsubseteq x^\prime$ and $y\sqsubseteq y^\prime$ imply $x\otimes y\sqsubseteq x^\prime\otimes y^\prime\quad (\forall x,x^\prime,y,y^\prime\in \mathcal{V}_0)$;
	\item[$\qquad$(Associative law)]
	$x\otimes (y\otimes z)=(x\otimes y)\otimes z\quad(\forall x,y,z\in \mathcal{V}_0)$;
	\item[$\qquad$(Unit law for $\otimes$)]
	$e\otimes x=x=x\otimes e\quad(\forall x\in \mathcal{V}_0)$;
	\item[$\qquad$(Commutative law for $\otimes$)]
	$x\otimes y=y\otimes x\quad (\forall x,y\in \mathcal{V}_0)$.
\end{description}
A \textbf{symmetric monoidal complete lattice} 
(\textbf{SM-CL}) is
an SM-P whose underlying poset $\mathcal{V}_0$
is a complete lattice.
\end{defn}

Thus, an SM-P is nothing but a
partially ordered commutative monoid.
The name of the operation $\otimes$, 
tensor product, comes from the fact that
in some (non-poset) symmetric monoidal categories,
the corresponding operation is given by the classical
tensor product of e.g., modules.

\begin{defn} 
A \textbf{symmetric closed poset}
(\textbf{SC-P})
$\mathcal{V}$ is a triple
$(\mathcal{V}_0,[-,-],e)$
where
\begin{itemize}
	\item
	$\mathcal{V}_0=(\mathcal{V}_0,\sqsubseteq)$ 
	is a poset called the \textbf{underlying poset};
	\item
	$[-,-]\colon\mathcal{V}_0^\op\times\mathcal{V}_0\longrightarrow \mathcal{V}_0$ is a binary operation on
	$\mathcal{V}_0$
	called the \textbf{internal-hom};
	\item
	$e$ is an element of $\mathcal{V}_0$ called the 
	\textbf{unit element};
\end{itemize}
such that the following axioms hold:
\begin{description}[font=\normalfont]
	\item[$\qquad$(Monotonicity of \mbox{$[-,-]$})]
	$y\sqsupseteq y^\prime$ and $z\sqsubseteq z^\prime$ imply $[y,z]\sqsubseteq [y^\prime, z^\prime]\quad (\forall y,y^\prime,z,z^\prime\in \mathcal{V}_0)$;
	\item[$\qquad$(Composition law)]
	$[y,z]\sqsubseteq [[x,y],[x,z]]\quad(\forall x,y,z\in\mathcal{V}_0)$;
	\item[$\qquad$(Unit law for \mbox{$[-,-]$})]
	$z=[e,z]\quad(\forall z\in \mathcal{V}_0)$;
	\item[$\qquad$(Commutative law for \mbox{$[-,-]$})]
	$x\sqsubseteq [y,z]\iff y\sqsubseteq[x,z]\quad(\forall x,y,z\in \mathcal{V}_0)$.
\end{description}
A \textbf{symmetric closed complete lattice} 
(\textbf{SC-CL}) is
an SC-P whose underlying poset $\mathcal{V}_0$
is a complete lattice.
\end{defn}

Note the sign of an inequality in the monotonicity
axiom: the domain of the first
argument of $[-,-]$ is
$\mathcal{V}_0^\op=(\mathcal{V}_0,\sqsupseteq)$,
the poset obtained by reversing the order
$\sqsubseteq$ in $\mathcal{V}_0$, thus a clause
in the antecedent of the monotonicity axiom
is $y\sqsupseteq y^\prime$,
not $y\sqsubseteq y^\prime$.
We will explain the mysterious name ``internal-hom'' 
later.

\begin{defn} 
A \textbf{symmetric monoidal closed poset}
(\textbf{SMC-P})
$\mathcal{V}$ is a quadruple
$(\mathcal{V}_0,\otimes,$ $[-,-],e)$
where
\begin{itemize}
	\item
	the triple $(\mathcal{V}_0,\otimes,e)$ is an SM-P;
	\item
	the triple $(\mathcal{V}_0,[-,-],e)$ is an SC-P;
\end{itemize}
such that these two structures are related as follows:
\begin{description}[font=\normalfont]
	\item[$\qquad$(Adjointness relation)]
	$x\otimes y\sqsubseteq z\iff x\sqsubseteq [y,z]\quad(\forall x,y,z\in \mathcal{V}_0)$.
\end{description}
A \textbf{symmetric monoidal closed complete lattice}
(\textbf{SMC-CL}) is
an SMC-P whose underlying poset $\mathcal{V}_0$
is a complete lattice.
\end{defn}

The adjointness relation, in the proof-tree notation
\begin{prooftree}
	\def\fCenter{\ \sqsubseteq\ }
	\alwaysDoubleLine
	\Axiom$x\otimes y\fCenter z$
	\UnaryInf$x\fCenter [y,z]$
\end{prooftree}
is equivalent to the requirement that 
the unary operations $(-\otimes y)$ and $[y,-]$,
obtained by substituting $y$, form a 
Galois connection with $(-\otimes y)$ the left (lower)
adjoint and $[y,-]$ the right (upper) adjoint.
One can view the commutative law for $[-,-]$ 
in a similar way.

Although we presented an SMC-P
$\mathcal{V}$ as a quadruple
$(\mathcal{V}_0,\otimes,$ $[-,-],e)$,
in fact the adjointness relation is so rigid that 
under which the tensor product $\otimes$ and the
internal-hom $[-,-]$ determine each other uniquely. 
In particular, when the underlying poset is a complete
lattice, the following holds:

\begin{prop}\label{prop:SMCCL} 
Let $\mathcal{V}_0=(\mathcal{V}_0,\sqsubseteq)$ be
a complete lattice. Then the following are
equivalent:
\begin{enumerate}
	\item
	$(\mathcal{V}_0,\otimes,[-,-],e)$ is an SMC-CL.
	\item
	$(\mathcal{V}_0,\otimes,e)$ is an SM-CL and 
	for each $y\in \mathcal{V}_0$, $(-\otimes y)$
	preserves suprema.
	\item
	$(\mathcal{V}_0,[-,-],e)$ is an SC-CL and
	for each $y\in \mathcal{V}_0$, $[y,-]$
	preserves infima.
\end{enumerate}
Moreover, the data in (ii) or (iii) are
sufficient to recover the whole data in (i).
\end{prop}
\begin{proof}
\begin{description}[font=\normalfont]
\item[[(i)$\implies$(ii)\!\!\!]]
The adjointness relation implies the following:
\begin{prooftree}
	\def\fCenter{\ \sqsubseteq\ }
	\alwaysDoubleLine
	\Axiom$(\bigvee x_i)\otimes y\fCenter z$
	\UnaryInf$\bigvee x_i\fCenter [y,z]$
	\UnaryInf$x_i\fCenter [y,z]\quad (\forall i)$
	\UnaryInf$x_i\otimes y\fCenter
	z\phantom{[y,]}\quad (\forall i)$
\end{prooftree}
Thus $(\bigvee x_i)\otimes y=\bigvee (x_i\otimes y)$,
as required.
\item[[(ii)$\implies$(i)\!\!\!]]
First we claim that in order to make the adjointness relation
\begin{prooftree}
	\def\fCenter{\ \sqsubseteq\ }
	\alwaysDoubleLine
	\Axiom$x\otimes y\fCenter z$
	\UnaryInf$x\fCenter [y,z]$
\end{prooftree}
hold for some binary operation $[-,-]$, 
the value of $[y,z]$ must be the supremum of $x$'s
satisfying $x\otimes y\sqsubseteq z$, namely,
\begin{align}\label{homAsSup}
	[y,z]=\bigvee\{x\in \mathcal{V}_0\mid x\otimes y\sqsubseteq z\}.
\end{align}
Note that the supremum certainly exists since $\mathcal{V}_0$
is assumed to be a complete lattice.
Suppose $[-,-]$ is a binary operation on $\mathcal{V}_0$
satisfying the adjointness relation:
\begin{description} [font=\normalfont]
\item[[\mbox{$\bigvee\{x\in \mathcal{V}_0\mid x\otimes y\sqsubseteq z\}\sqsubseteq[y,z]$}\!\!\!]]
For every $x\in \{x\in \mathcal{V}_0\mid x\otimes y\sqsubseteq z\}$,
$x\sqsubseteq [y,z]$
holds by the adjointness relation, so 
\[
	\bigvee\{x\in \mathcal{V}_0\mid x\otimes y\sqsubseteq z\}\sqsubseteq[y,z].
\]
\item[[\mbox{$[y,z]\sqsubseteq\bigvee\{x\in \mathcal{V}_0\mid x\otimes y\sqsubseteq z\}$}\!\!\!]]
$[y,z]\sqsubseteq [y,z]$
implies by the adjointness relation that $[y,z]\otimes y$ $\sqsubseteq z$
and hence
$[y,z]\in \{x\in \mathcal{V}_0\mid x\otimes y\sqsubseteq z\},$
thus 
\[
	[y,z]\sqsubseteq \bigvee\{x\in \mathcal{V}_0\mid x\otimes y\sqsubseteq z\}.
\]
\end{description}
Therefore (\ref{homAsSup}) is the only possible
definition for the operation $[-,-]$.
Conversely, (\ref{homAsSup}) implies the adjointness 
relation:
\begin{description} [font=\normalfont]
\item[[\mbox{$x\otimes y\sqsubseteq z\implies x\sqsubseteq [y,z]$}\!\!\!]]
Since $x\otimes y\sqsubseteq z$ implies 
$x \in\{x\in \mathcal{V}_0\mid x\otimes y\sqsubseteq z\}$,
\[
x\sqsubseteq \bigvee\{x\in \mathcal{V}_0\mid x\otimes y\sqsubseteq z\}=[y,z].
\]
\item[[\mbox{$x\sqsubseteq [y,z] \implies x\otimes y\sqsubseteq z$}\!\!\!]]
We first show that $[y,z]\otimes y\sqsubseteq z$:
\begin{align*}
[y,z]\otimes y 
&= \bigvee\{x\in \mathcal{V}_0\mid x\otimes y\sqsubseteq z\}\otimes y \\
&= \bigvee\{x\otimes y\mid x\otimes y\sqsubseteq z\}\\
&\sqsubseteq z.
\end{align*}
We used the assumption that $(-\otimes y)$ preserves 
suprema.
Using this fact and the monotonicity of $\otimes$, 
we conclude 
\[
	x\otimes y\sqsubseteq [y,z]\otimes y\sqsubseteq z.
\]
\end{description}

We now check that the operation $[-,-]$ defined by (\ref{homAsSup})
satisfies the axioms required for SC-Ps.
\begin{description}[font=\normalfont]
\item[[Monotonicity of \mbox{$[-,-]$}\!\!\!]]
Suppose $y\sqsupseteq y^\prime$ and 
$z\sqsubseteq z^\prime$ hold.
Then, 
\[
	\{x\in \mathcal{V}_0\mid x\otimes y\sqsubseteq z\}
	\subseteq
	\{x\in \mathcal{V}_0\mid x\otimes y^\prime\sqsubseteq z^\prime\}
\]
because $x\otimes y\sqsubseteq z$ implies
$	x\otimes y^\prime\sqsubseteq x\otimes y
	\sqsubseteq z\sqsubseteq z^\prime.$
Therefore
\[
	[y,z]=
	\bigvee\{x\in \mathcal{V}_0\mid x\otimes y\sqsubseteq z\}
	\sqsubseteq
	\bigvee\{x\in \mathcal{V}_0\mid x\otimes y^\prime\sqsubseteq z^\prime\}
	=[y^\prime,z^\prime],
\]
as required.
\item[[Composition law\!\!\!]]
The claim is 
$[y,z]\sqsubseteq [[x,y],[x,z]]$,
namely,
\[
	\bigvee\{w\in \mathcal{V}_0\mid w\otimes y\sqsubseteq z\}
	\sqsubseteq
	\bigvee\{w\in \mathcal{V}_0\mid w\otimes [x,y]\sqsubseteq [x,z]\}.
\]
Thus it suffices to show 
\begin{align*}
	\{w\in \mathcal{V}_0\mid w\otimes y\sqsubseteq z\}
	&\subseteq
	\{w\in \mathcal{V}_0\mid w\otimes [x,y]\sqsubseteq [x,z]\}\\
	&= \big\{w\mid w\otimes \bigvee\{v\mid v\otimes x\sqsubseteq y\}\sqsubseteq\bigvee\{v\mid v\otimes x\sqsubseteq z\}\big\}\\
	&= \big\{w\mid \bigvee\{w\otimes v\mid v\otimes x\sqsubseteq y\}\sqsubseteq\bigvee\{v\mid v\otimes x\sqsubseteq z\}\big\}
\end{align*}
(here we used the assumption that $(w\otimes -)=(-\otimes w)$ preserves suprema).
So let us suppose $w\otimes y\sqsubseteq z$ and aim
to show that
\[
	\{w\otimes v\mid v\otimes x\sqsubseteq y\}\subseteq\{v\mid v\otimes x\sqsubseteq z\}.
\]
This holds because for $v$ with $v\otimes x\sqsubseteq y$,
\[
	(w\otimes v)\otimes x=w\otimes(v\otimes x)
	\sqsubseteq w\otimes y\sqsubseteq z
\]
holds.
\item[[Unit law for \mbox{$[-,-]$}\!\!\!]]
The claim is $z=[e,z]$. This reduces to 
\begin{align*}
	z&=\bigvee\{x\in \mathcal{V}_0\mid x\otimes e\sqsubseteq z\}\\
	&=\bigvee\{x\in \mathcal{V}_0\mid x\sqsubseteq z\},
\end{align*}
a trivial equality.
\item[[Commutative law for \mbox{$[-,-]$}\!\!\!]]
The claim is $x\sqsubseteq[y,z]\iff y\sqsubseteq[x,z]$, 
which follows from the following proof tree:
\begin{prooftree}
	\def\fCenter{\ \sqsubseteq\ }
	\alwaysDoubleLine
	\Axiom$x\fCenter [y,z]$
	\UnaryInf$x\otimes y\fCenter z$
	\UnaryInf$y\otimes x\fCenter z$
	\UnaryInf$y\fCenter [x,z]$
\end{prooftree}
\end{description}
\item[[(i)$\implies$(iii)\!\!\!]]
The adjointness relation implies the following:
\begin{prooftree}
	\def\fCenter{\ \sqsubseteq\ }
	\alwaysDoubleLine
	\Axiom$x\fCenter [y,\bigwedge z_i]$
	\UnaryInf$x\otimes y\fCenter \bigwedge z_i$
	\UnaryInf$x\otimes y\fCenter z_i
	\phantom{[y,]}\quad (\forall i)$
	\UnaryInf$x\fCenter [y,z_i]\quad (\forall i)$
\end{prooftree}
Thus $[y,\bigwedge z_i]=\bigwedge [y,z_i]$,
as required.
\item[[(iii)$\implies$(i)\!\!\!]]
First we claim that in order to make the 
adjointness relation
\begin{prooftree}
	\def\fCenter{\ \sqsubseteq\ }
	\alwaysDoubleLine
	\Axiom$x\otimes y\fCenter z$
	\UnaryInf$x\fCenter [y,z]$
\end{prooftree}
hold for some binary operation $\otimes$, 
the value of $x\otimes y$ must be the infimum of $z$'s
satisfying $x\sqsubseteq [y,z]$, namely,
\begin{align}\label{tenAsInf}
	x\otimes y=\bigwedge\{z\in \mathcal{V}_0\mid x\sqsubseteq [y,z]\}.
\end{align}
Note that the infimum certainly exists since 
$\mathcal{V}_0$
is assumed to be a complete lattice.
Suppose $\otimes$ is a binary operation on $\mathcal{V}_0$
satisfying the adjointness relation:
\begin{description} [font=\normalfont]
\item[[\mbox{$x\otimes y\sqsubseteq \bigwedge\{z\in \mathcal{V}_0\mid x\sqsubseteq [y,z]\}$}\!\!\!]]
For every 
$z\in \{z\in \mathcal{V}_0\mid x\sqsubseteq [y,z]\}$,
$x\otimes y\sqsubseteq z$
holds by the adjointness relation, so 
\[
	x\otimes y\sqsubseteq \bigwedge\{z\in \mathcal{V}_0\mid x\sqsubseteq [y,z]\}.
\]
\item[[\mbox{$\bigwedge\{z\in \mathcal{V}_0\mid x\sqsubseteq [y,z]\}\sqsubseteq x\otimes y$}\!\!\!]]
$x\otimes y\sqsubseteq x\otimes y$
implies by the adjointness relation that
$x\sqsubseteq$ $[y,x\otimes y]$
and hence
$x\otimes y\in \{z\in \mathcal{V}_0\mid x\sqsubseteq [y,z]\},$
thus 
\[
	\bigwedge\{z\in \mathcal{V}_0\mid x\sqsubseteq [y,z]\}\sqsubseteq x\otimes y.
\]
\end{description}
Therefore (\ref{tenAsInf}) is the only possible 
definition for the operation $\otimes$.
Conversely, (\ref{tenAsInf}) implies the adjointness 
relation:
\begin{description} [font=\normalfont]
\item[[\mbox{$x\sqsubseteq [y,z] \implies x\otimes y\sqsubseteq z$}\!\!\!]]
Since $x\sqsubseteq [y,z]$ implies 
$z \in\{z\in \mathcal{V}_0\mid x\sqsubseteq [y,z]\}$,
\[
x\otimes y=\bigwedge\{z\in \mathcal{V}_0\mid x\sqsubseteq [y,z]\} \sqsubseteq z.
\]
\item[[\mbox{$x\otimes y\sqsubseteq z\implies x\sqsubseteq [y,z]$}\!\!\!]]
We first show that $x\sqsubseteq [y,x\otimes y]$:
\begin{align*}
[y,x\otimes y]
&= [y,\bigwedge\{z\in \mathcal{V}_0\mid x\sqsubseteq [y,z]\}] \\
&= \bigwedge\{[y,z]\mid x\sqsubseteq [y,z]\}\\
&\sqsupseteq x.
\end{align*}
We used the assumption that $[y,-]$ preserves 
infima.
Using this fact and the monotonicity of $[-,-]$, we conclude 
\[
	x\sqsubseteq [y,x\otimes y]\sqsubseteq [y,z].
\]
\end{description}
We now check that the operation $\otimes$ defined by
(\ref{tenAsInf})
satisfies the axioms required for SM-Ps.
\begin{description}[font=\normalfont]
\item[[Monotonicity of $\otimes$\!\!\!]]
Suppose $x\sqsubseteq x^\prime$ and 
$y\sqsubseteq y^\prime$ hold.
Then, 
\[
	\{z\in \mathcal{V}_0\mid x^\prime\sqsubseteq [y^\prime,z]\}\subseteq
	\{z\in \mathcal{V}_0\mid x\sqsubseteq [y,z]\}
\]
because $x^\prime\sqsubseteq [y^\prime,z]$ implies
$	x\sqsubseteq x^\prime\sqsubseteq [y^\prime,z]\sqsubseteq [y,z].$
Therefore
\[
	x\otimes y=
	\bigwedge\{z\in \mathcal{V}_0\mid x\sqsubseteq [y,z]\}
	\sqsubseteq
	\bigwedge\{z\in \mathcal{V}_0\mid x^\prime\sqsubseteq [y^\prime,z]\}
	=x^\prime\otimes y^\prime,
\]
as required.
\item[[Associative law\!\!\!]]
The claim is $x\otimes(y\otimes z)=(x\otimes y)\otimes z$,
namely,
\[
	\bigwedge\{w\in \mathcal{V}_0\mid x\sqsubseteq [y\otimes z,w]\}
	=\bigwedge\{w\in \mathcal{V}_0\mid x\otimes y\sqsubseteq [z,w]\}.
\]
We prove this by showing
$\{w\in \mathcal{V}_0\mid x\sqsubseteq [y\otimes z,w]\}=\{w\in \mathcal{V}_0\mid x\otimes y\sqsubseteq [z,w]\}$,
or more explicitly,
\begin{align}\label{assoc1}
	\big\{w\mid x\sqsubseteq [\bigwedge\{v\mid y\sqsubseteq [z,v]\},w]\big\}=\big\{w\mid \bigwedge\{v\mid x\sqsubseteq [y,v]\}\sqsubseteq [z,w]\big\}.
\end{align}
Note that by the commutative law for $[-,-]$, (\ref{assoc1})
is equivalent to
\begin{align}\label{assoc2}
	\big\{w\mid \bigwedge\{v\mid y\sqsubseteq [z,v]\}\sqsubseteq [x,w]\big\}=\big\{w\mid z\sqsubseteq [\bigwedge\{v\mid x\sqsubseteq [y,v]\},w]\big\}.
\end{align}
\begin{description}[font=\normalfont]
\item[[(LHS)$\subseteq$(RHS)\!\!\!]]
We show this in the form of (\ref{assoc1}).
It suffices to show 
\begin{align*}
x\sqsubseteq [\bigwedge\{v\mid y\sqsubseteq [z,v]\},w]\implies& [z,w]\in \{v\mid x\sqsubseteq [y,v]\}\\
\big(\iff&  x\sqsubseteq [y,[z,w]]\;\,\big)
\end{align*}
and this follows from
\begin{align*}
[\bigwedge\{v\mid y\sqsubseteq [z,v]\},w]
&\sqsubseteq [[z,\bigwedge\{v\mid y\sqsubseteq [z,v]\}],[z,w]]\\
&= [\bigwedge\{[z,v]\mid y\sqsubseteq [z,v]\},[z,w]]\\
&\sqsubseteq [y,[z,w]],
\end{align*}
where the first inequality is an instance of the
composition law,
the equality follows from the inf-preserving property of $[z,-]$,
and the last inequality from $\bigwedge\{[z,v]\mid y\sqsubseteq [z,v]\}\sqsupseteq y$ and the monotonicity
of $[-,-]$.
\item[[(LHS)$\supseteq$(RHS)\!\!\!]]
In this case, we refer to the form (\ref{assoc2}).
Now, a sufficient condition is
\begin{align*}
z\sqsubseteq [\bigwedge\{v\mid x\sqsubseteq [y,v]\},w]\implies& [x,w]\in \{v\mid y\sqsubseteq [z,v]\}\\
\big(\iff&  y\sqsubseteq [z,[x,w]]\\
\iff&  z\sqsubseteq [y,[x,w]]\;\,\big)
\end{align*}
and this follows from
\begin{align*}
[\bigwedge\{v\mid x\sqsubseteq [y,v]\},w]
&\sqsubseteq [[x,\bigwedge\{v\mid x\sqsubseteq [y,v]\}],[x,w]]\\
&= [\bigwedge\{[x,v]\mid x\sqsubseteq [y,v]\},[x,w]]\\
&= [\bigwedge\{[x,v]\mid y\sqsubseteq [x,v]\},[x,w]]\\
&\sqsubseteq [y,[x,w]],
\end{align*}
where the first inequality is an instance of
the composition law, 
the first equality follows from the inf-preserving property of
$[z,-]$, and the last inequality from
$\bigwedge\{[x,v]\mid y\sqsubseteq [x,v]\}\sqsupseteq y$
and the monotonicity of $[-,-]$.
\end{description}
\item[[Unit law for $\otimes$\!\!\!]]
The claim is $e\otimes x=x=x\otimes e$.
Assuming the commutative law for $\otimes$ proved below, it suffices 
to show $x=x\otimes e$, namely,
\begin{align*}
	x&=\bigwedge\{z\in \mathcal{V}_0\mid x\sqsubseteq [e,z]\}\\
	&=\bigwedge\{z\in \mathcal{V}_0\mid x\sqsubseteq z\},
\end{align*}
which obviously holds.
\item[[Commutative law for $\otimes$\!\!\!]]
The claim is $x\otimes y=y\otimes x$, which follows from
the following proof tree:
\begin{prooftree}
	\def\fCenter{\ \sqsubseteq\ }
	\alwaysDoubleLine
	\Axiom$x\otimes y\fCenter z$
	\UnaryInf$x\fCenter [y,z]$
	\UnaryInf$y\fCenter [x,z]$
	\UnaryInf$y\otimes x\fCenter z$
\end{prooftree}
where $z$ is an arbitrary element of $\mathcal{V}_0$.
\end{description}
\end{description}
\end{proof}

By virtue of this proposition, we will usually
check only the condition (ii) to prove that a 
particular system is an SMC-CL.
However, we will not ignore internal-hom's,
since writing them down explicitly often clarifies
the situation.

We complete the limit-preserving properties of
$\otimes$ and $[-,-]$ of SMC-CLs
by the following proposition:

\begin{prop}\label{prop:lim_ten_hom} 
Let $\mathcal{V}=(\mathcal{V}_0,\otimes,[-,-],e)$
be an SMC-CL.
Then the following hold:
\begin{enumerate}
	\item
	For each $x\in \mathcal{V}_0$, $(x\otimes -)$
	preserves suprema.
	\item
	For each $z\in \mathcal{V}_0$, $[-,z]$
	turns suprema into infima.
\end{enumerate}
\end{prop}
\begin{proof}
\begin{description}[font=\normalfont]
\item[[(i)\!\!\!]]
The commutativity of $\otimes$ 
and sup-preserving property of $(-\otimes x)$
imply the following:
\begin{prooftree}
	\def\fCenter{\ \sqsubseteq\ }
	\alwaysDoubleLine
	\Axiom$x\otimes(\bigvee y_i)\fCenter z$
	\UnaryInf$(\bigvee y_i)\otimes x\fCenter z$
	\UnaryInf$y_i\otimes x\fCenter
	z\quad\;\; (\forall i)$
	\UnaryInf$x\otimes y_i\fCenter
		z\quad\;\; (\forall i)$
\end{prooftree}
Thus $x\otimes(\bigvee y_i)=\bigvee (x\otimes y_i)$,
as required.
\item[[(ii)\!\!\!]]
The commutativity of $[-,-]$ implies the following:
\begin{prooftree}
	\def\fCenter{\ \sqsubseteq\ }
	\alwaysDoubleLine
	\Axiom$x\fCenter [(\bigvee y_i),z]$
	\UnaryInf$\bigvee y_i\fCenter [x,z]$
	\UnaryInf$y_i\fCenter [x,z]\phantom{y_i}\quad (\forall i)$
	\UnaryInf$x\fCenter
	[y_i,z]\phantom{x}\quad (\forall i)$
\end{prooftree}
Thus $[(\bigvee y_i),z]=\bigwedge [y_i,z]$,
as required.
\end{description}
\end{proof}

Now we proceed to see some examples. 
We begin with a toy example.

\begin{ex}[\cite{Law73}]\label{ex:2} 
 The quadruple
$\textbf{2}=(\{\true,\false\},\&,\supset,\true)$
is defined as follows:
\begin{itemize}
\item
$\{\true,\false\}=(\{\true,\false\},\vdash)$ is the 
two-element poset of truth values 
ordered by entailment $\vdash$; see the table below for detail (the symbol $\fullmoon$
denotes that the relation holds, and 
$\times$ that the relation does not hold):
\begin{table}[H]
	\centering
	\begin{tabular}{c c|c c}
	\multicolumn{2}{c|}{\multirow{2}{*}{$x\vdash y$}}& \multicolumn{2}{c}{$y$} \\	    
    &  & $\true$ & $\false$ \\ \hline
    \multirow{2}{*}{$x$} & $\true$ & $\fullmoon$ & $\times$ \\
    &$\false$ & $\fullmoon$ & $\fullmoon$
	\end{tabular}
\end{table}
\item
$\&$ is the binary operation defined as the usual
conjunction, as in the following table:
\begin{table}[H]
	\centering
	\begin{tabular}{c c|c c}
	\multicolumn{2}{c|}{\multirow{2}{*}{$x\myand y$}}& \multicolumn{2}{c}{$y$} \\	    
    &  & $\true$ & $\false$ \\ \hline
    \multirow{2}{*}{$x$} &$\true$ & $\true$ & $\false$ \\
    &$\false$ & $\false$ & $\false$
	\end{tabular}
\end{table}
\item
$\supset$ is the binary operation defined as the usual
implication:
\begin{table}[H]
	\centering
	\begin{tabular}{c c|c c}
	\multicolumn{2}{c|}{\multirow{2}{*}{$x\supset y$}}& \multicolumn{2}{c}{$y$} \\	    
    &  & $\true$ & $\false$ \\ \hline
    \multirow{2}{*}{$x$} &$\true$ & $\true$ & $\false$ \\
    &$\false$ & $\true$ & $\true$
	\end{tabular}
\end{table}
\end{itemize}
One can easily verify that $\textbf{2}$ is 
an SMC-CL.
\end{ex}

The current thesis is largely indebted to 
William Lawvere's 1973
paper \cite{Law73}.
In this paper, he observed that metric spaces
are a special kind of enriched categories,
using the following category (in fact, poset)
$\overline{\mathbb{R}}_+$:

\begin{ex}\label{ex:Rplus} 
The quadruple
$\overline{\mathbb{R}}_+=(\mathbb{R}_+\cup\{\infty\},+,\dotminus,0)$
is defined as follows:
\begin{itemize}
\item
$\mathbb{R}_+\cup\{\infty\}=(\mathbb{R}_+\cup\{\infty\},\geq)$
is the set of nonnegative real numbers 
$\mathbb{R}_+=\{x\in\mathbb{R}\mid$ $x\geq 0\}$
with an additional element $\infty$, 
ordered by the \textit{opposite} $\geq$ of 
the usual ordering $\leq$ (extended from $\mathbb{R}_+$
to $\mathbb{R}_+\cup\{\infty\}$ in an obvious way).
More explicitly, supposing $s,t\in \mathbb{R}_+$ (thus we
already know whether $s\geq t$ holds or not), the order is
given by the following table: 
\begin{table}[H]
	\centering
	\begin{tabular}{c c|c c}
	\multicolumn{2}{c|}{\multirow{2}{*}{$x\geq y$}} &\multicolumn{2}{c}{$y$} \\	    
    &  & $t$ & $\infty$\\ \hline
    \multirow{2}{*}{$x$} &$s$ & $s\geq t$ & $\times$ \\
    &$\infty$ & $\fullmoon$ & $\fullmoon$ 
	\end{tabular}
\end{table}
\item
$+$ is the binary operation defined as the
natural extension of the usual addition:
\begin{table}[H]
	\centering
	\begin{tabular}{c c|c c}
	\multicolumn{2}{c|}{\multirow{2}{*}{$x+y$}}& \multicolumn{2}{c}{$y$} \\	    
    &   & $t$ & $\infty$\\ \hline
    \multirow{2}{*}{$x$} &$s$ & $s+t$ & $\infty$ \\
    &$\infty$  & $\infty$ & $\infty$ 
	\end{tabular}
\end{table}
\item
$\dotminus$ is the binary operation called the truncated
subtraction. To define this, we first set
$t\dotminus s=\max\{t-s,0\}$ for $s,t\in \mathbb{R}_+$.
Then the domain of $\dotminus$ is extended to 
$\mathbb{R}_+\cup\{\infty\}$ as follows
(note the order of the arguments; we set $[x,y]=y\dotminus x$):
\begin{table}[H]
	\centering
	\begin{tabular}{c c|c c}
	\multicolumn{2}{c|}{\multirow{2}{*}{$y\dotminus x$}}& \multicolumn{2}{c}{$y$} \\	    
	&   & $t$ & $\infty$\\ \hline
	\multirow{2}{*}{$x$} &$s$ & $t\dotminus s$ & $\infty$ \\
	&$\infty$  & $0$ & $0$
	\end{tabular}
\end{table}
\end{itemize}
Then it turns out that $\overline{\mathbb{R}}_+$ is an
SMC-CL.

The reader may feel that the definition of  
the extended operations is somewhat arbitrary, 
especially in the case $\infty\dotminus\infty=0$.
We first note that, contrary to such a view,
\textit{in order
$\overline{\mathbb{R}}_+$ 
to be an SMC-CL,
the extensions of $+$ and $\dotminus$ are 
completely determined as soon as we extend the order}
$\geq$.
To see this, recall from Propositions 
\ref{prop:SMCCL} and \ref{prop:lim_ten_hom} that
necessary conditions for the quadruple 
$\overline{\mathbb{R}}_+=(\mathbb{R}_+\cup\{\infty\},+,\dotminus,0)$
to be an SMC-CL are that 
{\allowdisplaybreaks
\begin{align}
	\label{lim_ten1} & x+(-) \text{ preserves suprema for each } 
	x\in \mathbb{R}_+\cup\{\infty\};\\
	\label{lim_ten2} & (-)+y \text{ preserves suprema for each }
	y\in \mathbb{R}_+\cup\{\infty\};\\
	\label{lim_hom1} & (-)\dotminus x \text{ preserves infima for each }
	x\in \mathbb{R}_+\cup\{\infty\};\\
	\label{lim_hom2} & y\dotminus (-) \text{ turns suprema into infima for each }
	y\in \mathbb{R}_+\cup\{\infty\}.
\end{align}}
\!\!Now that
we are working with the opposite ordering $\geq$,
which means our supremum $\bigvee$ corresponds to the 
usual infimum (for which we use a different 
symbol $\inf$) and dually, we obtain the
following for each 
$x,y\in \mathbb{R}_+\cup \{\infty\}$ and
$s\in \mathbb{R}_+$:
{\allowdisplaybreaks
\begin{align*}
	& x+\infty = x+\inf \emptyset 
	\overset{\text{(\ref{lim_ten1})}}=
	\inf \emptyset = \infty; \\
	&\infty+y = \inf \emptyset+y 
	\overset{\text{(\ref{lim_ten2})}}=
	\inf \emptyset = \infty; \\
	&\infty\dotminus s = \sup_{t\in \mathbb{R}_+}\{t\} \dotminus s 
	\overset{\text{(\ref{lim_hom1})}}=
	\sup_{t\in \mathbb{R}_+}\{t\dotminus s\} = \infty;\\
	& y\dotminus \infty = y\dotminus \inf\emptyset
	\overset{\text{(\ref{lim_hom2})}}= 
	\sup \emptyset = 0. 
\end{align*}}
\!\!These specify the extensions.

We then show that $\overline{\mathbb{R}}_+$ is indeed
an SMC-CL,
following the condition (ii) of 
Proposition~\ref{prop:SMCCL}.
\begin{description}[font=\normalfont]
\item[[$\mathbb{R}_+\cup\{\infty\}$ is a complete lattice\!\!\!]]
A well-known fact.
\item[[Monotonicity of $+$\!\!\!]]
Suppose $x\geq x^\prime$ and $y\geq y^\prime$ hold.
If all of $x,x^\prime,y,y^\prime$ are elements of
$\mathbb{R}_+$, then 
$x+y\geq x^\prime +y^\prime$ holds.
Otherwise, either $x$ or $y$ is $\infty$; 
thus $x+y=\infty$ and again
$x+y\geq x^\prime +y^\prime$ holds.
\item[[Associative law\!\!\!]]
The claim is $x+(y+z)=(x+y)+z$. 
If all of $x,y,z$ are elements of
$\mathbb{R}_+$, the claim reduces to an
elementary fact.
Otherwise, one of $x,y,z$ is $\infty$,
and both of $x+(y+z)$ and $(x+y)+z$
are equal to $\infty$.
\item[[Unit law\!\!\!]]
The claim is $0+x=x=x+0$. 
If $x\in \mathbb{R}_+$, it certainly holds; 
if not, $0+\infty=\infty=\infty+0$ and again it holds. 
\item[[Commutative law\!\!\!]]
The claim is $x+y=y+x$. 
If $x,y\in \mathbb{R}_+$, it holds. 
Otherwise, both sides are
equal to $\infty$.
\item[[$(-)+y$ preserves suprema\!\!\!]]
If $y=\infty$, the function $(-)+y$ constantly returns
$\infty$, which is also the least element
(the supremum of the empty set).
Thus it preserves all suprema.

Suppose $y\in \mathbb{R}_+$ and take an 
arbitrary family $x_i\in \mathbb{R}_+\cup \{\infty\}$ 
$(i\in I)$.
We aim to prove 
\begin{align}\label{sup_preserve}
	\inf_{i\in I}\{x_i\}+y=\inf_{i\in I}\{x_i+y\}.
\end{align}

If $\inf_{i\in I}\{x_i\}=\infty$, then
for all $i\in I$, $x_i=\infty$.
Therefore for all $i\in I$, $x_i+y=\infty$
and both sides of (\ref{sup_preserve}) evaluate
to $\infty$.

If $\inf_{i\in I}\{x_i\}\in \mathbb{R}_+$, then 
there exists $i\in I$ with $x_i\in \mathbb{R}_+$, 
and we may simply omit those $x_i$'s with 
$x_i=\infty$, obtaining a new family 
$x_i\in \mathbb{R}_+$ $(i\in I\neq \emptyset)$
of nonnegative real numbers.
Now (\ref{sup_preserve}) essentially reduces to
the continuity of $+$.
\end{description}

Finally, we show that $\dotminus$ is indeed the 
internal-hom, by checking 
the adjointness relation.
\begin{description}[font=\normalfont]
\item[[$x+y\geq z\implies x\geq z\dotminus y\!\!\!$]]
The monotonicity of $\dotminus$ is easily verified.
Therefore, $x+y\geq z$ implies 
$(x+y)\dotminus y\geq z\dotminus y$.
So it suffices to show $x\geq (x+y)\dotminus y$, 
which follows from the table below:
\begin{table}[H]
	\centering
	\begin{tabular}{c c|c c}
	\multicolumn{2}{c|}{\multirow{2}{*}{$(x+y)\dotminus y$}} &\multicolumn{2}{c}{$y$} \\	    
    &  & $t$ & $\infty$\\ \hline
    \multirow{2}{*}{$\;\,\quad x$} &$s$ & $s$ & $0$ \\
    &$\infty$ & $\infty$ & $0$ 
	\end{tabular}
\end{table}
\item[[$x\geq z\dotminus y\implies x+y\geq z\!\!\!$]]
By the monotonicity of $+$, $x\geq z\dotminus y$ implies
$x+y\geq (z\dotminus y)+y$.
So it suffices to show $(z\dotminus y)+y\geq z$,
which follows from the table below:
\begin{table}[H]
	\centering
	\begin{tabular}{c c|c c}
	\multicolumn{2}{c|}{\multirow{2}{*}{$(z\dotminus y)+y$}} &\multicolumn{2}{c}{$z$} \\	    
    &  & $t$ & $\infty$\\ \hline
    \multirow{2}{*}{$\;\,\quad y$} &$s$ & $\max\{s,t\}$ & $\infty$ \\
    &$\infty$ & $\infty$ & $\infty$ 
	\end{tabular}
\end{table}
\end{description}
\end{ex}

Several variants of $\overline{\mathbb{R}}_+$
follow.
First we observe that a similar construction
works when $\mathbb{R}$ is replaced by $\mathbb{Z}$. 

\begin{ex} 
The quadruple
$\overline{\mathbb{Z}}_+=(\mathbb{Z}_+\cup\{\infty\},+,\dotminus,0)$
is defined as follows:
\begin{itemize}
\item
$\mathbb{Z}_+\cup\{\infty\}=(\mathbb{Z}_+\cup\{\infty\},\geq)$
is the set of nonnegative integers 
$\mathbb{Z}_+=\{x\in\mathbb{Z}\mid$ $x\geq 0\}$
with an additional element $\infty$, 
ordered by $\geq$, the order relation
similar to that on 
$\mathbb{R}_+\cup\{\infty\}$. 
For $s,t\in \mathbb{Z}_+$, it is given as follows: 
\begin{table}[H]
	\centering
	\begin{tabular}{c c|c c}
	\multicolumn{2}{c|}{\multirow{2}{*}{$x\geq y$}} &\multicolumn{2}{c}{$y$} \\	    
    &  & $t$ & $\infty$\\ \hline
    \multirow{2}{*}{$x$} &$s$ & $s\geq t$ & $\times$ \\
    &$\infty$ & $\fullmoon$ & $\fullmoon$ 
	\end{tabular}
\end{table}
\item
$+$ is an extension of the usual addition:
\begin{table}[H]
	\centering
	\begin{tabular}{c c|c c}
	\multicolumn{2}{c|}{\multirow{2}{*}{$x+y$}}& \multicolumn{2}{c}{$y$} \\	    
    &   & $t$ & $\infty$\\ \hline
    \multirow{2}{*}{$x$} &$s$ & $s+t$ & $\infty$ \\
    &$\infty$  & $\infty$ & $\infty$ 
	\end{tabular}
\end{table}
\item
$\dotminus$ is an extension of
the operation defined as
$t\dotminus s=\max\{t-s,0\}$:
\begin{table}[H]
	\centering
	\begin{tabular}{c c|c c}
	\multicolumn{2}{c|}{\multirow{2}{*}{$y\dotminus x$}}& \multicolumn{2}{c}{$y$} \\	    
	&   & $t$ & $\infty$\\ \hline
	\multirow{2}{*}{$x$} &$s$ & $t\dotminus s$ & $\infty$ \\
	&$\infty$  & $0$ & $0$
	\end{tabular}
\end{table}
\end{itemize}
$\overline{\mathbb{Z}}_+$ is an SMC-CL;
its proof is obtained by a slight modification
to that for $\overline{\mathbb{R}}_+$.
Again we have no choice for extensions of operations $+$ and 
$\dotminus$.
\end{ex}

The reader may notice that the cases for 
$\mathbb{Z}$ and for $\mathbb{R}$ are almost  
completely parallel;
such a phenomenon occurs frequently in what follows. 
We henceforth let the
symbol $\mathbb{K}$ denote either
$\mathbb{Z}$ or $\mathbb{R}$ and treat both cases 
simultaneously whenever convenient.
This convention lasts throughout the thesis.

Next example introduces the most 
important enriching posets in this thesis.
The definition when $\mathbb{K}=\mathbb{R}$
appears in \cite{Law84}.

\begin{ex}\label{ex:overlineK} 
The quadruple
$\overline{\mathbb{K}}=(\mathbb{K}\cup\{-\infty,\infty\},+,-,0)$
is defined as follows:
\begin{itemize}
\item
$\mathbb{K}\cup\{-\infty,\infty\}=(\mathbb{K}\cup\{-\infty,\infty\},\geq)$
is the set $\mathbb{K}$
with two additional elements $-\infty$ and $\infty$, 
ordered by $\geq$. 
Supposing $s,t\in \mathbb{K}$, the relation is
specified by the following table: 
\begin{table}[H]
	\centering
	\begin{tabular}{c c|c c c}
	\multicolumn{2}{c|}{\multirow{2}{*}{$x\geq y$}}& & $y$ & \\	    
    &  & $-\infty$ & $t$ & $\infty$\\ \hline
    \multirow{3}{*}{$x$} & $-\infty$ & $\fullmoon$ & $\times$ & $\times$ \\
    &$s$ & $\fullmoon$ & $s\geq t$ & $\times$ \\
    &$\infty$ & $\fullmoon$ & $\fullmoon$ & $\fullmoon$ 
	\end{tabular}
\end{table}
\item
$+$ is the binary operation defined as an extension of
the usual addition:
\begin{table}[H]
	\centering
	\begin{tabular}{c c|c c c}
	\multicolumn{2}{c|}{\multirow{2}{*}{$x+y$}}& & $y$ & \\	    
    &  & $-\infty$ & $t$ & $\infty$\\ \hline
    \multirow{3}{*}{$x$} & $-\infty$ & $-\infty$ & $-\infty$ & $\infty$ \\
    &$s$ & $-\infty$ & $s+t$ & $\infty$ \\
    &$\infty$ & $\infty$ & $\infty$ & $\infty$ 
	\end{tabular}
\end{table}
\item
$-$ is the binary operation defined as an extension of
the usual subtraction (the order of the arguments
is similar to that of the truncated subtraction,
namely, $[x,y]=y-x$):
\begin{table}[H]
	\centering
	\begin{tabular}{c c|c c c}
	\multicolumn{2}{c|}{\multirow{2}{*}{$y-x$}}& & $y$ & \\	    
    &  & $-\infty$ & $t$ & $\infty$\\ \hline
    \multirow{3}{*}{$x$} & $-\infty$ & $-\infty$ & $\infty$ & $\infty$ \\
    &$s$ & $-\infty$ & $t-s$ & $\infty$ \\
    &$\infty$ & $-\infty$ & $-\infty$ & $-\infty$ 
	\end{tabular}
\end{table}
\end{itemize}
We claim that $\overline{\mathbb{K}}$ is an SMC-CL.

First we see that the extensions of 
$+$ and $-$ are automatically derived by
a method similar to one we used in
the $\overline{\mathbb{R}}_+$-case.
Let us write down the results of 
Propositions \ref{prop:SMCCL} and 
\ref{prop:lim_ten_hom} concerning 
necessary conditions for the quadruple 
$\overline{\mathbb{K}}=(\mathbb{K}\cup\{-\infty,\infty\},+,-,0)$
to be an SMC-CL:
{\allowdisplaybreaks
\begin{align}
	\label{zlim_ten1} & x+(-) \text{ preserves suprema for each } 
	x\in \mathbb{K}\cup\{-\infty,\infty\};\\
	\label{zlim_ten2} & (-)+y \text{ preserves suprema for each }
	y\in \mathbb{K}\cup\{-\infty,\infty\};\\
	\label{zlim_hom1} & (-)- x \text{ preserves infima for each }
	x\in \mathbb{K}\cup\{-\infty,\infty\};\\
	\label{zlim_hom2} & y- (-) \text{ turns suprema into infima for each }
	y\in \mathbb{K}\cup\{-\infty,\infty\}.
\end{align}}
\!\!Since the ordering $\geq$ is again the reverse of
the usual one,
we introduce the notation in which 
we denote suprema and infima with respect to $\geq$
by $\bigvee$ and $\bigwedge$,
whereas infima and suprema with respect to the usual 
ordering $\leq$ by $\inf$ and $\sup$, respectively.
We obtain, for each 
$x,y\in \mathbb{K}\cup \{-\infty,\infty\}$ and
$s,t\in \mathbb{K}$,
the following:
{\allowdisplaybreaks
\begin{align}
	& x+\infty = x+\inf \emptyset
	\overset{\text{(\ref{zlim_ten1})}}=
	\inf \emptyset = \infty; \nonumber \\
	\label{sPlusMinusInfty}& s+(-\infty) = s+\inf_{t\in \mathbb{K}}\{t\} 
	\overset{\text{(\ref{zlim_ten1})}}=
	\inf_{t\in \mathbb{K}}\{s+t\}  = -\infty; \\
	&\infty+y = \inf \emptyset+y 
	\overset{\text{(\ref{zlim_ten2})}}=
	\inf \emptyset = \infty;\nonumber \\
	&(-\infty)+t = \inf_{s\in \mathbb{K}}\{s\}+t
	\overset{\text{(\ref{zlim_ten2})}}=
	\inf_{s\in \mathbb{K}}\{s+t\}  = -\infty;\nonumber \\
	\label{inftyMinusS}&\infty - s = \sup_{t\in \mathbb{K}}\{t\} - s 
	\overset{\text{(\ref{zlim_hom1})}}=
	\sup_{t\in \mathbb{K}}\{t-s\} = \infty;\\
	&(-\infty) - x = \sup\emptyset - x 
	\overset{\text{(\ref{zlim_hom1})}}=
	\sup\emptyset = -\infty;\nonumber\\
	& y- \infty = y-\inf\emptyset
	\overset{\text{(\ref{zlim_hom2})}}= 
	\sup \emptyset = -\infty;\nonumber\\
	& t- (-\infty) = t-\inf_{s\in \mathbb{K}}\{s\}
	\overset{\text{(\ref{zlim_hom2})}}= 
	\sup_{s\in \mathbb{K}}\{t-s\} = \infty.\nonumber
\end{align}}
\!\!Finally, we define two remaining cases:
{\allowdisplaybreaks
\begin{align*}
	&(-\infty)+(-\infty)=\inf_{s\in \mathbb{K}}\{s\}+(-\infty)
	\overset{\text{(\ref{zlim_ten2})}}=
	\inf_{s\in \mathbb{K}}\{s+(-\infty)\}\nonumber
	\overset{\text{(\ref{sPlusMinusInfty})}}=
	\inf_{s\in \mathbb{K}}\{-\infty\} = -\infty;\\
	&\infty-(-\infty)=\infty-\inf_{s\in \mathbb{K}}\{s\}
	\overset{\text{(\ref{zlim_hom2})}}=
	\sup_{s\in \mathbb{K}}\{\infty -s\}
	\overset{\text{(\ref{inftyMinusS})}}=
	\sup_{s\in \mathbb{K}}\{\infty\}=\infty.
\end{align*}}
\!\!Now the extensions have completed.

The proof that $\overline{\mathbb{K}}$ is an SMC-CL
is simply a little
bigger version of the case analyses we did in
Example \ref{ex:Rplus}.
\end{ex}

The last example below is of a bit different flavor: 
\begin{ex} 
Let $L=(L_0,\supset)$ be a complete Heyting algebra,
that is, a complete lattice $L_0=(L_0,\sqsubseteq)$ equipped with a 
binary operation 
$\supset \colon L_0^\op\times L_0\longrightarrow L_0$
called the implication, with the following
property:
\begin{prooftree}
	\def\fCenter{\ \sqsubseteq\ }
	\alwaysDoubleLine
	\Axiom$x\land y\fCenter z$
	\UnaryInf$x\fCenter y\supset z$
\end{prooftree}
Then, the quadruple $L=(L_0,\land,\supset,\top)$, where
$\top$ is the largest element of $L_0$, is an SMC-CL.
Such SMC-CLs
(in which the tensor product is given by the 
meet $\land$) are called 
Cartesian closed complete lattices,
which are nevertheless essentially the same as complete
Heyting algebras.

As a special case, we have the following Cartesian
closed complete lattices: the quadruple
$\overline{\mathbb{K}}_+^\text{Cart}=(\mathbb{K}_+\cup\{\infty\},\max,\supset,0)$, 
where $\mathbb{K}_+\cup\{\infty\}$ is ordered by $\geq$,
and $\supset$ (the order of arguments is 
$[x,y]=x\supset y$) is defined as 
\[
x\supset y=
\begin{cases}
0 &(x\geq y)\\
y &(x<y).
\end{cases}
\]
This appears in \cite{Law73} when
$\mathbb{K}=\mathbb{R}$.
Just as $\overline{\mathbb{K}}_+$ 
will be used to generalize metric spaces,
we will later use $\overline{\mathbb{K}}_+^\text{Cart}$ 
to generalize \textit{ultrametric spaces}.
\end{ex}

\section{Enriched Categories}
Every poset-enriched category has a particular 
enriching poset $\mathcal{V}$ which is used to describe
its structure.
Although we introduced the notion of SMC-CLs 
in the previous section,
in the early stages of the theory of poset-enriched categories,
it suffices to require the enriching poset $\mathcal{V}$
to be an SM-P.

\begin{defn}
Let $\mathcal{V}=(\mathcal{V}_0,\otimes,e)$ be
an SM-P. 
A $\mathcal{V}$\textbf{-category}
$\mathcal{A}$ is a pair
$(\ob{\mathcal{A}},$ $\Hom_\mathcal{A})$ where
\begin{itemize}
	\item $\ob{\mathcal{A}}$ is the set of
	\textbf{objects} of $\mathcal{A}$;
	\item $\Hom_\mathcal{A}$ is a map which
	assigns for each pair $a,b$ of objects 
	of $\mathcal{A}$,
	an element $\Hom_\mathcal{A}(a,b)$ of
	$\mathcal{V}_0$ called its
	\textbf{hom-object};
\end{itemize}
such that the following axioms hold:
\begin{description}[font=\normalfont]
	\item[$\qquad$(Composition law)]
	$\Hom_\mathcal{A}(a,b)\otimes\Hom_\mathcal{A}(b,c)\sqsubseteq\Hom_\mathcal{A}(a,c)\quad(\forall a,b,c\in \ob{\mathcal{A}})$;
	\item[$\qquad$(Identity law)]
	$e\sqsubseteq\Hom_\mathcal{A}(a,a)\quad(\forall a\in \ob{\mathcal{A}})$.
\end{description}
\end{defn}
The definition is best explained by the various
examples given below.
Note the apparent conflict of the term ``composition law'';
an axiom in the definition of SC-Ps
also bears the same name.
However, we will see later that
the composition law for 
$\mathcal{V}$ can be seen as an instance of the
composition law for $\mathcal{V}$-categories, thus
justifying the terminology.

\begin{ex} 
A $\textbf{2}$-category is nothing but a 
\textit{preordered set}. 
In fact, a $\textbf{2}$-category $\mathcal{A}$ 
is a pair
$(\ob{\mathcal{A}},\preceq_\mathcal{A})$ where
\begin{itemize}
	\item $\ob{\mathcal{A}}$ is a set;
	\item $\preceq_\mathcal{A}$ takes two elements
	$a,b$ of $\ob{\mathcal{A}}$ and returns
	a truth value, which can be interpreted to 
	denote whether the relation 
	$(a\preceq_\mathcal{A} b)$
	holds or not;
\end{itemize}
such that the following axioms hold:
\begin{description}[font=\normalfont]
	\item[$\qquad$(Composition law)]
	$(a\preceq_\mathcal{A} b) \myand (b\preceq_\mathcal{A} c) \vdash (a\preceq_\mathcal{A} c)\quad(\forall a,b,c\in \ob{\mathcal{A}})$;
	\item[$\qquad$(Identity law)]
	$\true\vdash (a\preceq_\mathcal{A} a)\quad(\forall a\in \ob{\mathcal{A}})$.
\end{description}
Of course, the composition law requires
transitivity of $\preceq_\mathcal{A}$ and 
the identity law the reflexivity of
$\preceq_\mathcal{A}$.
\end{ex}

\begin{ex} 
An $\overline{\mathbb{R}}_+$-category, or a 
\textbf{Lawvere metric space},
$\mathcal{A}$ is a pair
$(\ob{\mathcal{A}},$ $d_\mathcal{A})$ where
\begin{itemize}
\item
$\ob{\mathcal{A}}$ is the set of \textbf{points} of
$\mathcal{A}$;
\item
$d_\mathcal{A}$, called the 
\textbf{distance function} of $\mathcal{A}$,
takes two points $a,b$ of 
$\mathcal{A}$ and returns a value 
$d_\mathcal{A}(a,b)\in \mathbb{R}_+\cup\{\infty\}$;
\end{itemize}
such that the following axioms hold:
\begin{description}[font=\normalfont]
	\item[$\qquad$(Composition law)]
	$d_\mathcal{A}(a,b) + d_\mathcal{A} (b,c) \geq
	d_\mathcal{A} (a,c)\quad(\forall a,b,c\in \ob{\mathcal{A}})$;
	\item[$\qquad$(Identity law)]
	$0\geq d_\mathcal{A}(a,a)\quad(\forall a\in \ob{\mathcal{A}})$.
\end{description}
Since $0\geq x$ if and only if $x=0$, 
the identity law is equivalent to 
$d_\mathcal{A}(a,a) = 0$.
Lawvere remarked in \cite{Law73} that every metric
space is an $\overline{\mathbb{R}}_+$-category.
In fact, the composition law is nothing but the 
\textit{triangle inequality},
and the identity law is also a usual requirement 
for the distance functions of metric spaces.

Note that the notion of Lawvere metric spaces
generalizes that of metric spaces in the following
points:
\begin{itemize}
\item
the distance function $d_\mathcal{A}$ is not 
necessarily	symmetric, i.e., 
there may exist points $a,b$ of $\mathcal{A}$
with $d_\mathcal{A}(a,b)\neq d_\mathcal{A}(b,a)$;
\item
the values of distance function can attain $\infty$;
\item
the points $a,b$ of $\mathcal{A}$ with 
$d_\mathcal{A}(a,b)=0$, or even 
$d_\mathcal{A}(a,b)=d_\mathcal{A}(b,a)=0$,
may be distinct.
\end{itemize}
\end{ex}

\begin{ex} 
$\overline{\mathbb{Z}}_+$-categories are
defined similarly.
One can regard them as a generalization of 
$\mathbb{Z}_+$-valued (or discrete) metric spaces.
\end{ex}

The following classes of enriched categories will be 
the main field of our study: 

\begin{ex} 
A $\overline{\mathbb{K}}$-category
$\mathcal{A}$ is a pair
$(\ob{\mathcal{A}},$ $d_\mathcal{A})$ where
\begin{itemize}
\item
$\ob{\mathcal{A}}$ is the set of \textbf{points} of
$\mathcal{A}$;
\item
$d_\mathcal{A}$, called the 
\textbf{distance function} of $\mathcal{A}$,
takes two points $a,b$ of 
$\mathcal{A}$ and returns a value 
$d_\mathcal{A}(a,b)\in \mathbb{K}\cup\{-\infty,\infty\}$;
\end{itemize}
such that the following axioms hold:
\begin{description}[font=\normalfont]
	\item[$\qquad$(Composition law)]
	$d_\mathcal{A}(a,b) + d_\mathcal{A} (b,c) \geq
	d_\mathcal{A} (a,c)\quad(\forall a,b,c\in \ob{\mathcal{A}})$;
	\item[$\qquad$(Identity law)]
	$0\geq d_\mathcal{A}(a,a)\quad(\forall a\in \ob{\mathcal{A}})$.
\end{description}
Note that for each $a\in \ob{\mathcal{A}}$ 
the composition law yields 
\begin{align}\label{comp_k}
	d_\mathcal{A}(a,a) + d_\mathcal{A} (a,a) \geq
	d_\mathcal{A} (a,a);
\end{align}
we can combine this with the identity law and deduce
\[
	d_\mathcal{A}(a,a)=0\text{ or }-\infty,
\]
because no
negative number $d_\mathcal{A}(a,a)\in\mathbb{K}$
satisfies (\ref{comp_k}).
\end{ex}

\begin{ex} 
A $\overline{\mathbb{K}}_+^\text{Cart}$-category
$\mathcal{A}$ is a pair
$(\ob{\mathcal{A}},$ $d_\mathcal{A})$ where
\begin{itemize}
\item
$\ob{\mathcal{A}}$ is the set of \textbf{points} of
$\mathcal{A}$;
\item
$d_\mathcal{A}$
takes two points $a,b$ of 
$\mathcal{A}$ and returns a value 
$d_\mathcal{A}(a,b)\in \mathbb{K}_+\cup\{\infty\}$;
\end{itemize}
such that the following axioms hold:
\begin{description}[font=\normalfont]
	\item[$\qquad$(Composition law)]
	$\max\{d_\mathcal{A}(a,b),d_\mathcal{A} (b,c)\} \geq
	d_\mathcal{A} (a,c)\quad(\forall a,b,c\in \ob{\mathcal{A}})$;
	\item[$\qquad$(Identity law)]
	$0\geq d_\mathcal{A}(a,a)\quad(\forall a\in \ob{\mathcal{A}})$.
\end{description}
Of course, we can rewrite the identity law as 
$d_\mathcal{A}(a,a) = 0$.
As observed in \cite{Law73} 
(when $\mathbb{K}=\mathbb{R}$),
$\overline{\mathbb{K}}_+^\text{Cart}$-categories
generalize (discrete or ordinary) ultrametric spaces.
\end{ex}

\section{Enriched Functors and Canonical Orderings}
One can understand the notion of $\mathcal{V}$-functors as
structure-preserving maps
between $\mathcal{V}$-categories.
In the general enriched category theory (where the enriching
category is not necessarily a poset), there is a yet  
higher notion of
$\mathcal{V}$\textit{-natural transformations},
which are regarded as \textit{maps between
$\mathcal{V}$-functors}.
In our theory, however, the notion of
$\mathcal{V}$-natural transformations are largely 
trivialized, and it only remains as a certain 
preorder relation among
$\mathcal{V}$-functors with specified domain and codomain
$\mathcal{V}$-categories.
\begin{defn} 
Let $\mathcal{V}=(\mathcal{V}_0,\otimes,e)$ be
an SM-P, and 
$\mathcal{A}=$ $(\ob{\mathcal{A}},\Hom_\mathcal{A})$ and
$\mathcal{B}=(\ob{\mathcal{B}},\Hom_\mathcal{B})$ be
$\mathcal{V}$-categories. 
A $\mathcal{V}$\textbf{-functor} 
\[
	F\colon \mathcal{A}\longrightarrow\mathcal{B}
\]
from $\mathcal{A}$ to $\mathcal{B}$ is a map 
\[
	\ob{F}\colon \ob{\mathcal{A}}\longrightarrow\ob{\mathcal{B}}
\]
between the sets of objects
with the following property:
\begin{description}[font=\normalfont]
	\item[$\ \ $(Increasing condition)]
	$\Hom_\mathcal{A}(a,a^\prime)\sqsubseteq
	\Hom_\mathcal{B}\big(\ob{F}(a),\ob{F}(a^\prime)\big)$ $(\forall a,a^\prime\in \ob{\mathcal{A}})$.
\end{description}
In case a stronger condition
\[
	\Hom_\mathcal{A}(a,a^\prime)=
	\Hom_\mathcal{B}\big(\ob{F}(a),\ob{F}(a^\prime)\big)\quad(\forall a,a^\prime\in \ob{\mathcal{A}})
\]
holds, $F$ is called a \textbf{fully faithful}
$\mathcal{V}$-functor.

$\mathcal{A}$ is called the \textbf{domain} of $F$ and
$\mathcal{B}$ the \textbf{codomain} of $F$. 
We denote the set of all $\mathcal{V}$-functors
from $\mathcal{A}$ to $\mathcal{B}$ by 
$[\mathcal{A},\mathcal{B}]_0$.
\end{defn}
For brevity, we often denote the values like $\ob{F}(a)$
by $F(a)$.
Clearly, $\mathcal{V}$-functors are closed under
composition:

\begin{prop} 
Let $F\colon \mathcal{A}\longrightarrow \mathcal{B}$ and
$G\colon \mathcal{B}\longrightarrow \mathcal{C}$ be
$\mathcal{V}$-functors. 
Then, the composite map
\[
	\ob{G}\circ \ob{F}\colon \ob{\mathcal{A}}\longrightarrow \ob{\mathcal{B}}\longrightarrow\ob{\mathcal{C}}
\]
defines a $\mathcal{V}$-functor 
$G\circ F\colon\mathcal{A}\longrightarrow\mathcal{C}$.
\end{prop}
\begin{proof}
For all $a,a^\prime\in \ob{\mathcal{A}}$, we have
\[
	\Hom_{\mathcal{A}}(a,a^\prime)
	\sqsubseteq \Hom_{\mathcal{B}}\big(F(a),F(a^\prime)\big)
	\sqsubseteq
	\Hom_{\mathcal{C}}\big(G\circ F(a),G\circ F(a^\prime)\big).
\]
\end{proof}

Also, for every $\mathcal{V}$-category $\mathcal{A}$,
the identity map on $\ob{\mathcal{A}}$ defines a 
$\mathcal{V}$-functor. 
It is called the \textbf{identity} 
$\mathcal{V}$\textbf{-functor} on 
$\mathcal{A}$ and denoted by $1_\mathcal{A}$.

\begin{defn}
Let $\mathcal{V}$ be
an SM-P, and $\mathcal{A}$ and $\mathcal{B}$ be 
$\mathcal{V}$-categories.
A $\mathcal{V}$-functor 
\[
F\colon \mathcal{A}\longrightarrow \mathcal{B}
\] 
from $\mathcal{A}$ to $\mathcal{B}$
is an \textbf{isomorphism}
if there exists a $\mathcal{V}$-functor 
(called its \textbf{inverse})
\[
G\colon \mathcal{B}\longrightarrow \mathcal{A}
\]
from $\mathcal{B}$ to $\mathcal{A}$ such that
$G\circ F=1_{\mathcal{A}}$ and 
$F\circ G=1_{\mathcal{B}}$
hold.

If there is an isomorphism between 
$\mathcal{A}$ and $\mathcal{B}$, then they are said to be 
\textbf{isomorphic} and written as 
$\mathcal{A}\cong\mathcal{B}$.
\end{defn} 

\begin{prop}\label{iso_ffb}
A $\mathcal{V}$-functor 
\[
F\colon \mathcal{A}\longrightarrow \mathcal{B}
\] 
is an isomorphism if and only if it is fully faithful,
and the map
\[
\ob{F}\colon \ob{\mathcal{A}}\longrightarrow \ob{\mathcal{B}}
\]
between the sets of objects is a bijection. 
\end{prop}
\begin{proof}
If $F$ is an isomorphism, let $G$ be its inverse.
Then for each pair $a,a^\prime\in \ob{\mathcal{A}}$,
\begin{align*}
\Hom_\mathcal{A}(a,a^\prime)
&\sqsubseteq \Hom_\mathcal{B}\big(F(a),F(a^\prime)\big)\\
&\sqsubseteq \Hom_\mathcal{A}\big(G\circ F(a),G\circ F(a^\prime)\big)\\
&= \Hom_\mathcal{A}(a,a^\prime)
\end{align*}
holds and $F$ is fully faithful.
Also, $\ob{G}$ witnesses that 
$\ob{F}$ is a bijection as a map.

Suppose then, that $F$ is fully faithful and 
$\ob{F}$ is a bijection.
Let 
$\ob{G}\colon \ob{\mathcal{B}}\longrightarrow\ob{\mathcal{A}}$ 
be the inverse of $\ob{F}$ (as a map).
Then, for each pair $b,b^\prime\in \ob{\mathcal{B}}$, 
\[
\Hom_\mathcal{B}(b,b^\prime)
= \Hom_\mathcal{B}\big(F\circ G(b),F\circ G(b^\prime)\big)
= \Hom_\mathcal{A}\big(G(b),G(b^\prime)\big)
\]
holds, thus $G$ is also a (fully faithful) 
$\mathcal{V}$-functor.
By construction, it is clear that $G\circ F=1_\mathcal{A}$
and $F\circ G=1_\mathcal{B}$ hold.
\end{proof}

\begin{defn}\label{def:canord} 
Let $\mathcal{V}=(\mathcal{V}_0,\otimes,e)$ be
an SM-P, and 
$\mathcal{A}=$ $(\ob{\mathcal{A}},\Hom_\mathcal{A})$ and
$\mathcal{B}=(\ob{\mathcal{B}},\Hom_\mathcal{B})$ be
$\mathcal{V}$-categories. 
The set $[\mathcal{A},\mathcal{B}]_0$ of all $\mathcal{V}$-functors from $\mathcal{A}$
to $\mathcal{B}$ admits a canonical relation
$\Rightarrow$ defined as follows:
\[
	F\Rightarrow G \iff e\sqsubseteq \Hom_\mathcal{B}
	\big(F(a),G(a)\big)\quad(\forall a\in \ob{\mathcal{A}});
\]
where 
$F,G\colon \mathcal{A}\longrightarrow \mathcal{B}$
are $\mathcal{V}$-functors.
We call $\Rightarrow$ the \textbf{canonical ordering}
on $[\mathcal{A},\mathcal{B}]_0$.
\end{defn}

\begin{prop}
The relation $\Rightarrow$ defined on the set
$[\mathcal{A},\mathcal{B}]_0$ is a preorder relation.
\end{prop}
\begin{proof}
\begin{description}[font=\normalfont]
\item[[Reflexivity\!\!\!]]
We wish to prove
$e\sqsubseteq \Hom_\mathcal{B}\big(F(a),F(a)\big)$,
but this is an instance of the identity law for 
$\mathcal{B}$.
\item[[Transitivity\!\!\!]]
Suppose $F\Rightarrow G$ and $G\Rightarrow H$ for
$F,G,H\colon \mathcal{A}\longrightarrow \mathcal{B}$.
Then an application of the composition law for $\mathcal{B}$
yields
\begin{align*}
	e&=e\otimes e\\
	&\sqsubseteq \Hom_\mathcal{B}\big(F(a),G(a)\big)\otimes
	\Hom_\mathcal{B}\big(G(a),H(a)\big)\\
	&\sqsubseteq \Hom_\mathcal{B}\big(F(a),H(a)\big).
\end{align*}
\end{description}
\end{proof}
Also, the relation $\Rightarrow$ is preserved under 
composition of $\mathcal{V}$-functors, in the sense of
the following proposition:
\begin{prop}
Let $F,G\colon\mathcal{A}\longrightarrow \mathcal{B}$ and
$H,K\colon\mathcal{B}\longrightarrow\mathcal{C}$ be
$\mathcal{V}$-functors such that $F\Rightarrow G$ and 
$H\Rightarrow K$ hold.
Then, for the composite $\mathcal{V}$-functors
$H\circ F, K\circ G\colon\mathcal{A}\longrightarrow\mathcal{C}$,
$H\circ F\Rightarrow K\circ G$ holds. 
\end{prop}
\begin{proof}
\begin{align*}
e&= e\otimes e\\
&\sqsubseteq \Hom_{\mathcal{B}}\big(F(a),F(a)\big)\otimes \Hom_{\mathcal{B}}\big(F(a),G(a)\big)\\
&\sqsubseteq \Hom_{\mathcal{C}}\big(H\circ F(a),K\circ F(a)\big)\otimes \Hom_{\mathcal{C}}\big(K\circ F(a),K\circ G(a)\big)\\
&\sqsubseteq \Hom_{\mathcal{C}}\big(H\circ F(a),K\circ G(a)\big).
\end{align*}
\end{proof}

Below we see examples of $\mathcal{V}$-functors
and relation $\Rightarrow$.

\begin{ex} 
Let us take $\mathcal{V}=\textbf{2}$.
We claim that \textbf{2}-functors between 
\textbf{2}-categories are just 
\textit{monotonic maps}.
Indeed, for \textbf{2}-categories
$\mathcal{A}=(\ob{\mathcal{A}},\preceq_\mathcal{A})$ and
$\mathcal{B}=(\ob{\mathcal{B}},\preceq_\mathcal{B})$,
a \textbf{2}-functor 
$F\colon \mathcal{A}\longrightarrow \mathcal{B}$ is a map
between the sets of objects such that 
\begin{description}[font=\normalfont]
	\item[$\qquad$(Increasing condition)]
	$(a\preceq_\mathcal{A}a^\prime)\vdash
	\big(F(a)\preceq_\mathcal{B}F(a^\prime)\big)\quad(\forall a,a^\prime\in \ob{\mathcal{A}})$
\end{description}
holds.
If $F$ is moreover fully faithful, 
\[
	(a\preceq_\mathcal{A}a^\prime)=
	\big(F(a)\preceq_\mathcal{B}F(a^\prime)\big)\quad(\forall a,a^\prime\in \ob{\mathcal{A}})
\]
holds and $F$ is nothing but a usual embedding between preorders.
The notion of isomorphic \textbf{2}-categories 
agrees with that of isomorphic preorders.

The relation $\Rightarrow$ turns out to be the usual 
ordering on the set of monotonic maps:
\[
F\Rightarrow G 
\iff \true \vdash \big(F(a)\preceq_\mathcal{B}G(a)\big)
\quad(\forall a\in \ob{\mathcal{A}}).
\]
\end{ex}

\begin{ex} 
For $\overline{\mathbb{K}}_+$-categories, one can think of
$\overline{\mathbb{K}}_+$-functors as
\textit{nonexpansive maps}.
Let
$\mathcal{A}=(\ob{\mathcal{A}},d_\mathcal{A})$ and
$\mathcal{B}=(\ob{\mathcal{B}},d_\mathcal{B})$ be 
$\overline{\mathbb{K}}_+$-categories.
A $\overline{\mathbb{K}}_+$-functor 
$F\colon \mathcal{A}\longrightarrow \mathcal{B}$ is a map
between the sets of points with the following property: 
\begin{description}[font=\normalfont]
	\item[$\qquad$(Increasing condition)]
	$d_\mathcal{A}(a,a^\prime)\geq
	d_\mathcal{B}\big(F(a),F(a^\prime)\big)\quad(\forall a,a^\prime\in \ob{\mathcal{A}})$.
\end{description}
If $F$ is fully faithful, 
\[
	d_\mathcal{A}(a,a^\prime)=
	d_\mathcal{B}\big(F(a),F(a^\prime)\big)\quad(\forall a,a^\prime\in \ob{\mathcal{A}})
\]
holds; i.e., in the terminology of metric spaces, 
$F$ is an isometry.
The notion of isomorphic 
$\overline{\mathbb{K}}_+$-categories 
specializes to that of isometric metric spaces.

The relation $\Rightarrow$ is defined as follows:
\begin{align*}
F\Rightarrow G 
&\iff 0 \geq d_\mathcal{B}\big(F(a),G(a)\big)
\quad(\forall a\in \ob{\mathcal{A}})\\
&\iff d_\mathcal{B}\big(F(a),G(a)\big)=0
\quad(\forall a\in \ob{\mathcal{A}}).
\end{align*}
\end{ex}

\begin{ex} 
Let
$\mathcal{A}=(\ob{\mathcal{A}},d_\mathcal{A})$ and
$\mathcal{B}=(\ob{\mathcal{B}},d_\mathcal{B})$ be 
$\overline{\mathbb{K}}$-categories.
A $\overline{\mathbb{K}}$-functor 
$F\colon \mathcal{A}\longrightarrow \mathcal{B}$ is a map
between the sets of points with the following property: 
\begin{description}[font=\normalfont]
	\item[$\qquad$(Increasing condition)]
	$d_\mathcal{A}(a,a^\prime)\geq
	d_\mathcal{B}\big(F(a),F(a^\prime)\big)\quad(\forall a,a^\prime\in \ob{\mathcal{A}})$.
\end{description}
A fully faithful $\overline{\mathbb{K}}$-functor 
is an isometry (in an obvious sense).

The relation $\Rightarrow$ is defined as follows:
\[
F\Rightarrow G 
\iff 0 \geq d_\mathcal{B}\big(F(a),G(a)\big)
\quad(\forall a\in \ob{\mathcal{A}}).
\]
\end{ex}

\begin{ex} 
Let
$\mathcal{A}=(\ob{\mathcal{A}},d_\mathcal{A})$ and
$\mathcal{B}=(\ob{\mathcal{B}},d_\mathcal{B})$ be 
$\overline{\mathbb{K}}_+^\text{Cart}$-categories.
A $\overline{\mathbb{K}}_+^\text{Cart}$-functor 
$F\colon \mathcal{A}\longrightarrow \mathcal{B}$ is a map
between the sets of points with the following property: 
\begin{description}[font=\normalfont]
	\item[$\qquad$(Increasing condition)]
	$d_\mathcal{A}(a,a^\prime)\geq
	d_\mathcal{B}\big(F(a),F(a^\prime)\big)\quad(\forall a,a^\prime\in \ob{\mathcal{A}})$.
\end{description}
A fully faithful
$\overline{\mathbb{K}}_+^\text{Cart}$-functor is an isometry.

The relation $\Rightarrow$ is defined as follows:
\begin{align*}
F\Rightarrow G 
&\iff 0 \geq d_\mathcal{B}\big(F(a),G(a)\big)
\quad(\forall a\in \ob{\mathcal{A}})\\
&\iff d_\mathcal{B}\big(F(a),G(a)\big)=0
\quad(\forall a\in \ob{\mathcal{A}}).
\end{align*}
\end{ex}

\section{The Opposites of $\mathcal{V}$-Categories}
In this section we present a way to construct new 
$\mathcal{V}$-categories out of some already existing
$\mathcal{V}$-categories. 

\begin{prop} 
Let $\mathcal{V}=(\mathcal{V}_0,\otimes,e)$ 
be an SM-P, and 
$\mathcal{A}=(\ob{\mathcal{A}},$ $\Hom_\mathcal{A})$
be a $\mathcal{V}$-category.
Then there is a $\mathcal{V}$-category denoted by 
$\mathcal{A}^\op=(\ob{\mathcal{A}^\op},\Hom_{\mathcal{A}^\op})$,
where
\begin{itemize}
\item
$\ob{\mathcal{A}^\op} =\ob{\mathcal{A}}$, i.e.,
the set of objects for $\mathcal{A}^\op$
is identical with that for 
$\mathcal{A}$;
\item
the hom-object function
$\Hom_{\mathcal{A}^\op}$
is given as
\[
	\Hom_{\mathcal{A}^\op}(a,b)	= \Hom_\mathcal{A}(b,a).
\]
\end{itemize} 
\end{prop}
\begin{proof}
\begin{description}[font=\normalfont]
\item[[Composition law\!\!\!]]
The claim is 
\[
\Hom_{\mathcal{A}^\op}(a,b)\otimes 
\Hom_{\mathcal{A}^\op}(b,c)\sqsubseteq 
\Hom_{\mathcal{A}^\op}(a,c).
\]
It is proved as follows: 
\begin{align*}
\Hom_{\mathcal{A}^\op}(a,b)\otimes \Hom_{\mathcal{A}^\op}(b,c)
&=\Hom_{\mathcal{A}}(b,a)\otimes \Hom_{\mathcal{A}}(c,b)\\
&=\Hom_{\mathcal{A}}(c,b)\otimes \Hom_{\mathcal{A}}(b,a)\\
&\sqsubseteq \Hom_{\mathcal{A}}(c,a)\\
&=\Hom_{\mathcal{A}^\op}(a,c).
\end{align*}
\item[[Identity law\!\!\!]]
The claim is 
$e \sqsubseteq\Hom_{\mathcal{A}^\op}(a,a)$,
an obvious formula.
\end{description}
\end{proof}

\begin{defn}
The category $\mathcal{A}^\op$ is called 
the \textbf{opposite category} of 
$\mathcal{A}$.
\end{defn}

\begin{ex} 
Let $\mathcal{A}=(\ob{\mathcal{A}},\preceq_\mathcal{A})$ 
and $\mathcal{B}=(\ob{\mathcal{B}},\preceq_\mathcal{B})$
be $\textbf{2}$-categories. 
The opposite category
$\mathcal{A}^\op=(\ob{\mathcal{A}^\op},\preceq_{\mathcal{A}^\op})$
is the opposite (or dual) of preorder:
\begin{itemize}
	\item $\ob{\mathcal{A}^\op}=\ob{\mathcal{A}}$;
	\item $(a\preceq_{\mathcal{A}^\op} b)=
	(b\preceq_{\mathcal{A}}a)$.
\end{itemize}
\end{ex}

\begin{ex} 
Let $\mathcal{A}=(\ob{\mathcal{A}},d_\mathcal{A})$ 
and $\mathcal{B}=(\ob{\mathcal{B}},d_\mathcal{B})$
be $\overline{\mathbb{K}}_+$-categories. 
The opposite category
$\mathcal{A^\op}=(\ob{\mathcal{A^\op}},d_\mathcal{A^\op})$
is defined as follows:
\begin{itemize}
\item
$\ob{\mathcal{A}^\op}=\ob{\mathcal{A}}$;
\item
$d_{\mathcal{A}^\op}(a,b)=d_{\mathcal{A}}(b,a)$.
\end{itemize}
Note that since a classical metric space has
a symmetric distance function, it is isometric
to its opposite.
\end{ex}

\begin{ex} 
Let $\mathcal{A}=(\ob{\mathcal{A}},d_\mathcal{A})$ 
and $\mathcal{B}=(\ob{\mathcal{B}},d_\mathcal{B})$
be $\overline{\mathbb{K}}$-categories. 
The opposite category
$\mathcal{A^\op}=(\ob{\mathcal{A^\op}},d_\mathcal{A^\op})$
is defined as follows:
\begin{itemize}
\item
$\ob{\mathcal{A}^\op}=\ob{\mathcal{A}}$;
\item
$d_{\mathcal{A}^\op}(a,b)=d_{\mathcal{A}}(b,a)$.
\end{itemize}
\end{ex}

\begin{ex} 
Let $\mathcal{A}=(\ob{\mathcal{A}},d_\mathcal{A})$ 
and $\mathcal{B}=(\ob{\mathcal{B}},d_\mathcal{B})$
be $\overline{\mathbb{K}}_+^\text{Cart}$-categories. 
The opposite category
$\mathcal{A^\op}=(\ob{\mathcal{A^\op}},d_\mathcal{A^\op})$
is defined as follows:
\begin{itemize}
\item
$\ob{\mathcal{A}^\op}=\ob{\mathcal{A}}$;
\item
$d_{\mathcal{A}^\op}(a,b)=d_{\mathcal{A}}(b,a)$.
\end{itemize}
\end{ex}

\section{$\mathcal{V}$ as a $\mathcal{V}$-Category}
We henceforth assume that the enriching category 
$\mathcal{V}$ is an \textit{SMC-P}.
This assumption enables us to endow
(the underlying set of) $\mathcal{V}_0$ with
a canonical $\mathcal{V}$-category structure:
\begin{prop}
Let $\mathcal{V}=(\mathcal{V}_0,\otimes,[-,-],e)$ 
be an SMC-P.
Then there is a $\mathcal{V}$-category, 
also denoted by 
$\mathcal{V}=(\mathcal{V}_0,[-,-])$, where
\begin{itemize}
\item
$\ob{\mathcal{V}}$ is $\mathcal{V}_0$, the underlying
set of the poset $\mathcal{V}_0$;
\item
the hom-object function is the internal-hom operation $[-,-]$.
\end{itemize} 
\end{prop}

\begin{proof}
\begin{description}[font=\normalfont]
\item[[Composition law\!\!\!]]
The claim is $[x,y]\otimes[y,z]\sqsubseteq [x,z]$ for 
all $x,y,z\in \mathcal{V}_0$, but
this is equivalent to the composition law for SC-Ps, 
which we know to hold:
\begin{prooftree}
	\def\fCenter{\ \sqsubseteq\ }
	\alwaysDoubleLine
	\Axiom$[x,y]\otimes[y,z]\fCenter [x,z]$
	\UnaryInf$[y,z]\otimes[x,y]\fCenter [x,z]$
	\UnaryInf$[y,z]\fCenter [[x,y],[x,z]]$
\end{prooftree}
Incidentally, this proof tree explains the duplication 
of the name ``composition law.''
\item[[Identity law\!\!\!]]
The claim is $e\sqsubseteq [z,z]$ for 
all $z\in \mathcal{V}_0$.
Using the commutative law for $[-,-]$, it suffices to 
show $z\sqsubseteq [e,z]$.
But the unit law for $[-,-]$ states $z=[e,z]$, so
the proof is done.
\end{description}
\end{proof}

Now the name of the operation ``internal-hom'' is motivated.
In a sense, it acts like $\Hom$, and takes values inside
$\mathcal{V}$. 
As an aside, one can think of the (external)
hom of the poset $\mathcal{V}_0$
to be that takes two elements $x,y$ of $\mathcal{V}_0$ and 
returns the truth value of $x\sqsubseteq y$, which is a
value in $\textbf{2}$; thus 
(unless $\mathcal{V}=\textbf{2}$)
outside $\mathcal{V}$.

Let us take a look at some examples of 
such canonical $\mathcal{V}$-categories.

\begin{ex} 
$\textbf{2}_0=\{\true,\false\}$ becomes a
\textbf{2}-category 
(i.e., preordered set) in a canonical way.
By the way, $\textbf{2}_0$ is defined to be the underlying
poset of \textbf{2}, hence is already a preordered set.
In fact, the \textbf{2}\textit{-category} \textbf{2} is
isomorphic to the 
\textit{underlying poset} $\textbf{2}_0$ of
the SMC-CL \textbf{2}:
one can see this by observing that the tables for $\vdash$
and for $\supset$ in  
Example \ref{ex:2} carry the same pattern.
\end{ex}

\begin{ex} 
The set
$\mathbb{K}_+\cup\{\infty\}$
becomes a $\overline{\mathbb{K}}_+$-category 
by the following distance function:
\[
	d_{\overline{\mathbb{K}}_+}(x,y)=y\dotminus x.
\]
Therefore, if $x\leq y$ the distance to climb
from $x$ to $y$ is the same as their the difference, 
but if $x\geq y$, we descend and the corresponding 
distance is $0$.

Note that the $\overline{\mathbb{K}}_+$-category 
$\overline{\mathbb{K}}_+$ is \textit{not} 
an ordinary (discrete) metric space;
among others, its distance function is highly asymmetric.
\end{ex}

\begin{ex} 
The set
$\mathbb{K}\cup\{-\infty,\infty\}$
becomes a $\overline{\mathbb{K}}$-category 
by the following distance function:
\[
	d_{\overline{\mathbb{K}}}(x,y)=y-x.
\]
Recall that in any $\overline{\mathbb{K}}$-category 
$\mathcal{A}$
\[
	d_{\mathcal{A}}(a,a)=0\text{ or }-\infty
\]
holds.
$\overline{\mathbb{K}}$ as a
$\overline{\mathbb{K}}$-category 
indicates that the distances between the same points may 
take both $0$ and $-\infty$ in one
$\overline{\mathbb{K}}$-category; in fact, if $s\in \mathbb{K}$,
\begin{align*}
d_{\overline{\mathbb{K}}}(s,s)&= 0\text{, and}\\
d_{\overline{\mathbb{K}}}(-\infty,-\infty)=d_{\overline{\mathbb{K}}}(\infty,\infty)&= -\infty
\end{align*}
hold.
\end{ex}

\begin{ex} 
The set
$\mathbb{K}_+\cup\{\infty\}$
becomes a $\overline{\mathbb{K}}_+^\text{Cart}$-category
by the following distance function:
\begin{align*}
	d_{\overline{\mathbb{K}}_+^\text{Cart}}(x,y)
	&=x\supset y\\
	&=\begin{cases}
	0 &(x\geq y)\\
	y &(x<y).
	\end{cases}
\end{align*}
\end{ex}

\section{Functor Categories and the Yoneda Embedding}
Now the theory has reached the stage where we perform 
some limiting processes within $\mathcal{V}_0$.
Therefore, we finally assume that our enriching category
$\mathcal{V}$ is an \textit{SMC-CL}.
We begin this section with the definition of 
a $\mathcal{V}$-category of all $\mathcal{V}$-functors 
with specified domain and codomain, which is intuitively
seen as an analog of a function space in some sense.

\begin{prop}
Let $\mathcal{V}=(\mathcal{V}_0,\otimes,[-,-],e)$ 
be an SMC-CL, and 
$\mathcal{A}=(\ob{\mathcal{A}},\Hom_\mathcal{A})$
and 
$\mathcal{B}=(\ob{\mathcal{B}},\Hom_\mathcal{B})$ be
$\mathcal{V}$-categories.
Then there is a $\mathcal{V}$-category denoted by 
$\funct{\mathcal{A}}{\mathcal{B}}=([\mathcal{A},\mathcal{B}]_0,\Hom_\funct{\mathcal{A}}{\mathcal{B}})$, where
\begin{itemize}
\item
$\ob{\funct{\mathcal{A}}{\mathcal{B}}}$ is
$[\mathcal{A},\mathcal{B}]_0$, the set of 
all $\mathcal{V}$-functors from $\mathcal{A}$ 
to $\mathcal{B}$;
\item
the hom-object function
$\Hom_\funct{\mathcal{A}}{\mathcal{B}}$
is given as
\[
	\Hom_\funct{\mathcal{A}}{\mathcal{B}}(F,G)
	= \bigwedge_{a\in \ob{\mathcal{A}}}
	\left\{\Hom_\mathcal{B}\big(F(a),G(a)\big)\right\},
\]
where the infimum is taken in the complete lattice 
$\mathcal{V}_0$.
\end{itemize} 
\end{prop}
\begin{proof}
\begin{description}[font=\normalfont]
\item[[Composition law\!\!\!]]
The claim is 
\[
\Hom_\funct{\mathcal{A}}{\mathcal{B}}(F,G)\otimes 
\Hom_\funct{\mathcal{A}}{\mathcal{B}}(G,H)\sqsubseteq 
\Hom_\funct{\mathcal{A}}{\mathcal{B}}(F,H).
\]
For each $a\in \ob{\mathcal{A}}$, we have
\begin{align*}
&\Hom_\funct{\mathcal{A}}{\mathcal{B}}(F,G)\otimes 
\Hom_\funct{\mathcal{A}}{\mathcal{B}}(G,H)\\
&\qquad= \bigwedge_{a^\prime\in \ob{\mathcal{A}}}
\left\{\Hom_\mathcal{B}\big(F(a^\prime),G(a^\prime)\big)\right\}\otimes
\bigwedge_{a^\prime\in \ob{\mathcal{A}}}
\left\{\Hom_\mathcal{B}\big(G(a^\prime),H(a^\prime)\big)\right\}\\
&\qquad\sqsubseteq 
\Hom_\mathcal{B}\big(F(a),G(a)\big) \otimes
\Hom_\mathcal{B}\big(G(a),H(a)\big)\\
&\qquad\sqsubseteq
\Hom_\mathcal{B}\big(F(a),H(a)\big).
\end{align*}
Taking the infimum with respect to $a$, we
obtain the required result.
\item[[Identity law\!\!\!]]
The claim is 
$e\sqsubseteq \Hom_\funct{\mathcal{A}}{\mathcal{B}}(F,F)$.
The identity law for $\mathcal{B}$ assures 
\[
	e\sqsubseteq \Hom_\mathcal{B}\big(F(a),F(a)\big),
\]
and the inequality is preserved after we
take the infimum with respect to $a$.
\end{description}
\end{proof}

\begin{defn}
The $\mathcal{V}$-category  
$\funct{\mathcal{A}}{\mathcal{B}}$ is called a
\textbf{functor category},
$\mathcal{A}$ its \textbf{domain} and $\mathcal{B}$
its \textbf{codomain}.
\end{defn}

Given a $\mathcal{V}$-category $\mathcal{A}$,
there are many functor categories associated with $\mathcal{A}$. 
But two of them are of special importance; 
they are called the
\textit{presheaf category} and the
\textit{op-copresheaf category} of $\mathcal{A}$.

\begin{defn}
Let $\mathcal{V}=(\mathcal{V}_0,\otimes,[-,-],e)$ 
be an SMC-CL, and 
$\mathcal{A}=(\ob{\mathcal{A}},\Hom_\mathcal{A})$
be a $\mathcal{V}$-category.
A \textbf{presheaf} $P$ on $\mathcal{A}$ is a
$\mathcal{V}$-functor
$P\colon\mathcal{A}^\op \longrightarrow \mathcal{V}$,
that is, a map 
$\ob{P}\colon \ob{\mathcal{A}}\longrightarrow \mathcal{V}_0$
satisfying
\[
	\Hom_\mathcal{A}(a,b)\sqsubseteq [P(b),P(a)]
\]
for every $a,b\in \ob{\mathcal{A}}$.
A \textbf{copresheaf} $Q$ on $\mathcal{A}$ is a
$\mathcal{V}$-functor
$Q\colon\mathcal{A} \longrightarrow \mathcal{V}$,
that is, a map 
$\ob{Q}\colon \ob{\mathcal{A}}\longrightarrow \mathcal{V}_0$
satisfying
\[
	\Hom_\mathcal{A}(a,b)\sqsubseteq [Q(a),Q(b)]
\]
for every $a,b\in \ob{\mathcal{A}}$.

We call the functor category
$\funct{\mathcal{A}^\op}{\mathcal{V}}$
the \textbf{presheaf category} of $\mathcal{A}$, and
the category
$\funct{\mathcal{A}}{\mathcal{V}}^\op$
the \textbf{op-copresheaf category} of $\mathcal{A}$.
\end{defn}

Presheaf categories and 
op-copresheaf categories are important because
the original category embeds into them canonically.
This is called the Yoneda embedding theorem;
but before we present and prove it, let us
see some examples of 
functor categories and presheaf or op-copresheaf
categories:

\begin{ex} 
Let $\mathcal{A}=(\ob{\mathcal{A}},\preceq_\mathcal{A})$ 
and $\mathcal{B}=(\ob{\mathcal{B}},\preceq_\mathcal{B})$
be $\textbf{2}$-categories. 
The functor category 
$\funct{\mathcal{A}}{\mathcal{B}}=([\mathcal{A},\mathcal{B}]_0,\preceq_\funct{\mathcal{A}}{\mathcal{B}})$ 
is given as follows, noting that $\bigwedge$ in
$\textbf{2}_0$ corresponds to ``for all'':
\begin{itemize}
	\item $[\mathcal{A},\mathcal{B}]_0$ is the set of
	all monotonic maps from $\mathcal{A}$ to
	$\mathcal{B}$;
	\item $(F\preceq_\funct{\mathcal{A}}{\mathcal{B}}G)=\true\iff
	\big(F(a)\preceq_\mathcal{B}G(a)\big)=\true$ for all
	$a\in \ob{\mathcal{A}}$.
\end{itemize}
Thus, as a binary relation on the set
$[\mathcal{A},\mathcal{B}]_0$, 
$\preceq_\funct{\mathcal{A}}{\mathcal{B}}$ 
coincides with the canonical ordering $\Rightarrow$.

A presheaf $P$ on $\mathcal{A}$ is a monotonic map
$P\colon\mathcal{A}^\op \longrightarrow \textbf{2}$,
and by taking the inverse image of 
$\{\true\}\subseteq\textbf{2}_0$, it
can be seen as a \textit{lower set} of $\mathcal{A}$
(a lower set $S$ of a preorder $(A,\preceq)$ is a subset
$S\subseteq A$ such that $a\in S$ and $b\preceq a$
imply $b\in S$).
In fact, presheaves on $\mathcal{A}$ correspond
one-to-one with lower sets of $\mathcal{A}$.
Moreover, the ordering 
$\preceq_\funct{\mathcal{A}^\op}{\textbf{2}}$
on the presheaf category coincides with the 
inclusion ordering $\subseteq$ on the set of lower sets:
for all presheaves $P_1,P_2$ on $\mathcal{A}$,
\[
P_1\preceq_\funct{\mathcal{A}^\op}{\textbf{2}}P_2
\iff \ob{P_1}^{-1}\big(\{\true\}\big)\subseteq \ob{P_2}^{-1}\big(\{\true\}\big)
\]
hold.

A copresheaf $Q$ on $\mathcal{A}$ is a monotonic map
$Q\colon\mathcal{A} \longrightarrow \textbf{2}$.
If we take the inverse image of 
$\{\true\}\subseteq\textbf{2}_0$, we obtain an
\textit{upper set} of $\mathcal{A}$.
Copresheaves on $\mathcal{A}$ correspond
one-to-one with upper sets of $\mathcal{A}$; also,
the ordering 
$\preceq_{\funct{\mathcal{A}}{\textbf{2}}^\op}=\succeq_\funct{\mathcal{A}}{\textbf{2}}$
on the op-copresheaf category coincides with the opposite
$\supseteq$ of the 
inclusion ordering $\subseteq$ on the set of upper sets:
for all copresheaves $Q_1,Q_2$ on $\mathcal{A}$,
\[
Q_1\preceq_{\funct{\mathcal{A}}{\textbf{2}}^\op}Q_2
\iff \ob{Q_1}^{-1}\big(\{\true\}\big)\supseteq \ob{Q_2}^{-1}\big(\{\true\}\big)
\]
hold.

To summarize, the presheaf category
$\funct{\mathcal{A}^\op}{\textbf{2}}$ 
of $\mathcal{A}$ can naturally be
seen as the preordered set (in fact, poset) of
lower sets of $\mathcal{A}$, and the 
op-copresheaf category 
$\funct{\mathcal{A}}{\textbf{2}}^\op$
as that of upper sets of $\mathcal{A}$.
\end{ex}

\begin{ex} 
Let $\mathcal{A}=(\ob{\mathcal{A}},d_\mathcal{A})$ 
and $\mathcal{B}=(\ob{\mathcal{B}},d_\mathcal{B})$
be $\overline{\mathbb{K}}_+$-categories.
The functor category 
$\funct{\mathcal{A}}{\mathcal{B}}=([\mathcal{A},\mathcal{B}]_0,d_\funct{\mathcal{A}}{\mathcal{B}})$ 
is given as follows, noting that $\bigwedge$ in
$\mathbb{K}_+\cup\{\infty\}$ corresponds to (the usual)
$\sup$:
\begin{itemize}
	\item $[\mathcal{A},\mathcal{B}]_0$ is the set of
	all nonexpansive maps from $\mathcal{A}$ to
	$\mathcal{B}$;
	\item $d_\funct{\mathcal{A}}{\mathcal{B}}(F,G)=
	\sup_{a\in \ob{\mathcal{A}}}
	\left\{d_\mathcal{B}\big(F(a),G(a)\big)\right\}$.
\end{itemize}
Therefore, in the current setting, the natural distance
on a function space turns out to be the sup-distance.
Since we adjoined $\infty$ and turned $\mathbb{K}_+$ into
a complete lattice, we can get rid of additional 
assumptions such as the compactness of $\mathcal{A}$ or 
the boundedness of $\mathcal{B}$, the usual 
requirements when performing the sup-distance construction
in the theory of classical metric spaces.

A presheaf $P$ on $\mathcal{A}$ is a nonexpansive map
$P\colon\mathcal{A}^\op \longrightarrow \overline{\mathbb{K}}_+$,
namely, a map 
$\ob{P}\colon$ $\ob{A}\longrightarrow \mathbb{K}_+\cup\{\infty\}$
satisfying 
\[
	d_\mathcal{A}(a,b)\geq P(a)\dotminus P(b)
\]
for all $a,b\in \ob{\mathcal{A}}$,
and can be seen as a scalar-valued function on
$\mathcal{A}$.
The distance between two presheaves $P_1$ and $P_2$ is 
given by
\[
	d_\funct{\mathcal{A}^\op}{\overline{\mathbb{K}}_+}(P_1,P_2)
	=\sup_{a\in \ob{\mathcal{A}}}
	\{P_2(a)\dotminus P_1(a)\}.
\]
The canonical ordering $\Rightarrow$ is defined as 
\begin{align*}
	P_1\Rightarrow P_2 
	&\iff 0\geq P_2(a)\dotminus P_1(a)\quad (\forall a\in \ob{\mathcal{A}})\\
	&\iff P_2(a)\dotminus P_1(a)=0\quad (\forall a\in \ob{\mathcal{A}})\\
	&\iff d_\funct{\mathcal{A}^\op}{\overline{\mathbb{K}}_+}(P_1,P_2)=0\\
	&\iff P_1(a)\geq P_2(a)\quad (\forall a\in \ob{\mathcal{A}}).
\end{align*}

A copresheaf $Q$ on $\mathcal{A}$ is a nonexpansive map
$Q\colon\mathcal{A} \longrightarrow \overline{\mathbb{K}}_+$,
namely, a map 
$\ob{Q}\colon$ $\ob{A}\longrightarrow \mathbb{K}_+\cup\{\infty\}$
satisfying 
\[
	d_\mathcal{A}(a,b)\geq Q(b)\dotminus Q(a)
\]
for all $a,b\in \ob{\mathcal{A}}$,
and forms another natural class of scalar-valued
functions on $\mathcal{A}$.
The distance between two copresheaves $Q_1$ and $Q_2$ 
in the op-copresheaf category
$\funct{\mathcal{A}}{\overline{\mathbb{K}}_+}^\op$ 
is given by
\[
	d_{\funct{\mathcal{A}}{\overline{\mathbb{K}}_+}^\op}(Q_1,Q_2)
	=\sup_{a\in \ob{\mathcal{A}}}
	\{Q_1(a)\dotminus Q_2(a)\}.
\]
As $\overline{\mathbb{K}}_+$-functors, the canonical 
ordering $\Rightarrow$ on copresheaves is determined 
by conditions similar to that for presheaves.
\end{ex}

\begin{ex} 
Let $\mathcal{A}=(\ob{\mathcal{A}},d_\mathcal{A})$ 
and $\mathcal{B}=(\ob{\mathcal{B}},d_\mathcal{B})$
be $\overline{\mathbb{K}}$-categories.
The functor category 
$\funct{\mathcal{A}}{\mathcal{B}}=([\mathcal{A},\mathcal{B}]_0,d_\funct{\mathcal{A}}{\mathcal{B}})$ 
is given as follows:
\begin{itemize}
	\item $[\mathcal{A},\mathcal{B}]_0$ is the set of
	all nonexpansive maps from $\mathcal{A}$ to
	$\mathcal{B}$;
	\item $d_\funct{\mathcal{A}}{\mathcal{B}}(F,G)=
	\sup_{a\in \ob{\mathcal{A}}}
	\left\{d_\mathcal{B}\big(F(a),G(a)\big)\right\}$.
\end{itemize}

A presheaf $P$ on $\mathcal{A}$ is a nonexpansive map
$P\colon\mathcal{A}^\op \longrightarrow \overline{\mathbb{K}}$,
namely, a map 
$\ob{P}\colon$ $\ob{A}\longrightarrow \mathbb{K}\cup\{-\infty,\infty\}$
satisfying 
\[
	d_\mathcal{A}(a,b)\geq P(a)-P(b)
\]
for all $a,b\in \ob{\mathcal{A}}$.
The distance between two presheaves $P_1$ and $P_2$ is 
given by
\[
	d_\funct{\mathcal{A}^\op}{\overline{\mathbb{K}}}(P_1,P_2)
	=\sup_{a\in \ob{\mathcal{A}}}
	\{P_2(a)- P_1(a)\}.
\]
The canonical ordering $\Rightarrow$ is defined as 
\begin{align*}
	P_1\Rightarrow P_2 
	&\iff 0\geq P_2(a)- P_1(a)\quad (\forall a\in \ob{\mathcal{A}})\\
	&\iff d_\funct{\mathcal{A}^\op}{\overline{\mathbb{K}}}(P_1,P_2)\leq 0\\
	&\iff P_1(a)\geq P_2(a)\quad (\forall a\in \ob{\mathcal{A}}).
\end{align*}

A copresheaf $Q$ on $\mathcal{A}$ is a nonexpansive map
$Q\colon\mathcal{A} \longrightarrow \overline{\mathbb{K}}$,
namely, a map 
$\ob{Q}\colon$ $\ob{A}\longrightarrow \mathbb{K}\cup\{-\infty,\infty\}$
satisfying 
\[
	d_\mathcal{A}(a,b)\geq Q(b)- Q(a)
\]
for all $a,b\in \ob{\mathcal{A}}$.
The distance between two copresheaves $Q_1$ and $Q_2$ 
in the op-copresheaf category
$\funct{\mathcal{A}}{\overline{\mathbb{K}}}^\op$ is given by
\[
	d_{\funct{\mathcal{A}}{\overline{\mathbb{K}}}^\op}(Q_1,Q_2)
	=\sup_{a\in \ob{\mathcal{A}}}
	\{Q_1(a)- Q_2(a)\}.
\]
\end{ex}

\begin{ex} 
Let $\mathcal{A}=(\ob{\mathcal{A}},d_\mathcal{A})$ 
and $\mathcal{B}=(\ob{\mathcal{B}},d_\mathcal{B})$
be $\overline{\mathbb{K}}_+^\text{Cart}$-categories.
The functor category 
$\funct{\mathcal{A}}{\mathcal{B}}=([\mathcal{A},\mathcal{B}]_0,d_\funct{\mathcal{A}}{\mathcal{B}})$ 
is given as follows:
\begin{itemize}
	\item $[\mathcal{A},\mathcal{B}]_0$ is the set of
	all nonexpansive maps from $\mathcal{A}$ to
	$\mathcal{B}$;
	\item $d_\funct{\mathcal{A}}{\mathcal{B}}(F,G)=
	\sup_{a\in \ob{\mathcal{A}}}
	\left\{d_\mathcal{B}\big(F(a),G(a)\big)\right\}$.
\end{itemize}

A presheaf $P$ on $\mathcal{A}$ is a nonexpansive map
$P\colon\mathcal{A}^\op \longrightarrow \overline{\mathbb{K}}_+^\text{Cart}$,
namely, a map 
$\ob{P}\colon$ $\ob{A}\longrightarrow \mathbb{K}_+\cup\{\infty\}$
satisfying 
\begin{align*}
	d_\mathcal{A}(a,b)&\geq P(b)\supset P(a)\\
		&=\begin{cases}
		0 &(P(b)\geq P(a))\\
		P(a) &(P(b)<P(a)),
		\end{cases}
\end{align*}
or equivalently,
\[
	P(b)<P(a)\implies d_\mathcal{A}(a,b)\geq P(a)
\]
for all $a,b\in \ob{\mathcal{A}}$.
The distance between two presheaves $P_1$ and $P_2$ is 
given by
\begin{align*}
	d_\funct{\mathcal{A}^\op}{\overline{\mathbb{K}}_+^\text{Cart}}(P_1,P_2)
	&=\sup_{a\in \ob{\mathcal{A}}}
	\{P_1(a)\supset P_2(a)\}\\
	&=\sup_{\substack{a\in \ob{\mathcal{A}}\\P_1(a)<P_2(a)}}
	\{P_2(a)\}.
\end{align*}
The canonical ordering $\Rightarrow$ is defined as 
\begin{align*}
	P_1\Rightarrow P_2 
	&\iff 0\geq P_1(a)\supset P_2(a)\quad (\forall a\in \ob{\mathcal{A}})\\
	&\iff P_1(a)\supset P_2(a)=0\quad (\forall a\in \ob{\mathcal{A}})\\
	&\iff d_\funct{\mathcal{A}^\op}{\overline{\mathbb{K}}_+^\text{Cart}}(P_1,P_2)=0\\
	&\iff P_1(a)\geq P_2(a)\quad (\forall a\in \ob{\mathcal{A}}).
\end{align*}

A copresheaf $Q$ on $\mathcal{A}$ is a nonexpansive map
$Q\colon\mathcal{A} \longrightarrow \overline{\mathbb{K}}_+^\text{Cart}$,
namely, a map 
$\ob{Q}\colon$ $\ob{A}\longrightarrow \mathbb{K}_+\cup\{\infty\}$
satisfying 
\begin{align*}
	d_\mathcal{A}(a,b)&\geq Q(a)\supset Q(b)\\
	&=\begin{cases}
	0 &(Q(a)\geq Q(b))\\
	Q(b) &(Q(a)<Q(b)),
	\end{cases}
\end{align*}
or equivalently,
\[
	Q(a)<Q(b)\implies d_\mathcal{A}(a,b)\geq Q(b)
\]
for all $a,b\in \ob{\mathcal{A}}$.
The distance between two copresheaves $Q_1$ and $Q_2$ 
in the op-copresheaf category
$\funct{\mathcal{A}}{\overline{\mathbb{K}}_+^\text{Cart}}^\op$ is given by
\begin{align*}
	d_{\funct{\mathcal{A}}{\overline{\mathbb{K}}_+^\text{Cart}}^\op}(Q_1,Q_2)
	&=\sup_{a\in \ob{\mathcal{A}}}
	\{Q_2(a)\supset Q_1(a)\}\\
	&=\sup_{\substack{a\in \ob{\mathcal{A}}\\Q_2(a)<Q_1(a)}}
	\{Q_1(a)\}.
\end{align*}
\end{ex}

Now we present the Yoneda embedding theorem:
\begin{thm}
Let $\mathcal{V}=(\mathcal{V}_0,\otimes,[-,-],e)$ 
be an SMC-CL, and 
$\mathcal{A}=(\ob{\mathcal{A}},\Hom_\mathcal{A})$ be a
$\mathcal{V}$-category.
Then there are fully faithful $\mathcal{V}$-functors
$Y$ and $\overline{Y}$, defined as follows:
\begin{align*}
Y\colon \mathcal{A}\longrightarrow \funct{\mathcal{A}^\op}{\mathcal{V}},\qquad
& \ob{Y(b)}=\lambda a\in \ob{\mathcal{A}^\op}.\,\Hom_\mathcal{A}(a,b)
\quad (\forall b\in \ob{\mathcal{A}});\\
\overline{Y}\colon \mathcal{A}\longrightarrow \funct{\mathcal{A}}{\mathcal{V}}^\op,\qquad
&\ob{\overline{Y}(a)}=\lambda b\in \ob{\mathcal{A}}.\,\Hom_\mathcal{A}(a,b)
\quad (\forall a\in \ob{\mathcal{A}^\op}).
\end{align*}
\end{thm} 
\begin{proof}
We first show that the values $Y(b)$ and $\overline{Y}(a)$
are indeed a presheaf or a copresheaf, respectively.
\begin{description}[font=\normalfont]
\item[[$Y(b)$ is a presheaf on $\mathcal{A}\!\!\!$]]
	The claim is 
	\begin{align*}
		\Hom_{\mathcal{A}}(a,a^\prime) &\sqsubseteq
		[Y(b)(a^\prime),Y(b)(a)] \\
		&= [\Hom_\mathcal{A}(a^\prime,b),\Hom_\mathcal{A}(a,b)],
	\end{align*}
	and this is equivalent to an instance of the 
	composition law for $\mathcal{A}$:
	\begin{prooftree}
		\def\fCenter{\ \sqsubseteq\ }
		\alwaysDoubleLine
		\Axiom$\Hom_{\mathcal{A}}(a,a^\prime)\fCenter [\Hom_\mathcal{A}(a^\prime,b),\Hom_\mathcal{A}(a,b)]$
		\UnaryInf$\Hom_{\mathcal{A}}(a,a^\prime)\otimes  \Hom_\mathcal{A}(a^\prime,b)\fCenter \Hom_\mathcal{A}(a,b)$
	\end{prooftree}
\item[[$\overline{Y}(a)$ is a copresheaf on $\mathcal{A}\!\!\!$]]
	The claim is 
	\begin{align*}
		\Hom_{\mathcal{A}}(b,b^\prime) &\sqsubseteq
		[\overline{Y}(a)(b),\overline{Y}(a)(b^\prime)] \\
		&= [\Hom_\mathcal{A}(a,b),\Hom_\mathcal{A}(a,b^\prime)],
	\end{align*}
	and this is equivalent to an instance of the 
	composition law for $\mathcal{A}$:
	\begin{prooftree}
		\def\fCenter{\ \sqsubseteq\ }
		\alwaysDoubleLine
		\Axiom$\Hom_{\mathcal{A}}(b,b^\prime)\fCenter [\Hom_\mathcal{A}(a,b),\Hom_\mathcal{A}(a,b^\prime)]$
		\UnaryInf$\Hom_{\mathcal{A}}(b,b^\prime)\otimes  \Hom_\mathcal{A}(a,b)\fCenter \Hom_\mathcal{A}(a,b^\prime)$
		\UnaryInf$\Hom_{\mathcal{A}}(a,b)\otimes  \Hom_\mathcal{A}(b,b^\prime)\fCenter \Hom_\mathcal{A}(a,b^\prime)$
	\end{prooftree}
\end{description}
Next we show that $Y$ and $\overline{Y}$ are 
fully faithful $\mathcal{V}$-functors:
\begin{description}[font=\normalfont]
\item[[$Y$ is a fully faithful $\mathcal{V}$-functor\!\!\!]]
	The claim is 
	\begin{align*}
		\Hom_{\mathcal{A}}(b,b^\prime) &=
		\Hom_\funct{\mathcal{A}^\op}{\mathcal{V}}(Y(b),Y(b^\prime)) \\
		&= \bigwedge_{a\in \ob{\mathcal{A}}}
		\left\{[Y(b)(a),Y(b^\prime)(a)]\right\} \\
		&= \bigwedge_{a\in \ob{\mathcal{A}}}
		\left\{[\Hom_\mathcal{A}(a,b),\Hom_\mathcal{A}(a,b^\prime)]\right\}.
	\end{align*}
	We prove this by showing the two inequalities:
	\begin{description}[font=\normalfont]
	\item[[\mbox{$\Hom_{\mathcal{A}}(b,b^\prime)
	\sqsubseteq \bigwedge_{a\in \ob{\mathcal{A}}}
	\left\{[\Hom_\mathcal{A}(a,b),\Hom_\mathcal{A}(a,b^\prime)]\right\}$}\!\!\!]]
	It is equivalent to 
	\[
	\Hom_{\mathcal{A}}(b,b^\prime)
	\sqsubseteq [\Hom_\mathcal{A}(a,b),\Hom_\mathcal{A}(a,b^\prime)]\quad(\forall a\in \ob{\mathcal{A}})
	\]
	and its proof for some fixed $a\in \ob{\mathcal{A}}$
	is already given above when we prove that
	$\overline{Y}(a)$ is a copresheaf on $\mathcal{A}$.
	\item[[\mbox{$\bigwedge_{a\in \ob{\mathcal{A}}}
	\left\{[\Hom_\mathcal{A}(a,b),\Hom_\mathcal{A}(a,b^\prime)]\right\}
	\sqsubseteq \Hom_{\mathcal{A}}(b,b^\prime)$}\!\!\!]]
	First note that
	\[
	\bigwedge_{a\in \ob{\mathcal{A}}}
	\left\{[\Hom_\mathcal{A}(a,b),\Hom_\mathcal{A}(a,b^\prime)]\right\}\sqsubseteq
	[\Hom_\mathcal{A}(b,b),\Hom_\mathcal{A}(b,b^\prime)]
	\]
	holds.
	Using an instance of the identity law for $\mathcal{A}$,
	$e\sqsubseteq \Hom_\mathcal{A}(b,b)$,
	we have the following:
	\begin{align*}
	[\Hom_\mathcal{A}(b,b),\Hom_\mathcal{A}(b,b^\prime)]
	&\sqsubseteq [e,\Hom_\mathcal{A}(b,b^\prime)]\\
	&= \Hom_\mathcal{A}(b,b^\prime).
	\end{align*}
	Therefore the claim follows.
	\end{description}
\item[[$\overline{Y}$ is a fully faithful $\mathcal{V}$-functor\!\!\!]]
	The claim is 
	\begin{align*}
		\Hom_{\mathcal{A}}(a,a^\prime) &=
		\Hom_\funct{\mathcal{A}}{\mathcal{V}}(\overline{Y}(a^\prime),\overline{Y}(a)) \\
		&= \bigwedge_{b\in \ob{\mathcal{A}}}
		\left\{[\overline{Y}(a^\prime)(b),\overline{Y}(a)(b)]\right\} \\
		&= \bigwedge_{b\in \ob{\mathcal{A}}}
		\left\{[\Hom_\mathcal{A}(a^\prime,b),\Hom_\mathcal{A}(a,b)]\right\}.
	\end{align*}
	We prove this by showing the two inequalities:
	\begin{description}[font=\normalfont]
	\item[[\mbox{$\Hom_{\mathcal{A}}(a,a^\prime)
	\sqsubseteq \bigwedge_{b\in \ob{\mathcal{A}}}
	\left\{[\Hom_\mathcal{A}(a^\prime,b),\Hom_\mathcal{A}(a,b)]\right\}$}\!\!\!]]
	It is equivalent to 
	\[
	\Hom_{\mathcal{A}}(a,a^\prime)
	\sqsubseteq [\Hom_\mathcal{A}(a^\prime,b),\Hom_\mathcal{A}(a,b)]\quad(\forall b\in \ob{\mathcal{A}})
	\]
	and its proof for some fixed $b\in \ob{\mathcal{A}}$
	is already given above when we prove that
	$Y(b)$ is a presheaf on $\mathcal{A}$.
	\item[[\mbox{$\bigwedge_{b\in \ob{\mathcal{A}}}
	\left\{[\Hom_\mathcal{A}(a^\prime,b),\Hom_\mathcal{A}(a,b)]\right\}
	\sqsubseteq \Hom_{\mathcal{A}}(a,a^\prime)$}\!\!\!]]
	First note that
	\[
	\bigwedge_{b\in \ob{\mathcal{A}}}
	\left\{[\Hom_\mathcal{A}(a^\prime,b),\Hom_\mathcal{A}(a,b)]\right\}\sqsubseteq
	[\Hom_\mathcal{A}(a^\prime,a^\prime),\Hom_\mathcal{A}(a,a^\prime)]
	\]
	holds.
	Using an instance of the identity law for $\mathcal{A}$,
	$e\sqsubseteq \Hom_\mathcal{A}(a^\prime,a^\prime)$,
	we have the following:
	\begin{align*}
	[\Hom_\mathcal{A}(a^\prime,a^\prime),\Hom_\mathcal{A}(a,a^\prime)]
	&\sqsubseteq [e,\Hom_\mathcal{A}(a,a^\prime)]\\
	&= \Hom_\mathcal{A}(a,a^\prime).
	\end{align*}
	Therefore the claim follows.
	\end{description}
\end{description}
\end{proof}
\begin{defn}
The $\mathcal{V}$-functor $Y$ is called the 
\textbf{Yoneda embedding}, and 
$\overline{Y}$ the \textbf{co-Yoneda embedding}.
\end{defn}
We see some examples of the (co-)Yoneda embeddings.

\begin{ex} 
Let $\mathcal{A}=(\ob{\mathcal{A}},\preceq_\mathcal{A})$ 
be a $\textbf{2}$-category. 
Under the identification of presheaves on $\mathcal{A}$
and lower sets of $\mathcal{A}$,
the Yoneda embedding turns out to be the well-known
\textit{principal ideal} operation 
$\mathord{\downarrow}(-)$:
\begin{align*}
\ob{Y(b)}^{-1}\big(\{\true\}\big)&=\mathord{\downarrow}(b)\\
&=\{a\in\ob{\mathcal{A}}\mid a\preceq_\mathcal{A} b \}.
\end{align*}
Dually, the co-Yoneda embedding corresponds to the
\textit{principal filter} operation $\mathord{\uparrow}(-)$:
\begin{align*}
\ob{\overline{Y}(a)}^{-1}\big(\{\true\}\big)&=\mathord{\uparrow} (a)\\
&=\{b\in\ob{\mathcal{A}}\mid a\preceq_\mathcal{A} b \}.
\end{align*}

Therefore, the Yoneda embedding theorem specializes to
the statement that 
\textit{every preordered set embeds into the poset of 
its lower subsets (ordered by the inclusion) and 
the poset of its upper subsets (ordered by the 
opposite of the inclusion)}:
\begin{align*}
b\preceq_\mathcal{A} b^\prime 
&\iff \mathord{\downarrow}(b)\subseteq\mathord{\downarrow} (b^\prime);\\
a\preceq_\mathcal{A} a^\prime 
&\iff \mathord{\uparrow}(a)\supseteq\mathord{\uparrow} (a^\prime).
\end{align*}
But note that such \textit{embeddings} are not 
injective unless the preorder is an order.
\end{ex}

\begin{ex} 
Let $\mathcal{A}=(\ob{\mathcal{A}},d_\mathcal{A})$ 
be a $\overline{\mathbb{K}}_+$-category. 
The Yoneda embedding $Y$ sends a point $b$ of 
$\mathcal{A}$ to the ``distance to $b$'' function:
\[
	\ob{Y(b)}=\lambda a\in \ob{\mathcal{A}^\op}.\,d_\mathcal{A}(a,b).
\]
The co-Yoneda embedding $\overline{Y}$ sends a 
point $a$ of $\mathcal{A}$ to the ``distance from $a$''
function:
\[
	\ob{\overline{Y}(a)}=\lambda b\in \ob{\mathcal{A}}.\,d_\mathcal{A}(a,b).
\]

The Yoneda embedding theorem reads:
\begin{align*}
d_\mathcal{A}(b,b^\prime)
&=d_\funct{\mathcal{A}^\op}{\overline{\mathbb{K}}_+}
(Y(b),Y(b^\prime));\\
d_\mathcal{A}(a,a^\prime)
&=d_{\funct{\mathcal{A}}{\overline{\mathbb{K}}_+}^\op}
(\overline{Y}(a),\overline{Y}(a^\prime)).
\end{align*}
If $\mathcal{A}$ is a classical metric space, 
then by symmetry $Y$ and $\overline{Y}$ coincide and 
turn out to be the so-called 
\textit{Fr\'{e}chet embedding}.
\end{ex}

\begin{ex} 
Let $\mathcal{A}=(\ob{\mathcal{A}},d_\mathcal{A})$ 
be a $\overline{\mathbb{K}}$-category. We have
\[
	\ob{Y(b)}=\lambda a\in \ob{\mathcal{A}^\op}.\,d_\mathcal{A}(a,b),\quad
	\ob{\overline{Y}(a)}=\lambda b\in \ob{\mathcal{A}}.\,d_\mathcal{A}(a,b)
\]
and 
\[
d_\mathcal{A}(b,b^\prime)
=d_\funct{\mathcal{A}^\op}{\overline{\mathbb{K}}}
(Y(b),Y(b^\prime)),\quad
d_\mathcal{A}(a,a^\prime)
=d_{\funct{\mathcal{A}}{\overline{\mathbb{K}}}^\op}
(\overline{Y}(a),\overline{Y}(a^\prime)).
\]
\end{ex}

\begin{ex} 
Let $\mathcal{A}=(\ob{\mathcal{A}},d_\mathcal{A})$ 
be a $\overline{\mathbb{K}}_+^\text{Cart}$-category. 
We have
\[
	\ob{Y(b)}=\lambda a\in \ob{\mathcal{A}^\op}.\,d_\mathcal{A}(a,b),\quad
	\ob{\overline{Y}(a)}=\lambda b\in \ob{\mathcal{A}}.\,d_\mathcal{A}(a,b)
\]
and 
\[
d_\mathcal{A}(b,b^\prime)
=d_\funct{\mathcal{A}^\op}{\overline{\mathbb{K}}_+^\text{Cart}}
(Y(b),Y(b^\prime)),\quad
d_\mathcal{A}(a,a^\prime)
=d_{\funct{\mathcal{A}}{\overline{\mathbb{K}}_+^\text{Cart}}^\op}
(\overline{Y}(a),\overline{Y}(a^\prime)).
\]
\end{ex}

\chapter[$\overline{\mathbb{K}}$-Categories and $\overline{\mathbb{K}}$-Extended L-Convex Sets]{$\overline{\mathbb{K}}$-Categories and \newline $\overline{\mathbb{K}}$-Extended L-Convex Sets}
We are now ready to show the main result, 
a correspondence between 
$\overline{\mathbb{K}}$-categories and 
$\overline{\mathbb{K}}$-extended L-convex sets,
which are variants of L-convex sets or L-convex 
polyhedra of discrete convex analysis.
A chief difference between $\overline{\mathbb{K}}$-extended 
L-convex sets and L-convex sets (corresponds to the case 
$\mathbb{K}=\mathbb{Z}$) or L-convex polyhedra
($\mathbb{K}=\mathbb{R}$) is that the former have 
the ambient sets of the form $\overline{\mathbb{K}}^V$  
whereas the latter two $\mathbb{Z}^V$ or $\mathbb{R}^V$
($V$ is some set).
In this chapter, we introduce 
$\overline{\mathbb{K}}$-extended L-convex sets and
homomorphisms between them, and exploit these to
show a duality theorem between 
$\overline{\mathbb{K}}$-categories and 
$\overline{\mathbb{K}}$-extended L-convex sets.
The theorem rests on the fact that the set
$\mathbb{K}\cup\{-\infty,\infty\}$ can canonically be
seen both as
a $\overline{\mathbb{K}}$-category and as
a $\overline{\mathbb{K}}$-extended L-convex set,  
and has two faces.
Using this fact, both directions of the constructions
in the duality 
(namely, from
$\overline{\mathbb{K}}$-categories to  
$\overline{\mathbb{K}}$-extended L-convex sets,
and conversely) 
are realized as that of forming the function spaces 
with codomain $\overline{\mathbb{K}}$ (either as a 
$\overline{\mathbb{K}}$-category or as a 
$\overline{\mathbb{K}}$-extended L-convex set).
As a byproduct, we remark that (the underlying sets of)
$\overline{\mathbb{K}}$-extended L-convex sets,
seen as $\overline{\mathbb{K}}$-categories with 
distance functions the restrictions of that for
$\overline{\mathbb{K}}^V$ with sup-distances,
are nothing but $\overline{\mathbb{K}}$-presheaf 
categories.
This result gives rise to a claim that (variants of)
\textit{L-convex sets arise canonically 
in the enriched-categorical setting},
suggesting the possibility of a categorical approach
to discrete convex analysis.

\section{$\overline{\mathbb{K}}$-Extended L-Convex Sets}
We begin with the definition of 
$\overline{\mathbb{K}}$-extended L-convex sets
(recall that $\mathbb{K}$ denotes either $\mathbb{Z}$ 
or $\mathbb{R}$).

\begin{defn} 
A $\overline{\mathbb{K}}$\textbf{-extended L-convex set}
$\lcs{D}$ is a pair 
$(\ind{\lcs{D}},\lcs{D}_0)$ where
\begin{itemize}
\item 
$\ind{\lcs{D}}$ is a set called the \textbf{index set} 
of $\lcs{D}$;
\item
$\lcs{D}_0$ is a subset of 
$\overline{\mathbb{K}}^\ind{\lcs{D}}$
called the \textbf{underlying set} of $\lcs{D}$;
\end{itemize}
such that the following axioms hold:
\begin{description}[font=\normalfont]
	\item[$\qquad$(Order completeness)]
	$\bigvee_{p \in S} p,\ 
	\bigwedge_{p \in S} p\in \lcs{D}_0
	\quad(\forall S\subseteq \lcs{D}_0)$;
	\item[$\qquad$(Weight completeness)]
	$p+\alpha \cdot\textbf{1}\in \lcs{D}_0\quad 
	(\forall p\in \lcs{D}_0,\ \forall \alpha \in \mathbb{K}\cup\{-\infty,\infty\})$.
\end{description}
\end{defn}
Let us clarify the notations.
$\overline{\mathbb{K}}^\ind{\lcs{D}}$ denotes the 
set of all maps from the set $\ind{\lcs{D}}$ to 
$\mathbb{K}\cup\{-\infty,\infty\}$.
The supremum $\bigvee$ and the infimum $\bigwedge$ in 
$\overline{\mathbb{K}}^\ind{\lcs{D}}$ are
taken coordinate-wise:
\[
\big(\bigvee_{p\in S}p\big)(v)=\bigvee_{p\in S} \big(p(v)\big),\quad
\big(\bigwedge_{p\in S}p\big)(v)=\bigwedge_{p\in S} \big(p(v)\big)\quad
(\forall v\in \ind{\lcs{D}}),
\]
where the supremum and the infimum in the right 
hand sides 
are taken in the set $\mathbb{K}\cup\{-\infty,\infty\}$ 
using the order $\geq$ (recall that since 
$\geq$ is the opposite of the usual ordering $\leq$, 
the supremum $\bigvee$ here corresponds
to the usual $\inf$, and the infimum $\bigwedge$ 
the usual $\sup$; we continue to use our previous notational 
conventions $\bigvee=\inf$, $\bigwedge=\sup$).
$\alpha\cdot \textbf{1}$ represents an element of 
$\overline{\mathbb{K}}^\ind{\lcs{D}}$ which is the
``all-$\alpha$ vector'' or the constant map with the unique
value $\alpha$:
\[
(\alpha \cdot\textbf{1})(v)=\alpha\quad (\forall v\in \ind{\lcs{D}}).
\]
The operation $+$  on 
$\overline{\mathbb{K}}^\ind{\lcs{D}}$
is defined coordinate-wise using the
extended addition of Example~\ref{ex:overlineK}:
\[
(p+q)(v)=p(v)+q(v)\quad (\forall v\in \ind{\lcs{D}}).
\]
One can define the operation $-$ on
$\overline{\mathbb{K}}^\ind{\lcs{D}}$
in the similar way.
The name ``weight completeness'' has its origin in
the fact that together with the order completeness, 
it assures the existence of all \textit{weighted limits}
and \textit{weighted colimits}, which are fundamental 
notions of enriched category theory.

First we check that the $-$ version of the weight 
completeness condition comes for free. 
\begin{prop}
Let $\lcs{D}=(\ind{\lcs{D}},\lcs{D}_0)$ be a
$\overline{\mathbb{K}}$-extended L-convex set.
Then the following hold:
\begin{align}\label{eq:weightminus}
	p-\alpha \cdot\mathbf{1}\in \lcs{D}_0\quad 
	(\forall p\in \lcs{D}_0,\ \forall \alpha \in \mathbb{K}\cup\{-\infty,\infty\}).
\end{align}
\end{prop}
\begin{proof}
When $\alpha \in \mathbb{K}$, then clearly 
\[
p-\alpha \cdot\mathbf{1}=p+(-\alpha)\cdot \mathbf{1}.
\]
If $\alpha =-\infty$, then 
\[
(p-(-\infty)\cdot\mathbf{1})(v)=
\begin{cases}
	-\infty &(p(v)=-\infty)\\
	\infty &(\text{otherwise})
\end{cases}
\quad (\forall v\in \ind{\lcs{D}}),
\]
and this is realized as 
\[
\bigwedge_{\alpha\in\mathbb{K}}(p+\alpha\cdot\mathbf{1})=\sup_{\alpha\in\mathbb{K}}\{p+\alpha\cdot\mathbf{1}\}.
\]
Finally, if $\alpha =\infty$, then 
\[
(p-\infty\cdot\mathbf{1})(v)=-\infty\quad (\forall v\in \ind{\lcs{D}})
\]
and this is written as 
\[
\bigwedge \emptyset=\sup\emptyset.
\]
Hence, in either case 
$(p-\alpha\cdot\mathbf{1})\in\lcs{D}_0$ by the order and 
weight completeness of $\lcs{D}$.
\end{proof}

In what follows, we refer to the condition
(\ref{eq:weightminus}) also as the weight 
completeness of $\lcs{D}$.
The definition of $\overline{\mathbb{K}}$-extended
L-convex sets slightly differs from that of L-convex
sets (when $\mathbb{K}=\mathbb{Z}$) or
L-convex polyhedra (when $\mathbb{K}=\mathbb{R}$).
For comparison, below we present the definitions of
them.

\begin{defn}[\cite{Mur03}] 
Let $V$ be a \textit{finite} set.
A subset $D\subseteq \mathbb{Z}^V$ is an
\textbf{L-convex set} if the following hold:
\begin{description}[font=\normalfont]
	\item[$\qquad$(Nonemptiness)]
	$D\neq \emptyset$;
	\item[$\qquad$(SBS\mbox{$[\mathbb{Z}]$})]
	$p\land q,\ p\lor q\in D
	\quad(\forall p,q\in D)$;
	\item[$\qquad$(TRS\mbox{$[\mathbb{Z}]$})]
	$p+\alpha \cdot\textbf{1}\in D\quad (\forall p\in D,\ \forall \alpha \in \mathbb{Z})$.
\end{description}
\end{defn}

\begin{defn}[\cite{Mur03}] 
Let $V$ be a \textit{finite} set.
A subset $D\subseteq \mathbb{R}^V$ is an
\textbf{L-convex polyhedron} if the following hold:
\begin{description}[font=\normalfont]
	\item[$\qquad$(Nonemptiness)]
	$D\neq \emptyset$;
	\item[$\qquad$(Closedness)]
	$D$ is a closed subset of $\mathbb{R}^V$;
	\item[$\qquad$(SBS\mbox{$[\mathbb{R}]$})]
	$p\land q,\ p\lor q\in D
	\quad(\forall p,q\in D)$;
	\item[$\qquad$(TRS\mbox{$[\mathbb{R}]$})]
	$p+\alpha \cdot\textbf{1}\in D\quad (\forall p\in D,\ \forall \alpha \in \mathbb{R})$.
\end{description}
\end{defn}

We note that the order completeness condition 
entails the nonemptiness condition;
just take $S=\emptyset$.
Since we impose on $\overline{\mathbb{K}}$-extended 
L-convex sets the 
order completeness condition stronger than
its binary versions 
(SBS[$\mathbb{Z}$] or SBS[$\mathbb{R}$]),
the index sets are no longer limited to be finite.
However, we adopt the order completeness 
condition not just because we can treat infinite
index sets; for our purpose, the binary versions 
do not work even when $\ind{\lcs{D}}$ is finite.

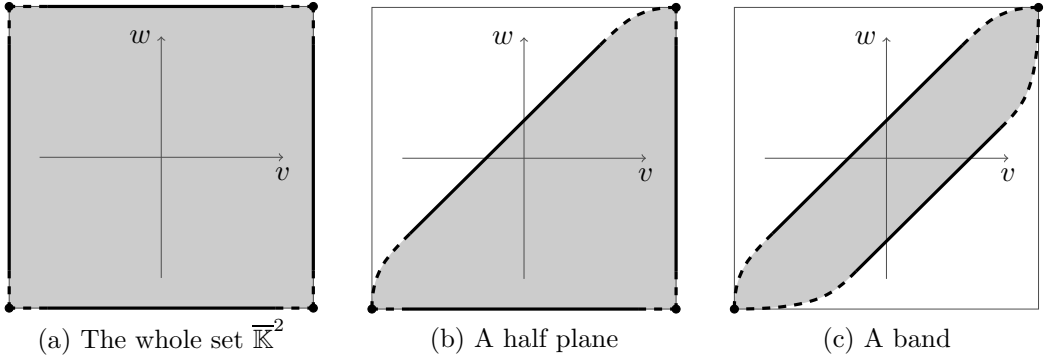
\begin{figure}
\centering
\begin{subfigure}{0.32\textwidth}
\centering
\begin{tikzpicture}
\useasboundingbox (-2,-1.935) rectangle (2,2);
\draw [my grid] (-2,-2) rectangle (2,2);
\draw [my grid,my coordarrow] (-1.6,0) -- (1.6,0) node[anchor=north,my cood] {$v$};
\draw [my grid,my coordarrow] (0,-1.6) -- (0,1.6) node[anchor=east,my cood] {$w$};
\draw [my frame,dashed] 
      (-1.5,-2) -- (-2,-2);
\draw [my frame]
      (-1.5,-2) -- (1.5,-2);
\draw [my frame,dashed]
      (1.5,-2) -- (2,-2); %
\draw [my frame,dashed] 
      (2,-1.5) -- (2,-2);
\draw [my frame]
      (2,-1.5) -- (2,1.5);
\draw [my frame,dashed]
      (2,1.5) -- (2,2); %
\draw [my frame,dashed] 
      (-1.5,2) -- (-2,2);
\draw [my frame]
      (-1.5,2) -- (1.5,2);
\draw [my frame,dashed]
      (1.5,2) -- (2,2); %
\draw [my frame,dashed] 
      (-2,-1.5) -- (-2,-2);
\draw [my frame]
      (-2,-1.5) -- (-2,1.5);
\draw [my frame,dashed]
      (-2,1.5) -- (-2,2); %
\begin{scope}[on background layer]
\fill [black!20!white] 
      (-2,-2) -- (2,-2) -- (2,2) -- (-2,2) -- cycle;
\end{scope}
\filldraw [black] (-2,-2) circle (\csize)
          (2,2) circle (\csize)
          (2,-2) circle (\csize)
          (-2,2) circle (\csize);
\end{tikzpicture}
\caption{The whole set $\overline{\mathbb{K}}^2$}
\label{fig:a}
\end{subfigure}
\begin{subfigure}{0.32\textwidth}
\centering
\begin{tikzpicture}
\useasboundingbox (-2,-2) rectangle (2,2);
\draw [my grid] (-2,-2) rectangle (2,2);
\draw [my grid,my coordarrow] (-1.6,0) -- (1.6,0) node[anchor=north,my cood] {$v$};
\draw [my grid,my coordarrow] (0,-1.6) -- (0,1.6) node[anchor=east,my cood] {$w$};
\newcommand{\Da}{1.0}
\draw [my frame,dashed] (-1.5,-\Da) .. 
      controls (-1.8,-\Da-0.3) 
      and (-2,-\Da-0.5) .. 
      (-2,-2);
\draw [my frame] (-1.5,-\Da) -- (\Da,1.5);
\draw [my frame,dashed] (\Da,1.5) .. 
      controls (\Da+0.3,1.8) 
      and (\Da+0.5,2) .. 
      (2,2);
\draw [my frame,dashed] 
      (-1.5,-2) -- (-2,-2);
\draw [my frame]
      (-1.5,-2) -- (1.5,-2);
\draw [my frame,dashed]
      (1.5,-2) -- (2,-2);
\draw [my frame,dashed] 
      (2,-1.5) -- (2,-2);
\draw [my frame]
      (2,-1.5) -- (2,1.5);
\draw [my frame,dashed]
      (2,1.5) -- (2,2);
\begin{scope}[on background layer]
\fill [black!20!white] 
      (-2,-2) ..controls(-2,-\Da-0.5)and(-1.8,-\Da-0.3).. 
      (-1.5,-\Da) -- 
      (\Da,1.5) ..controls(\Da+0.3,1.8)and(\Da+0.5,2).. 
      (2,2) --
      (2,-2) --
      (-2,-2);
\end{scope}
\filldraw [black] (-2,-2) circle (\csize)
          (2,2) circle (\csize)
          (2,-2) circle (\csize);
\end{tikzpicture}
\caption{A half plane}
\label{fig:b}
\end{subfigure}
\begin{subfigure}{0.32\textwidth}
\centering
\begin{tikzpicture}
\useasboundingbox (-2,-2) rectangle (2,2);
\draw [my grid] (-2,-2) rectangle (2,2);
\draw [my grid,my coordarrow] (-1.6,0) -- (1.6,0) node[anchor=north,my cood] {$v$};
\draw [my grid,my coordarrow] (0,-1.6) -- (0,1.6) node[anchor=east,my cood] {$w$};
\newcommand{\Da}{1.0}
\draw [my frame,dashed] (-1.5,-\Da) .. 
      controls (-1.8,-\Da-0.3) 
      and (-2,-\Da-0.5) .. 
      (-2,-2);
\draw [my frame] (-1.5,-\Da) -- (\Da,1.5);
\draw [my frame,dashed] (\Da,1.5) .. 
      controls (\Da+0.3,1.8) 
      and (\Da+0.5,2) .. 
      (2,2);
\newcommand{\Ea}{0.4}
\draw [my frame,dashed] (-\Ea,-1.5) ..
      controls (-\Ea-0.3,-1.8)
      and (-\Ea-0.5,-2) ..
      (-2,-2);
\draw [my frame] (-\Ea,-1.5) -- (1.5,\Ea);
\draw [my frame,dashed] (1.5,\Ea) ..
      controls (1.8,\Ea+0.3)
      and (2,\Ea+0.5) ..
      (2,2);
\begin{scope}[on background layer]
\fill [black!20!white] 
      (-2,-2) ..controls(-2,-\Da-0.5)and(-1.8,-\Da-0.3).. 
      (-1.5,-\Da) -- 
      (\Da,1.5) ..controls(\Da+0.3,1.8)and(\Da+0.5,2).. 
      (2,2) ..controls(2,\Ea+0.5)and(1.8,\Ea+0.3)..
      (1.5,\Ea) --
      (-\Ea,-1.5) ..controls(-\Ea-0.3,-1.8)and(-\Ea-0.5,-2)..
      (-2,-2);
\end{scope}
\filldraw [black] (-2,-2) circle (\csize)
          (2,2) circle (\csize);
\end{tikzpicture}
\caption{A band}
\label{fig:c}
\end{subfigure}
\caption{The underlying sets of  $\overline{\mathbb{K}}$-extended L-convex sets}
\label{fig:lcs}
\end{figure}

Let us illustrate the point by observing examples
of $\overline{\mathbb{K}}$-extended L-convex sets
and a non-example we wish to exclude.
Fig.~\ref{fig:lcs} presents examples of 
(the underlying sets of)
typical $\overline{\mathbb{K}}$-extended 
L-convex sets with two-element index sets.
We first explain the intended meaning of the figures
using Fig.~\ref{fig:c};
let $\lcs{D}$ be the $\overline{\mathbb{K}}$-extended 
L-convex set drawn on it.
Since the ambient set of $\lcs{D}_0$ is  
$\overline{\mathbb{K}}^\ind{\lcs{D}}\cong \overline{\mathbb{K}}^2$,
there are points whose coordinates include $-\infty$
or $\infty$;
the outer frame of the figure represents these points.
The subset $\lcs{D}_0$ is represented by the gray region
(when $\mathbb{K}=\mathbb{Z}$, the region 
should be interpreted as a set of lattice points).
The boundaries of $\lcs{D}_0$ is two
``lines of gradient 1''; 
although they are straight, 
they pass the points $(-\infty,-\infty)$ 
(at the lower left) and $(\infty,\infty)$
(the upper right), and to indicate this fact
we bend the lines in the margin, where
they are also dashed. 
We place two black dots at the lower left and the upper 
right to indicate 
$(-\infty,-\infty)\in \lcs{D}_0$ and 
$(\infty,\infty)\in \lcs{D}_0$, respectively.
In Fig.~\ref{fig:b}, parts of the frame, namely
the bottom and right lines, are drawn 
boldly.
This represents that points of the forms respectively
$(s,-\infty)$ and $(\infty,t)$ ($s,t\in\mathbb{K}$), are 
in the underlying set.
Later we show that $\overline{\mathbb{K}}$-extended 
L-convex sets with two-element index sets correspond to 
$\overline{\mathbb{K}}$-categories with two points,
and as an illustration present the correspondence 
at the level of shapes, with a more exhaustive 
enumeration.

An example of subsets of $\overline{\mathbb{K}}^2$
which we want to exclude from the underlying sets of 
$\overline{\mathbb{K}}$-extended L-convex sets
with two-element index sets is drawn in
Fig.~\ref{fig:nonlcs}.
This set satisfies the weight completeness condition and 
the binary version of the order completeness condition.
These weaker conditions fail to assure e.g.,
if $(0,0),(1,0),(2,0),\dots $ are in the set 
then so is $(\infty,0)$, and for such entities we cannot
establish a nice connection with
$\overline{\mathbb{K}}$-categories. 

\begin{figure}
\centering
\begin{tikzpicture}
\useasboundingbox (-2,-2) rectangle (2,2);
\draw [my grid] (-2,-2) rectangle (2,2);
\draw [my grid,my coordarrow] (-1.6,0) -- (1.6,0) node[anchor=north,my cood] {$v$};
\draw [my grid,my coordarrow] (0,-1.6) -- (0,1.6) node[anchor=east,my cood] {$w$};
\newcommand{\Da}{1.0}
\draw [my frame,dashed] (-1.5,-\Da) .. 
      controls (-1.8,-\Da-0.3) 
      and (-2,-\Da-0.5) .. 
      (-2,-2);
\draw [my frame] (-1.5,-\Da) -- (\Da,1.5);
\draw [my frame,dashed] (\Da,1.5) .. 
      controls (\Da+0.3,1.8) 
      and (\Da+0.5,2) .. 
      (2,2);
\begin{scope}[on background layer]
\fill [black!20!white] 
      (-2,-2) ..controls(-2,-\Da-0.5)and(-1.8,-\Da-0.3).. 
      (-1.5,-\Da) -- 
      (\Da,1.5) ..controls(\Da+0.3,1.8)and(\Da+0.5,2).. 
      (2,2) --
      (2,-2) --
      (-2,-2);
\draw [white,very thick] (-2,-2) -- (2,-2) -- (2,2);
\end{scope}
\filldraw [black] (-2,-2) circle (\csize)
          (2,2) circle (\csize);
\end{tikzpicture}
\caption{Fig.~\ref{fig:b} without the bottom and right lines}
\label{fig:nonlcs}
\end{figure}
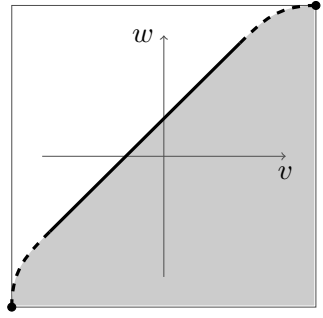

Another, more concrete example of 
$\overline{\mathbb{K}}$-extended
L-convex sets is the following.
It will play a crucial role in the formulation of 
the duality theorems.

\begin{ex}
The pair 
$\overline{\mathbb{K}}=(\{\ast\},\mathbb{K}\cup\{-\infty,\infty\})$, 
where $\{\ast\}$ is a singleton and  
$\mathbb{K}\cup\{-\infty,\infty\}$ is a 
(non-proper) subset of 
$\overline{\mathbb{K}}^{\{\ast\}}$, is clearly a 
$\overline{\mathbb{K}}$-extended L-convex set.
\end{ex}

\section{Homomorphisms of $\overline{\mathbb{K}}$-Extended L-Convex Sets}
\begin{defn}
Let $\lcs{D}=(\ind{\lcs{D}},\lcs{D}_0)$
and $\lcs{E}=(\ind{\lcs{E}},\lcs{E}_0)$ be 
$\overline{\mathbb{K}}$-extended L-convex sets.
A \textbf{homomorphism}
\[
\Phi\colon \lcs{D}\longrightarrow \lcs{E}
\] 
from $\lcs{D}$ to $\lcs{E}$ is a map
\[
\ind{\Phi}\colon \ind{\lcs{E}}\longrightarrow \ind{\lcs{D}}
\]
from the index set of $\lcs{E}$ to 
that of $\lcs{D}$, such that
for all $p\in \lcs{D}_0$, the element 
$\Phi_0(p)\in \overline{\mathbb{K}}^\ind{\lcs{E}}$ 
given by
\[
\Phi_0(p)=p\circ \ind{\Phi}\colon \ind{\lcs{E}}\overset{\ind{\Phi}}{\longrightarrow} \ind{\lcs{D}}\overset{p}{\longrightarrow} \mathbb{K}\cup \{-\infty,\infty\}
\]
is an element of $\lcs{E}_0$.

$\lcs{D}$ is called the \textbf{domain} of $\Phi$ and
$\lcs{E}$ the \textbf{codomain} of $\Phi$. 
We denote the set of all homomorphisms
from $\lcs{D}$ to $\lcs{E}$ by 
$\db{\lcs{D},\lcs{E}}_0$.
\end{defn} 

The condition for a map 
$\ind{\Phi}\colon \ind{\lcs{E}}\longrightarrow \ind{\lcs{D}}$
to define a homomorphism of 
$\overline{\mathbb{K}}$-extended
L-convex sets assures that we can always construct
a map (called the \textbf{underlying map} of $\Phi$)
\[
\Phi_0\colon \lcs{D}_0\longrightarrow \lcs{E}_0
\]
from the underlying set of $\lcs{D}$ to that 
of $\lcs{E}$ 
(the same direction as the homomorphism $\Phi$).
Therefore one can intuitively think of a homomorphism
as \textit{a map between the underlying sets
with a special property},
namely that all it does is to relabel the indices
in some fixed manner (determined by $\ind{\Phi}$). 
However, note that the \textit{equality} between 
homomorphisms 
$\Phi,\Psi\colon \lcs{D}\longrightarrow\lcs{E}$ 
are defined as that between the maps 
$\ind{\Phi},\ind{\Psi}\colon \ind{\lcs{E}}\longrightarrow \ind{\lcs{D}}$
of the index sets 
and not just between the underlying maps
$\Phi_0,\Psi_0\colon \lcs{D}_0\longrightarrow \lcs{E}_0$
(there do exist different homomorphisms with a 
common underlying map).
Fig.~\ref{fig:lcshom} illustrates an example of 
homomorphisms between $\overline{\mathbb{K}}$-extended
L-convex sets with two-element index sets.
In this example, $\Phi$ is given by 
$\ind{\Phi}\colon\{w,w^\prime\}\longrightarrow\{v,v^\prime\}$
with $\ind{\Phi}(w)=v$ and 
$\ind{\Phi}(w^\prime)=v^\prime$.
Correspondingly, we find that 
$\Phi_0\colon\lcs{D}_0\longrightarrow\lcs{E}_0$
embeds the ``thin band'' $\lcs{D}_0$ into 
the ``thick band'' $\lcs{E}_0$ 
(the dark region in $\lcs{E}_0$ represents the image of 
$\lcs{D}_0$ through the map $\Phi_0$).
Observe that unlike $\ind{\Phi}$,
$\ind{\Phi}^{-1}\colon\{v,v^\prime\}\longrightarrow\{w,w^\prime\}$
does not define a homomorphism, 
because $\lcs{E}_0$ does not fit in $\lcs{D}_0$.

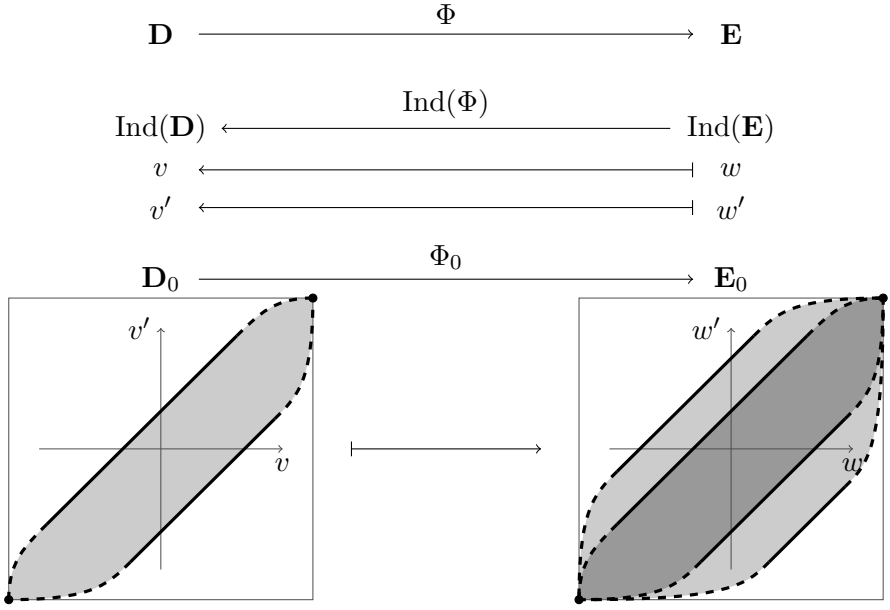
\begin{figure}
\centering
\begin{tikzpicture}
\node at (0,5.5) {$\lcs{D}$};
\node at (7.5,5.5) {$\lcs{E}$};
\draw [my arrow] (0.5,5.5) -- (7,5.5);
\node at (3.75,5.5)[anchor=south] {$\Phi$};
\node at (0,4.25) {$\ind{\lcs{D}}$};
\node at (7.5,4.25) {$\ind{\lcs{E}}$};
\node at (0,3.7) {$v$};
\node at (7.5,3.7) {$w$};
\node at (0,3.2) {$v^\prime$};
\node at (7.5,3.2) {$w^\prime$};
\draw [my mapsto] (7,3.7) -- (0.5,3.7);
\draw [my mapsto] (7,3.2) -- (0.5,3.2);
\draw [my arrow] (6.7,4.25) -- (0.8,4.25);
\node at (3.75,4.25) [anchor=south]{$\ind{\Phi}$};
\draw [my mapsto] (2.5,0) -- (5,0);
\draw [my arrow] (0.5,2.25) -- (7,2.25);
\node at (3.75,2.25)[anchor=south] {$\Phi_0$};
\node at (0,2.25) {$\lcs{D}_0$};
\node at (7.5,2.25) {$\lcs{E}_0$};

\begin{scope} 
\draw [my grid] (-2,-2) rectangle (2,2);
\draw [my grid,my coordarrow] (-1.6,0) -- (1.6,0) node[anchor=north,my cood] {$v$};
\draw [my grid,my coordarrow] (0,-1.6) -- (0,1.6) node[anchor=east,my cood] {$v^\prime$};
\newcommand{\Da}{1.0}
\draw [my frame,dashed] (-1.5,-\Da) .. 
      controls (-1.8,-\Da-0.3)
      and (-2,-\Da-0.5) .. 
      (-2,-2);
\draw [my frame] (-1.5,-\Da) -- (\Da,1.5);
\draw [my frame,dashed] (\Da,1.5) .. 
      controls (\Da+0.3,1.8) 
      and (\Da+0.5,2) .. 
      (2,2);
\newcommand{\Ea}{0.4}
\draw [my frame,dashed] (-\Ea,-1.5) ..
      controls (-\Ea-0.3,-1.8)
      and (-\Ea-0.5,-2) ..
      (-2,-2);
\draw [my frame] (-\Ea,-1.5) -- (1.5,\Ea);
\draw [my frame,dashed] (1.5,\Ea) ..
      controls (1.8,\Ea+0.3)
      and (2,\Ea+0.5) ..
      (2,2);
\begin{scope}[on background layer]
\fill [black!20!white] 
      (-2,-2) ..controls(-2,-\Da-0.5)and(-1.8,-\Da-0.3).. 
      (-1.5,-\Da) -- 
      (\Da,1.5) ..controls(\Da+0.3,1.8)and(\Da+0.5,2).. 
      (2,2) ..controls(2,\Ea+0.5)and(1.8,\Ea+0.3)..
      (1.5,\Ea) --
      (-\Ea,-1.5) ..controls(-\Ea-0.3,-1.8)and(-\Ea-0.5,-2)..
      (-2,-2);
\end{scope}
\filldraw [black] (-2,-2) circle (\csize)
          (2,2) circle (\csize);
\end{scope}

\begin{scope}[shift={(7.5,0)}] 
\draw [my grid] (-2,-2) rectangle (2,2);
\draw [my grid,my coordarrow] (-1.6,0) -- (1.6,0) node[anchor=north,my cood] {$w$};
\draw [my grid,my coordarrow] (0,-1.6) -- (0,1.6) node[anchor=east,my cood] {$w^\prime$};
\newcommand{\Db}{0.3}
\draw [my frame,dashed] (-1.5,-\Db) .. 
      controls (-1.8,-\Db-0.3) 
      and (-2,-\Db-0.5) .. 
      (-2,-2);
\draw [my frame] (-1.5,-\Db) -- (\Db,1.5);
\draw [my frame,dashed] (\Db,1.5) .. 
      controls (\Db+0.3,1.8) 
      and (\Db+0.5,2) .. 
      (2,2);
\newcommand{\Eb}{-0.5}
\draw [my frame,dashed] (-\Eb,-1.5) ..
      controls (-\Eb-0.3,-1.8)
      and (-\Eb-0.5,-2) ..
      (-2,-2);
\draw [my frame] (-\Eb,-1.5) -- (1.5,\Eb);
\draw [my frame,dashed] (1.5,\Eb) ..
      controls (1.8,\Eb+0.3)
      and (2,\Eb+0.5) ..
      (2,2);
\begin{scope}[on background layer]
\fill [black!20!white] 
      (-2,-2) ..controls(-2,-\Db-0.5)and(-1.8,-\Db-0.3).. 
      (-1.5,-\Db) -- 
      (\Db,1.5) ..controls(\Db+0.3,1.8)and(\Db+0.5,2).. 
      (2,2) ..controls(2,\Eb+0.5)and(1.8,\Eb+0.3)..
      (1.5,\Eb) --
      (-\Eb,-1.5) ..controls(-\Eb-0.3,-1.8)and(-\Eb-0.5,-2)..
      (-2,-2);
\end{scope}

\newcommand{\Da}{1.0}
\draw [my frame,dashed] (-1.5,-\Da) .. 
      controls (-1.8,-\Da-0.3) 
      and (-2,-\Da-0.5) .. 
      (-2,-2);
\draw [my frame] (-1.5,-\Da) -- (\Da,1.5);
\draw [my frame,dashed] (\Da,1.5) .. 
      controls (\Da+0.3,1.8) 
      and (\Da+0.5,2) .. 
      (2,2);
\newcommand{\Ea}{0.4}
\draw [my frame,dashed] (-\Ea,-1.5) ..
      controls (-\Ea-0.3,-1.8)
      and (-\Ea-0.5,-2) ..
      (-2,-2);
\draw [my frame] (-\Ea,-1.5) -- (1.5,\Ea);
\draw [my frame,dashed] (1.5,\Ea) ..
      controls (1.8,\Ea+0.3)
      and (2,\Ea+0.5) ..
      (2,2);
\begin{scope}[on background layer]
\fill [black!40!white] 
      (-2,-2) ..controls(-2,-\Da-0.5)and(-1.8,-\Da-0.3).. 
      (-1.5,-\Da) -- 
      (\Da,1.5) ..controls(\Da+0.3,1.8)and(\Da+0.5,2).. 
      (2,2) ..controls(2,\Ea+0.5)and(1.8,\Ea+0.3)..
      (1.5,\Ea) --
      (-\Ea,-1.5) ..controls(-\Ea-0.3,-1.8)and(-\Ea-0.5,-2)..
      (-2,-2);
\end{scope}
\filldraw [black] (-2,-2) circle (\csize)
          (2,2) circle (\csize);
\end{scope}
\useasboundingbox (-2,-2) rectangle (9.5,5.5);
\end{tikzpicture}
\caption{A homomorphism between  $\overline{\mathbb{K}}$-extended L-convex sets}
\label{fig:lcshom}
\end{figure}

We note that homomorphisms include identity 
maps and are closed under composition (therefore,
$\overline{\mathbb{K}}$-extended L-convex sets 
together with their
homomorphisms, form a category). 
Using these facts, we define the notion of isomorphisms
between $\overline{\mathbb{K}}$-extended L-convex sets
as follows:
\begin{defn}
Let $\lcs{D}$ and $\lcs{E}$ be 
$\overline{\mathbb{K}}$-extended L-convex sets.
A homomorphism 
\[
\Phi\colon \lcs{D}\longrightarrow \lcs{E}
\] 
from $\lcs{D}$ to $\lcs{E}$
is an \textbf{isomorphism} 
if there exists a homomorphism 
(called its \textbf{inverse})
\[
\Psi\colon \lcs{E}\longrightarrow \lcs{D}
\]
from $\lcs{E}$ to $\lcs{D}$ such that
$\Psi\circ\Phi=1_{\lcs{D}}$ and 
$\Phi\circ\Psi=1_{\lcs{E}}$
hold, where $1_{\lcs{D}}$ and $1_{\lcs{E}}$ are the
identity homomorphisms on $\lcs{D}$ and $\lcs{E}$,
respectively.

If there is an isomorphism between 
$\lcs{D}$ and $\lcs{E}$, then they are said to be 
\textbf{isomorphic} and written as 
$\lcs{D}\cong\lcs{E}$.
\end{defn} 
In other words, 
$\Psi\colon \lcs{E}\longrightarrow \lcs{D}$ is the 
inverse of 
$\Phi\colon \lcs{D}\longrightarrow \lcs{E}$
if and only if the map
$\ind{\Psi}\colon \ind{\lcs{D}}\longrightarrow \ind{\lcs{E}}$ 
is the inverse of
$\ind{\Phi}\colon \ind{\lcs{E}}\longrightarrow \ind{\lcs{D}}$.

In Definition \ref{def:canord}
we defined the canonical ordering 
$\Rightarrow$ on the set $[\mathcal{A},\mathcal{B}]_0$
of all $\mathcal{V}$-functors from a 
$\mathcal{V}$-category $\mathcal{A}$ to $\mathcal{B}$,
hence also for the special case of 
$\mathcal{V}=\overline{\mathbb{K}}$.
Now we define a similar relation (also denoted by 
$\Rightarrow$) on the set 
$\db{\lcs{D},\lcs{E}}_0$ of homomorphisms.

\begin{defn}
Let $\lcs{D}=(\ind{\lcs{D}},\lcs{D}_0)$
and $\lcs{E}=(\ind{\lcs{E}},\lcs{E}_0)$ be 
$\overline{\mathbb{K}}$-extended L-convex sets.
The set $\db{\lcs{D},\lcs{E}}$ of all
homomorphisms from $\lcs{D}$ to $\lcs{E}$ admits
a canonical relation $\Rightarrow$ defined as follows:
\[
	\Phi\Rightarrow\Psi\iff\Phi_0(p)\geq\Psi_0(p)\quad(\forall p\in\lcs{D}_0),
\]
where $\Phi,\Psi\colon\lcs{D}\longrightarrow\lcs{E}$
are homomorphisms and 
$\geq$ is the coordinate-wise order on $\lcs{E}_0$.
We call $\Rightarrow$ the \textbf{canonical ordering}
on $\db{\lcs{D},\lcs{E}}_0$.
\end{defn}

Note that we can equip the underlying set $\lcs{D}_0$
with a distance function $d_{\lcs{D}_0}$ defined by
the restriction of the sup-distance
on $\overline{\mathbb{K}}^\ind{\lcs{D}}$.
Under these distances, one can show that 
the underlying map $\Phi_0$ of every homomorphism
$\Phi\colon\lcs{D}\longrightarrow\lcs{E}$ is
nonexpansive,
hence defines a $\overline{\mathbb{K}}$-functor.
Therefore, the notion of the canonical ordering
$\Rightarrow$ for $\overline{\mathbb{K}}$-functors
restricts to that for the underlying maps 
of homomorphisms.
The above definition coincides with 
the result of this approach.
Hence $\Rightarrow$ for homomorphisms is 
also a preorder, and is preserved under composition
(note that $(\Psi\circ\Phi)_0=\Psi_0\circ\Phi_0$).

\section{Duality Theorems}
In this section we show the main result of 
the thesis, a duality theorem between
$\overline{\mathbb{K}}$-categories and 
$\overline{\mathbb{K}}$-extended L-convex sets.
First we establish the duality at the level of 
individual objects;
the resulting Theorem \ref{mainthm_ob} states
that there is a one-to-one correspondence 
between the isomorphism class of 
$\overline{\mathbb{K}}$-categories and that of
$\overline{\mathbb{K}}$-extended L-convex sets.
In fact, a similar result for L-convex sets
and L-convex polyhedra is already known via almost 
parallel constructions in the
context of discrete convex analysis.
However, there is an essentially novel aspect 
(conceptually, rather than technically)
in our approach; we employ the notions of 
$\overline{\mathbb{K}}$-functors and homomorphisms of
$\overline{\mathbb{K}}$-extended L-convex sets in the
constructions of the duality, and 
realize them as the function space 
constructions 
with codomain $\overline{\mathbb{K}}$.

By introducing homomorphisms, we can compare these
``structure-preserving maps'' of 
$\overline{\mathbb{K}}$-extended L-convex sets with
$\overline{\mathbb{K}}$-functors,
the equally natural maps between 
$\overline{\mathbb{K}}$-categories.
We can describe the duality at the level of 
maps as well;
we obtain Theorem~\ref{mainthm_map} as a result.
In categorical terminology, we establish a
\textit{dual equivalence} between the category of 
$\overline{\mathbb{K}}$-categories and 
the category of $\overline{\mathbb{K}}$-extended 
L-convex sets.

Finally, we conclude the duality 
in Theorem \ref{mainthm_canord} by showing
that there is also a correspondence between
the canonical orderings $\Rightarrow$ on
$\overline{\mathbb{K}}$-functors and 
on homomorphisms, hence at the level of
canonical orderings.
In fact, the categories of 
$\overline{\mathbb{K}}$-categories and 
$\overline{\mathbb{K}}$-extended L-convex sets are
(strict) 2-categories 
(do not confuse with \textbf{2}-categories,
which are just preordered sets) 
by the canonical orderings $\Rightarrow$.
What we show is that our construction gives
a (strict) \textit{2-equivalence of 2-categories}.

\paragraph{}
If we denote the constructions for objects
by $[-,\overline{\mathbb{K}}]$
(from a $\overline{\mathbb{K}}$-category to a 
$\overline{\mathbb{K}}$-extended L-convex set)
and $\db{-,\overline{\mathbb{K}}}$
(from a $\overline{\mathbb{K}}$-extended L-convex set
to a $\overline{\mathbb{K}}$-category),
the duality at the object-level is expressed as 
follows:

\begin{thm}\label{mainthm_ob}
Let $\mathcal{A}=(\ob{\mathcal{A}},d_\mathcal{A})$ be
a $\overline{\mathbb{K}}$-category and
$\lcs{D}=(\ind{\lcs{D}},\lcs{D}_0)$
be a $\overline{\mathbb{K}}$-extended L-convex set.
Then the following hold:
\begin{enumerate}
\item
$\mathcal{A}\cong \db{[\mathcal{A},\overline{\mathbb{K}}],\overline{\mathbb{K}}}$.
\item
$\lcs{D}\cong [\db{\lcs{D},\overline{\mathbb{K}}},\overline{\mathbb{K}}]$.
\end{enumerate}
\end{thm}

Recall that $\cong$ means ``is isomorphic to.''
Let us define the constructions and prove that 
they certainly work. 

\begin{lem}
Let $\mathcal{A}=(\ob{\mathcal{A}},d_\mathcal{A})$ be
a $\overline{\mathbb{K}}$-category.
Then the pair 
$[\mathcal{A},\overline{\mathbb{K}}]=(\ob{\mathcal{A}},$ 
$[\mathcal{A},\overline{\mathbb{K}}]_0)$, 
where 
$[\mathcal{A},\overline{\mathbb{K}}]_0$
is the set of all nonexpansive maps from $\mathcal{A}$ to
$\overline{\mathbb{K}}$ 
(as a $\overline{\mathbb{K}}$-category),
seen as a subset of 
$\overline{\mathbb{K}}^{\ob{\mathcal{A}}}$, 
is a $\overline{\mathbb{K}}$-extended
L-convex set. 
\end{lem}
\begin{proof}
\begin{description}[font=\normalfont]
\item[[Order completeness\!\!\!]]
Let $S$ be a set of nonexpansive maps from
$\mathcal{A}$ to $\overline{\mathbb{K}}$.
The claim is that $\bigvee_{p\in S} p$, 
$\bigwedge_{p\in S}p$
are again nonexpansive maps from $\mathcal{A}$ to
$\overline{\mathbb{K}}$.
In other words, that
\begin{align}
d_\mathcal{A}(a,b)
&\geq d_{\overline{\mathbb{K}}}\big(\bigvee_{p\in S} p(a),
\bigvee_{p\in S} p(b)\big)
=\inf_{p\in S} \{p(b)\}-\inf_{p\in S} \{p(a)\},\label{oc_inf}\\
d_\mathcal{A}(a,b)
&\geq d_{\overline{\mathbb{K}}}\big(\bigwedge_{p\in S} p(a),\bigwedge_{p\in S} p(b)\big)
=\sup_{p\in S} \{p(b)\}-\sup_{p\in S} \{p(a)\}\label{oc_sup}
\end{align}
hold for all $a,b\in \ob{\mathcal{A}}$.
\begin{description}[font=\normalfont]
\item[[(\ref{oc_inf})\!\!\!]]
For each fixed $p^\prime\in S$, the increasing condition
implies
\[
	d_\mathcal{A}(a,b)\geq p^\prime(b)-p^\prime(a)
	\geq \inf_{p\in S} \{p(b)\}-p^\prime (a).
\]
By taking $\sup$ with respect to $p^\prime\in S$,
we have
\[
	d_\mathcal{A}(a,b)\geq \sup_{p^\prime\in S}
	\left\{\inf_{p\in S}\{p(b)\}-p^\prime(a)\right\}
	=\inf_{p\in S} \{p(b)\}-\inf_{p\in S} \{p(a)\},
\]
where the equality follows from an instance of
a statement in
Proposition \ref{prop:lim_ten_hom} that
\[
\inf_{p\in S} \{p(b)\}-(-)
\]
turns suprema ($\inf$'s) into infima ($\sup$'s).
\item[[(\ref{oc_sup})\!\!\!]]
For each fixed $p^\prime\in S$, the increasing condition
implies
\[
	d_\mathcal{A}(a,b)\geq p^\prime(b)-p^\prime(a)
	\geq p^\prime(b)-\sup_{p\in S} \{p(a)\}.
\]
By taking $\sup$ with respect to $p^\prime\in S$,
we have
\[
	d_\mathcal{A}(a,b)\geq \sup_{p^\prime\in S}
	\left\{p^\prime(b)-\sup_{p\in S} \{p(a)\}\right\}
	=\sup_{p\in S} \{p(b)\}-\sup_{p\in S} \{p(a)\},
\]
where the equality follows from an instance of
a statement in
Proposition \ref{prop:SMCCL} that 
\[
(-)-\sup_{p\in S} \{p(a)\}
\]
preserves infima ($\sup$'s).
\end{description}
\item[[Weight completeness (for $+$)\!\!\!]]
Let $p$ be a nonexpansive map from $\mathcal{A}$ to 
$\overline{\mathbb{K}}$ and 
$\alpha\in \mathbb{K}\cup\{-\infty,\infty\}$.
The claim is that $p+\alpha\cdot\textbf{1}$
is again a nonexpansive map 
from $\mathcal{A}$ to $\overline{\mathbb{K}}$.
In other words, that 
\begin{align*}
d_\mathcal{A}(a,b)&\geq d_{\overline{\mathbb{K}}}\big(p(a)+\alpha,p(b)+\alpha\big)
=\big(p(b)+\alpha\big)-\big(p(a)+\alpha\big)
\end{align*}
hold for all $a,b\in \ob{\mathcal{A}}$.
We use the increasing condition for $p$
\[
	d_\mathcal{A}(a,b)\geq p(b)-p(a)
\]
and show 
\[
p(b)-p(a)\geq \big(p(b)+\alpha\big)-\big(p(a)+\alpha\big),
\]
which has an equivalent formula:
\begin{prooftree}
	\def\fCenter{\ \geq\ }
	\alwaysDoubleLine
	\Axiom$p(b)-p(a)\fCenter \big(p(b)+\alpha\big)-\big(p(a)+\alpha\big)$
	\UnaryInf$\big(p(b)-p(a)\big)+\big(p(a)+\alpha\big) \fCenter p(b)+\alpha$
	\UnaryInf$\big(\big(p(b)-p(a)\big)+p(a)\big)+\alpha \fCenter p(b)+\alpha$
\end{prooftree}
Therefore it suffices to prove 
$\big(p(b)-p(a)\big)+p(a)\geq p(b)$, as follows:
\begin{prooftree}
	\def\fCenter{\ \geq\ }
	\alwaysDoubleLine
	\Axiom$\big(p(b)-p(a)\big)+p(a)\fCenter p(b)$
	\UnaryInf$p(b)-p(a) \fCenter p(b)-p(a)$
\end{prooftree}
\end{description}
\end{proof}

\begin{lem}\label{lcs2kcat_ob}
Let $\lcs{D}=(\ind{\lcs{D}},\lcs{D}_0)$
be a $\overline{\mathbb{K}}$-extended L-convex set.
Then the pair
$\db{\lcs{D},\overline{\mathbb{K}}}=\big(\db{\lcs{D},\overline{\mathbb{K}}}_0,d_{\db{\lcs{D},\overline{\mathbb{K}}}}\big)$,
where
$\db{\lcs{D},\overline{\mathbb{K}}}_0$ is the set of all 
homomorphisms from $\lcs{D}$ to $\overline{\mathbb{K}}$
(as a $\overline{\mathbb{K}}$-extended L-convex set)
and $d_{\db{\lcs{D},\overline{\mathbb{K}}}}$  
is the distance function on it defined by the 
sup-distance between the underlying maps, 
is a $\overline{\mathbb{K}}$-category.
\end{lem}

\begin{proof}
First note that elements of 
$\db{\lcs{D},\overline{\mathbb{K}}}_0$ 
can canonically be identified with that of 
$\ind{\lcs{D}}$, 
via a bijection $\pi_\lcs{D}$ defined as follows:
\[
	\pi_\lcs{D}\colon \ind{\lcs{D}}\longrightarrow \db{\lcs{D},\overline{\mathbb{K}}}_0,
	\qquad 
	\pi_\lcs{D}(v)=\pi_v
	\quad (\forall v\in \ind{\lcs{D}}),
\]
where
\[
	\ind{\pi_v}\colon \{\ast\}\longrightarrow \ind{\lcs{D}},\qquad
	\ind{\pi_v}(\ast)=v.
\]
The underlying maps of $\pi_v$'s are given as the 
``projection maps'':
\[
(\pi_v)_0\colon \lcs{D}_0\longrightarrow \mathbb{K}\cup\{-\infty,\infty\}\subseteq \overline{\mathbb{K}}^{\{\ast\}}\qquad
(\pi_v)_0(p)=p(v).
\]
See Fig.~\ref{fig:lcsproj} for an illustration.
Thus we may write $\db{\lcs{D},\overline{\mathbb{K}}}_0=\{\pi_v\}_{v\in \ind{\lcs{D}}}$.
Also note that 
\begin{align}
d_{\db{\lcs{D},\overline{\mathbb{K}}}}(\pi_v,\pi_w)
&=\sup_{p\in \lcs{D}_0}\{d_{\overline{\mathbb{K}}}((\pi_v)_0(p),(\pi_w)_0(p))\}\nonumber\\
&=\sup_{p\in \lcs{D}_0}\{d_{\overline{\mathbb{K}}}(p(v),p(w))\}\nonumber\\
&=\sup_{p\in \lcs{D}_0}\{p(w)-p(v)\}.\label{d_dk}
\end{align}
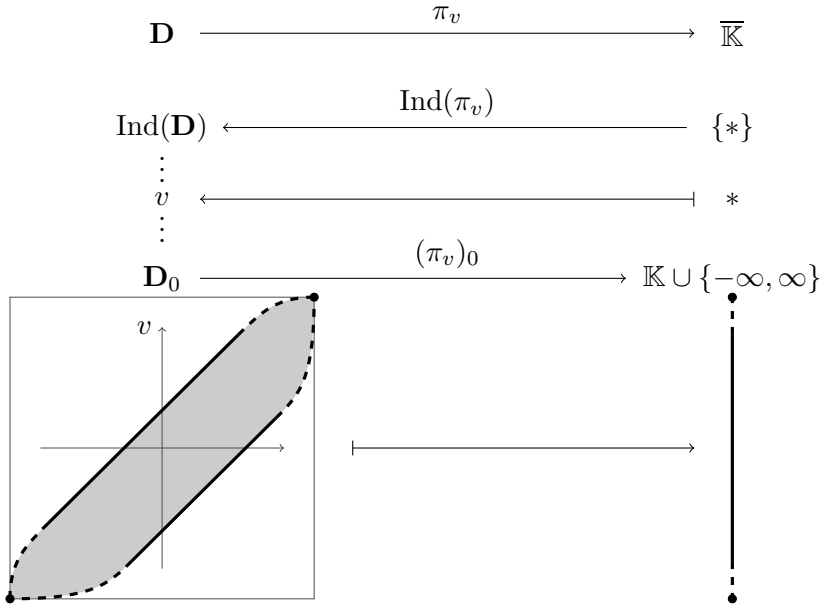
\begin{figure}
\centering
\begin{tikzpicture}
\node at (0,5.5) {$\lcs{D}$};
\node at (7.5,5.5) {$\overline{\mathbb{K}}$};
\draw [my arrow] (0.5,5.5) -- (7,5.5);
\node at (3.75,5.5)[anchor=south] {$\pi_v$};
\node at (0,4.25) {$\ind{\lcs{D}}$};
\node at (7.5,4.25) {$\{\ast\}$};
\node at (0,3.8) {$\vdots$};
\node at (0,3.3) {$v$};
\node at (7.5,3.3) {$\ast$};
\node at (0,3) {$\vdots$};
\draw [my mapsto] (7,3.3) -- (0.5,3.3);
\draw [my arrow] (6.9,4.25) -- (0.8,4.25);
\node at (3.75,4.25) [anchor=south]{$\ind{\pi_v}$};
\draw [my mapsto] (2.5,0) -- (7,0);
\draw [my arrow] (0.5,2.25) -- (6.1,2.25);
\node at (3.75,2.25)[anchor=south] {$(\pi_v)_0$};
\node at (0,2.25) {$\lcs{D}_0$};
\node at (7.5,2.25) {$\mathbb{K}\cup\{-\infty,\infty\}$};

\begin{scope} 
\draw [my grid] (-2,-2) rectangle (2,2);
\draw [my grid,my coordarrow] (-1.6,0) -- (1.6,0);
\draw [my grid,my coordarrow] (0,-1.6) -- (0,1.6) node[anchor=east,my cood] {$v$};
\newcommand{\Da}{1.0}
\draw [my frame,dashed] (-1.5,-\Da) .. 
      controls (-1.8,-\Da-0.3) 
      and (-2,-\Da-0.5) .. 
      (-2,-2);
\draw [my frame] (-1.5,-\Da) -- (\Da,1.5);
\draw [my frame,dashed] (\Da,1.5) .. 
      controls (\Da+0.3,1.8) 
      and (\Da+0.5,2) .. 
      (2,2);
\newcommand{\Ea}{0.4}
\draw [my frame,dashed] (-\Ea,-1.5) ..
      controls (-\Ea-0.3,-1.8)
      and (-\Ea-0.5,-2) ..
      (-2,-2);
\draw [my frame] (-\Ea,-1.5) -- (1.5,\Ea);
\draw [my frame,dashed] (1.5,\Ea) ..
      controls (1.8,\Ea+0.3)
      and (2,\Ea+0.5) ..
      (2,2);
\begin{scope}[on background layer]
\fill [black!20!white] 
      (-2,-2) ..controls(-2,-\Da-0.5)and(-1.8,-\Da-0.3).. 
      (-1.5,-\Da) -- 
      (\Da,1.5) ..controls(\Da+0.3,1.8)and(\Da+0.5,2).. 
      (2,2) ..controls(2,\Ea+0.5)and(1.8,\Ea+0.3)..
      (1.5,\Ea) --
      (-\Ea,-1.5) ..controls(-\Ea-0.3,-1.8)and(-\Ea-0.5,-2)..
      (-2,-2);
\end{scope}
\filldraw [black] (-2,-2) circle (\csize)
          (2,2) circle (\csize);
\end{scope}

\begin{scope}[shift={(7.5,0)}] 
\draw [my frame,dashed] (0,-1.5) -- (0,-2);
\draw [my frame] (0,-1.5) -- (0,1.5);
\draw [my frame,dashed] (0,1.5) -- (0,2);
\filldraw [black] (0,-2) circle (\csize)
          (0,2) circle (\csize);
\end{scope}
\useasboundingbox (-2,-2) rectangle (9.5,5.5);
\end{tikzpicture}
\caption{The homomorphism $\pi_v$}
\label{fig:lcsproj}
\end{figure}
\begin{description}[font=\normalfont]
\item[[Composition law\!\!\!]]
The claim is 
\[
	d_{\db{\lcs{D},\overline{\mathbb{K}}}}(\pi_u,\pi_v)+
	d_{\db{\lcs{D},\overline{\mathbb{K}}}}(\pi_v,\pi_w)\geq
	d_{\db{\lcs{D},\overline{\mathbb{K}}}}(\pi_u,\pi_w),
\]
or by (\ref{d_dk}), equivalently,
\begin{align}\label{triangle}
\sup_{p\in \lcs{D}_0}\{p(v)-p(u)\}+\sup_{p\in \lcs{D}_0}\{p(w)-p(v)\}
\geq\sup_{p\in \lcs{D}_0}\{p(w)-p(u)\}.
\end{align}
We first prove 
\begin{align}\label{triangle3}
\big(p(v)-p(u)\big) + \big(p(w)-p(v)\big)\geq p(w)-p(u)
\end{align}
for each $p\in \lcs{D}_0$, as follows:
\begin{prooftree}
	\def\fCenter{\ \geq\ }
	\alwaysDoubleLine
	\Axiom$\big(p(v)-p(u)\big) + \big(p(w)-p(v)\big)\fCenter p(w)-p(u)$
	\UnaryInf$\big(p(w)-p(v)\big)+\big(p(v)-p(u)\big) \fCenter p(w)-p(u)$
	\UnaryInf$p(w)-p(v) \fCenter \big(p(w)-p(u)\big)-\big(p(v)-p(u)\big)$
\end{prooftree}
where the bottom expression is an instance of the
composition law for $-$. 
Thus, by taking $\sup$ in (\ref{triangle3}) 
with respect to $p$, we have
\begin{align}\label{triangle2}
	\sup_{p\in \lcs{D}_0}\big\{\big(p(v)-p(u)\big) +\big(p(w)-p(v)\big)\big\}
	\geq\sup_{p\in \lcs{D}_0}\{p(w)-p(u)\}.
\end{align}
On the other hand, we have 
\[
\sup_{p\in \lcs{D}_0}\{p(v)-p(u)\} +\sup_{p\in \lcs{D}_0}\{p(w)-p(v)\}
\geq\big(p^\prime(v)-p^\prime(u)\big) +\big(p^\prime(w)-p^\prime(v)\big)
\]
for each $p^\prime\in \lcs{D}_0$, and therefore
\[
\sup_{p\in \lcs{D}_0}\{p(v)-p(u)\} +\sup_{p\in \lcs{D}_0}\{p(w)-p(v)\}\geq 
\sup_{p\in \lcs{D}_0}\big\{\big(p(v)-p(u)\big) +\big(p(w)-p(v)\big)\big\}.
\]
We combine this with (\ref{triangle2}) and obtain
(\ref{triangle}).
\item[[Identity law\!\!\!]]
The claim is 
\[
	0\geq
	d_{\db{\lcs{D},\overline{\mathbb{K}}}}(\pi_v,\pi_v),
\]
or by (\ref{d_dk}), equivalently,
\begin{align}\label{identity1}
0\geq\sup_{p\in \lcs{D}_0}\{p(v)-p(v)\}.
\end{align}
First note that for each $p\in \lcs{D}_0$
\begin{align}\label{identity2}
0\geq p(v)-p(v)
\end{align}
holds, as proved below:
\begin{prooftree}
	\def\fCenter{\ \geq\ }
	\alwaysDoubleLine
	\Axiom$0\fCenter p(v)-p(v)$
	\UnaryInf$p(v)\fCenter p(v)$
\end{prooftree}
Taking $\sup$ in (\ref{identity2}) with respect to $p$
yields (\ref{identity1}).
\end{description}
\end{proof}

Now all the symbols are defined and 
we are ready to prove Theorem~\ref{mainthm_ob}:

\begin{proof}[Proof of Theorem~\ref{mainthm_ob}]
\begin{description}[font=\normalfont]
\item[[(i)\!\!\!]]
Since 
$[\mathcal{A},\overline{\mathbb{K}}]=(\ob{\mathcal{A}},[\mathcal{A},\overline{\mathbb{K}}]_0)$,
$\db{[\mathcal{A},\overline{\mathbb{K}}],\overline{\mathbb{K}}}_0=\{\pi_a\}_{a\in \ob{\mathcal{A}}}$.
We claim that 
\[
\unit{\mathcal{A}}\colon\mathcal{A}\longrightarrow\db{[\mathcal{A},\overline{\mathbb{K}}],\overline{\mathbb{K}}}
\]
given by
\[
\ob{\unit{\mathcal{A}}}\colon\ob{\mathcal{A}}\longrightarrow\db{[\mathcal{A},\overline{\mathbb{K}}],\overline{\mathbb{K}}}_0,\qquad \ob{\unit{\mathcal{A}}}(a)=\unit{\mathcal{A}}(a)=\pi_a
\quad (\forall a\in \ob{\mathcal{A}})
\]
as a map between the sets of points, is an isomorphism
(note that 
$\ob{\unit{\mathcal{A}}}=\pi_{[\mathcal{A},\overline{\mathbb{K}}]}$,
where $\pi_{[\mathcal{A},\overline{\mathbb{K}}]}$ is
the bijection defined in Lemma \ref{lcs2kcat_ob}).
By Proposition \ref{iso_ffb}, it suffices to show
that $\unit{\mathcal{A}}$ is a fully faithful 
$\overline{\mathbb{K}}$-functor 
(because $\ob{\unit{\mathcal{A}}}$ is 
certainly bijective).
Therefore, we aim to show
\begin{align*}
d_\mathcal{A}(a,b)
&=d_{\db{[\mathcal{A},\overline{\mathbb{K}}],\overline{\mathbb{K}}}}(\pi_a,\pi_b),
\end{align*}
or equivalently,
\begin{align*}
d_\mathcal{A}(a,b)
&=\sup_{p\in [\mathcal{A},\overline{\mathbb{K}}]_0}\{p(b)-p(a)\}.
\end{align*}
\begin{description}[font=\normalfont]
\item[[\mbox{$d_\mathcal{A}(a,b)
\leq \sup_{p\in [\mathcal{A},\overline{\mathbb{K}}]_0}\{p(b)-p(a)\}\!\!\!$}]]
Recall that $\overline{Y}(a)$ defined by
\[
\ob{\overline{Y}(a)}=\lambda b\in \ob{\mathcal{A}}.\,\Hom_\mathcal{A}(a,b)
\]
is an element of $[\mathcal{A},\overline{\mathbb{K}}]_0$.
Therefore, 
\[
d_\mathcal{A}(a,b)-d_\mathcal{A}(a,a)=
\overline{Y}(a)(b)-\overline{Y}(a)(a)\leq 
\sup_{p\in [\mathcal{A},\overline{\mathbb{K}}]_0}\{p(b)-p(a)\}
\]
and it suffices to show 
\[
d_\mathcal{A}(a,b)\geq d_\mathcal{A}(a,b)-d_\mathcal{A}(a,a),
\]
but this follows from the identity law 
$0\geq d_\mathcal{A}(a,a)$ for $\mathcal{A}$ and 
the unit law and the monotonicity of $-$.
\item[[\mbox{$d_\mathcal{A}(a,b)
\geq \sup_{p\in [\mathcal{A},\overline{\mathbb{K}}]_0}\{p(b)-p(a)\}\!\!\!$}]]
It suffices to show 
\[
d_\mathcal{A}(a,b)\geq p(b)-p(a)
\]
for each $p\in [\mathcal{A},\overline{\mathbb{K}}]_0$;
but this is nothing but the increasing condition for
$p$.
\end{description}
\item[[(ii)\!\!\!]]
Recall that 
$\db{\lcs{D},\overline{\mathbb{K}}}=(\{\pi_v\}_{v\in \ind{\lcs{D}}},d_{\db{\lcs{D},\overline{\mathbb{K}}}})$,
where
\[
d_{\db{\lcs{D},\overline{\mathbb{K}}}}(\pi_v,\pi_w)
=\sup_{p\in \lcs{D}_0}\{p(w)-p(v)\}.
\]
Now, the $\overline{\mathbb{K}}$-extended L-convex set
$[\db{\lcs{D},\overline{\mathbb{K}}},\overline{\mathbb{K}}]=
(\{\pi_v\}_{v\in\ind{\lcs{D}}},[\db{\lcs{D},\overline{\mathbb{K}}},\overline{\mathbb{K}}]_0)$
is defined as follows: for all $q\in \overline{\mathbb{K}}^{\{\pi_v\}_{v\in \ind{\lcs{D}}}}$,
\[
q\in [\db{\lcs{D},\overline{\mathbb{K}}},\overline{\mathbb{K}}]_0
\iff 
d_{\db{\lcs{D},\overline{\mathbb{K}}}}(\pi_v,\pi_w)\geq 
q(\pi_w)-q(\pi_v)\quad (\forall \pi_v,\pi_w\in \{\pi_v\}_{v\in \ind{\lcs{D}}}).
\]
We prove that bijections
\begin{align*}
\ind{\counit{\lcs{D}}}\colon \ind{\lcs{D}}\longrightarrow \{\pi_v\}_{v\in \ind{\lcs{D}}},\qquad &
\ind{\counit{\lcs{D}}}(v)=\pi_v \quad (\forall v\in \ind{\lcs{D}}),\\
\ind{\counitinv{\lcs{D}}}\colon \{\pi_v\}_{v\in \ind{\lcs{D}}}\longrightarrow \ind{\lcs{D}},\qquad &
\ind{\counitinv{\lcs{D}}}(\pi_v)=v\quad (\forall \pi_v\in \{\pi_v\}_{v\in \ind{\lcs{D}}})
\end{align*}
define homomorphisms 
(note that $\ind{\counit{\lcs{D}}}=\pi_\lcs{D}$,
where $\pi_\lcs{D}$ is the bijection defined in the 
proof of Lemma \ref{lcs2kcat_ob}); 
in other words, there are maps
\begin{align*}
(\counit{\lcs{D}})_0\colon [\db{\lcs{D},\overline{\mathbb{K}}},\overline{\mathbb{K}}]_0\longrightarrow \lcs{D}_0,\quad&
(\counit{\lcs{D}})_0(q)=\lambda v\in \ind{\lcs{D}}.\, q(\pi_v) \quad (\forall q\in [\db{\lcs{D},\overline{\mathbb{K}}},\overline{\mathbb{K}}]_0),\\
(\counitinv{\lcs{D}})_0\colon\lcs{D}_0\longrightarrow [\db{\lcs{D},\overline{\mathbb{K}}},\overline{\mathbb{K}}]_0,\quad  &
(\counitinv{\lcs{D}})_0(p)=\lambda \pi_v\in \{\pi_v\}_{v\in \ind{\lcs{D}}}.\,p(v)\quad (\forall p\in \lcs{D}_0).
\end{align*}
To be more specific about the point, we have to verify
that the images of the above maps are 
contained in the codomains.
\begin{description}[font=\normalfont]
\item[[\mbox{$(\counit{\lcs{D}})_0([\db{\lcs{D},\overline{\mathbb{K}}}, \overline{\mathbb{K}}]_0)\subseteq \lcs{D}_0\!\!\!$}]]
First we prove that for all images 
$\overline{Y}(\pi_v)\in [\db{\lcs{D},\overline{\mathbb{K}}}, \overline{\mathbb{K}}]_0$
of the co-Yoneda embedding, 
$(\counit{\lcs{D}})_0\big(\ob{\overline{Y}(\pi_v)}\big)$,
abbreviated as 
$(\counit{\lcs{D}})_0\big(\overline{Y}(\pi_v)\big)$, 
is an element of $\lcs{D}_0$.
To see this, observe that
\begin{align*}
(\counit{\lcs{D}})_0\big(\overline{Y}(\pi_v)\big)
&=\lambda w\in \ind{\lcs{D}}.\,d_{\db{\lcs{D},\overline{\mathbb{K}}}}(\pi_v,\pi_w)\\
&=\lambda w\in \ind{\lcs{D}}.\,\sup_{p\in \lcs{D}_0}\{p(w)-p(v)\},
\end{align*}
in other words, 
\begin{align*}
(\counit{\lcs{D}})_0\big(\overline{Y}(\pi_v)\big)
&=\sup_{p\in \lcs{D}_0}\{p-p(v)\cdot\textbf{1}\}\\
&=\bigwedge_{p\in \lcs{D}_0}\{p-p(v)\cdot\textbf{1}\}
\end{align*}
holds.
So the weight and order completeness conditions for
$\lcs{D}$ assure 
$(\counit{\lcs{D}})_0\big(\overline{Y}(\pi_v)\big)\in\lcs{D}_0$.

Now take arbitrary 
$q\in [\db{\lcs{D},\overline{\mathbb{K}}}, \overline{\mathbb{K}}]_0$. 
We have
\[
d_{\db{\lcs{D},\overline{\mathbb{K}}}}(\pi_v,\pi_w)\geq 
q(\pi_w)-q(\pi_v)\quad (\forall \pi_v,\pi_w\in \{\pi_v\}_{v\in \ind{\lcs{D}}}).
\]
The adjointness relation (for the SMC-CL 
$\overline{\mathbb{K}}$) yields
\[
d_{\db{\lcs{D},\overline{\mathbb{K}}}}(\pi_v,\pi_w)
+q(\pi_v)\geq 
q(\pi_w)\quad (\forall \pi_v,\pi_w\in \{\pi_v\}_{v\in \ind{\lcs{D}}}),
\]
so, with respect to the coordinate-wise order on 
$\overline{\mathbb{K}}^\ind{\lcs{D}}$
(also denoted by $\geq$), we obtain
\begin{align}\label{yoneda_q}
(\counit{\lcs{D}})_0\big(\overline{Y}(\pi_v)\big)
+q(\pi_v)\cdot \textbf{1}\geq 
(\counit{\lcs{D}})_0(q)\quad (\forall v\in \ind{\lcs{D}}).
\end{align}
Since by the weight completeness 
$(\counit{\lcs{D}})_0\big(\overline{Y}(\pi_v)\big)+q(\pi_v)\cdot \textbf{1}$
is an element of $\lcs{D}_0$, so is 
\[
q^\prime=\inf_{v\in \ind{\lcs{D}}}\left\{(\counit{\lcs{D}})_0\big(\overline{Y}(\pi_v)\big)+q(\pi_v)\cdot \textbf{1}\right\}=
\bigvee_{v\in \ind{\lcs{D}}}\left\{(\counit{\lcs{D}})_0\big(\overline{Y}(\pi_v)\big)+q(\pi_v)\cdot \textbf{1}\right\}
\]
by the order completeness.
We claim that $q^\prime=(\counit{\lcs{D}})_0(q)$.
Because the definition of $q^\prime$ and (\ref{yoneda_q})
imply $q^\prime\geq (\counit{\lcs{D}})_0(q)$,
it suffices to show that for each $v\in \ind{\lcs{D}}$,
$q^\prime(v)\leq (\counit{\lcs{D}})_0(q)(v)=q(\pi_v)$.
This is proved as follows:
\begin{align*}
q^\prime(v)
&=\left(\inf_{v\in \ind{\lcs{D}}}\left\{(\counit{\lcs{D}})_0\big(\overline{Y}(\pi_v)\big)+q(\pi_v)\cdot \textbf{1}\right\}\right)(v)\\
&\leq \left((\counit{\lcs{D}})_0\big(\overline{Y}(\pi_v)\big)+q(\pi_v)\cdot \textbf{1}\right)(v)\\
&=(\counit{\lcs{D}})_0\big(\overline{Y}(\pi_v)\big)(v)+(q(\pi_v)\cdot \textbf{1})(v)\\
&=d_{\db{\lcs{D},\overline{\mathbb{K}}}}(\pi_v,\pi_v)+q(\pi_v)\\
&\leq 0+q(\pi_v)\\
&=q(\pi_v).
\end{align*}
\item[[\mbox{$(\counitinv{\lcs{D}})_0(\lcs{D}_0)\subseteq [\db{\lcs{D},\overline{\mathbb{K}}}, \overline{\mathbb{K}}]_0\!\!\!$}]]
The condition for $(\counitinv{\lcs{D}})_0(p)$ 
to be an element of 
$[\db{\lcs{D},\overline{\mathbb{K}}}, \overline{\mathbb{K}}]_0$
is that 
\[
d_{\db{\lcs{D},\overline{\mathbb{K}}}}(\pi_v,\pi_w)\geq 
(\counitinv{\lcs{D}})_0(p)(\pi_w)-(\counitinv{\lcs{D}})_0(p)(\pi_v)\quad (\forall \pi_v,\pi_w\in \{\pi_v\}_{v\in \ind{\lcs{D}}}),
\]
which is equivalent to
\[
\sup_{p\in \lcs{D}_0}\{p(w)-p(v)\}\geq 
p(w)-p(v)\quad (\forall v,w\in \ind{\lcs{D}}),
\]
an obvious inequality.
\end{description}
Therefore, we have two homomorphisms
\begin{align*}
\counit{\lcs{D}}&\colon [\db{\lcs{D},\overline{\mathbb{K}}},\overline{\mathbb{K}}]\longrightarrow \lcs{D},\\
\counitinv{\lcs{D}}&\colon \lcs{D}\longrightarrow [\db{\lcs{D},\overline{\mathbb{K}}},\overline{\mathbb{K}}].
\end{align*}
Clearly
$\ind{\counitinv{\lcs{D}}}\circ \ind{\counit{\lcs{D}}}=1_\ind{\lcs{D}}$ and
$\ind{\counit{\lcs{D}}}\circ \ind{\counitinv{\lcs{D}}}=1_{\{\pi_v\}_{v\in \ind{\lcs{D}}}}$ hold, 
and they are isomorphisms.
\end{description}
\end{proof}

Theorem~\ref{mainthm_ob} immediately assures that 
the underlying sets of 
$\overline{\mathbb{K}}$-extended L-convex sets 
are, in a sense, nothing but the sets of points 
of presheaf categories for some
$\overline{\mathbb{K}}$-categories.

\begin{cor}\label{cor}
For any $\overline{\mathbb{K}}$-extended L-convex set
$\lcs{D}=(\ind{\lcs{D}},\lcs{D}_0)$, the pair
$\mathcal{D}=(\lcs{D}_0,d_{\lcs{D}_0})$
where $d_{\lcs{D}_0}$ is a distance function on 
$\lcs{D}_0$ given as the restriction of 
the sup-distance on 
$\overline{\mathbb{K}}^\ind{\lcs{D}}$, is 
isomorphic to the presheaf category 
$\funct{\db{\lcs{D},\overline{\mathbb{K}}}}{\overline{\mathbb{K}}}$
of the $\overline{\mathbb{K}}$-category
$\db{\lcs{D},\overline{\mathbb{K}}}^\op$.
Conversely, the presheaf category
$\funct{\mathcal{A}^\op}{\overline{\mathbb{K}}}=([\mathcal{A}^\op,\overline{\mathbb{K}}]_0,d_{[\mathcal{A}^\op,\overline{\mathbb{K}}]_0})$
of any $\overline{\mathbb{K}}$-category
$\mathcal{A}=(\ob{\mathcal{A}},d_\mathcal{A})$ 
gives rise to a 
$\overline{\mathbb{K}}$-extended L-convex set
$[\mathcal{A}^\op,\overline{\mathbb{K}}]=(\ob{\mathcal{A}},[\mathcal{A}^\op,\overline{\mathbb{K}}]_0)$.
\end{cor}

Let us apply Theorem \ref{mainthm_ob} to 
$\overline{\mathbb{K}}$-extended L-convex sets with 
two-element index sets and see how the shapes of 
the underlying sets relate to the distances of 
the corresponding two-point 
$\overline{\mathbb{K}}$-categories.
We obtain Fig.~\ref{fig:lcskcat}; 
in each figure, the lower half part is a 
picture of the underlying set $\lcs{D}_0$ of  
some $\overline{\mathbb{K}}$-extended 
L-convex set $\lcs{D}$, and above it we show
the $\overline{\mathbb{K}}$-category
$\db{\lcs{D},\overline{\mathbb{K}}}$
so that the number on the arrow
from a node $a$ to a node $b$ represents the distance
$d_{\db{\lcs{D},\overline{\mathbb{K}}}}(a,b)$.
In fact, these ten figures enumerate all the possible 
shapes of $\overline{\mathbb{K}}$-extended L-convex
sets with two-element index sets, and consequently also 
gives a complete classification  of 
$\overline{\mathbb{K}}$-categories with two points.

\begin{figure}[!p]
\centering
\begin{subfigure}{\wfig}
\centering
\begin{tikzpicture}
\useasboundingbox (-2,-1.935) rectangle (2,3.7);
\draw [my grid] (-2,-2) rectangle (2,2);
\draw [my grid,my coordarrow] (-1.6,0) -- (1.6,0) node[anchor=north,my cood] {$v$};
\draw [my grid,my coordarrow] (0,-1.6) -- (0,1.6) node[anchor=east,my cood] {$w$};
\draw [my frame,dashed] 
      (-1.5,-2) -- (-2,-2);
\draw [my frame]
      (-1.5,-2) -- (1.5,-2);
\draw [my frame,dashed]
      (1.5,-2) -- (2,-2); %
\draw [my frame,dashed] 
      (2,-1.5) -- (2,-2);
\draw [my frame]
      (2,-1.5) -- (2,1.5);
\draw [my frame,dashed]
      (2,1.5) -- (2,2); %
\draw [my frame,dashed] 
      (-1.5,2) -- (-2,2);
\draw [my frame]
      (-1.5,2) -- (1.5,2);
\draw [my frame,dashed]
      (1.5,2) -- (2,2); %
\draw [my frame,dashed] 
      (-2,-1.5) -- (-2,-2);
\draw [my frame]
      (-2,-1.5) -- (-2,1.5);
\draw [my frame,dashed]
      (-2,1.5) -- (-2,2); %
\begin{scope}[on background layer]
\fill [black!20!white] 
      (-2,-2) -- (2,-2) -- (2,2) -- (-2,2) -- cycle;
\end{scope}
\filldraw [black] (-2,-2) circle (\csize)
          (2,2) circle (\csize)
          (2,-2) circle (\csize)
          (-2,2) circle (\csize);
\node [thick,circle,draw] (v) at (-0.8,2.8) {$v$};
\node [thick,circle,draw] (w) at (0.8,2.8) {$w$};
\draw [->-=.5] (v) [in=210,out=150,loop] to node[above=6]
      {$0$} (v);
\draw [->-=.5] (v) to [bend right] node[below]
      {$\infty$} (w);
\draw [->-=.5] (w) to [bend right] node[above]
      {$\infty$} (v);
\draw [->-=.5] (w) [in=30,out=330,loop] to node[above=6]
      {$0$} (w);
\end{tikzpicture}
\caption{The whole set $\overline{\mathbb{K}}^2$}
\label{fig:cata}
\end{subfigure}
\begin{subfigure}{\wfig}
\centering
\begin{tikzpicture}
\useasboundingbox (-2,-2) rectangle (2,3.7);
\draw [my grid] (-2,-2) rectangle (2,2);
\draw [my grid,my coordarrow] (-1.6,0) -- (1.6,0) node[anchor=north,my cood] {$v$};
\draw [my grid,my coordarrow] (0,-1.6) -- (0,1.6) node[anchor=east,my cood] {$w$};
\newcommand{\Da}{1.0}
\draw [my frame,dashed] (-1.5,-\Da) .. 
      controls (-1.8,-\Da-0.3) 
      and (-2,-\Da-0.5) .. 
      (-2,-2);
\draw [my frame] (-1.5,-\Da) -- (\Da,1.5);
\draw [my frame,dashed] (\Da,1.5) .. 
      controls (\Da+0.3,1.8) 
      and (\Da+0.5,2) .. 
      (2,2);
\draw [my frame,dashed] 
      (-1.5,-2) -- (-2,-2);
\draw [my frame]
      (-1.5,-2) -- (1.5,-2);
\draw [my frame,dashed]
      (1.5,-2) -- (2,-2);
\draw [my frame,dashed] 
      (2,-1.5) -- (2,-2);
\draw [my frame]
      (2,-1.5) -- (2,1.5);
\draw [my frame,dashed]
      (2,1.5) -- (2,2);
\begin{scope}[on background layer]
\fill [black!20!white] 
      (-2,-2) ..controls(-2,-\Da-0.5)and(-1.8,-\Da-0.3).. 
      (-1.5,-\Da) -- 
      (\Da,1.5) ..controls(\Da+0.3,1.8)and(\Da+0.5,2).. 
      (2,2) --
      (2,-2) --
      (-2,-2);
\end{scope}
\filldraw [black] (-2,-2) circle (\csize)
          (2,2) circle (\csize)
          (2,-2) circle (\csize);
\node at (0,1.5-\Da)[left] {$s$};
\node at (-1.65+\Da,-0.05)[above] {$-s$};
\node [thick,circle,draw] (v) at (-0.8,2.8) {$v$};
\node [thick,circle,draw] (w) at (0.8,2.8) {$w$};
\draw [->-=.5] (v) [in=210,out=150,loop] to node[above=6]
      {$0$} (v);
\draw [->-=.5] (v) to [bend right] node[below]
      {$s$} (w);
\draw [->-=.5] (w) to [bend right] node[above]
      {$\infty$} (v);
\draw [->-=.5] (w) [in=30,out=330,loop] to node[above=6]
      {$0$} (w);
\end{tikzpicture}
\caption{A half plane}
\label{fig:catb}
\end{subfigure}
\begin{subfigure}{\wfig}
\centering
\begin{tikzpicture}
\useasboundingbox (-2,-2) rectangle (2,3.7);
\draw [my grid] (-2,-2) rectangle (2,2);
\draw [my grid,my coordarrow] (-1.6,0) -- (1.6,0) node[anchor=north,my cood] {$v$};
\draw [my grid,my coordarrow] (0,-1.6) -- (0,1.6) node[anchor=east,my cood] {$w$};
\newcommand{\Da}{1.0}
\draw [my frame,dashed] (-1.5,-\Da) .. 
      controls (-1.8,-\Da-0.3) 
      and (-2,-\Da-0.5) .. 
      (-2,-2);
\draw [my frame] (-1.5,-\Da) -- (\Da,1.5);
\draw [my frame,dashed] (\Da,1.5) .. 
      controls (\Da+0.3,1.8) 
      and (\Da+0.5,2) .. 
      (2,2);
\newcommand{\Ea}{0.4}
\draw [my frame,dashed] (-\Ea,-1.5) ..
      controls (-\Ea-0.3,-1.8)
      and (-\Ea-0.5,-2) ..
      (-2,-2);
\draw [my frame] (-\Ea,-1.5) -- (1.5,\Ea);
\draw [my frame,dashed] (1.5,\Ea) ..
      controls (1.8,\Ea+0.3)
      and (2,\Ea+0.5) ..
      (2,2);
\begin{scope}[on background layer]
\fill [black!20!white] 
      (-2,-2) ..controls(-2,-\Da-0.5)and(-1.8,-\Da-0.3).. 
      (-1.5,-\Da) -- 
      (\Da,1.5) ..controls(\Da+0.3,1.8)and(\Da+0.5,2).. 
      (2,2) ..controls(2,\Ea+0.5)and(1.8,\Ea+0.3)..
      (1.5,\Ea) --
      (-\Ea,-1.5) ..controls(-\Ea-0.3,-1.8)and(-\Ea-0.5,-2)..
      (-2,-2);
\end{scope}
\filldraw [black] (-2,-2) circle (\csize)
          (2,2) circle (\csize);
\node at (0,1.5-\Da)[left] {$s$};
\node at (-1.65+\Da,-0.05)[above] {$-s$};
\node at (1.5-\Ea,0)[below] {$t$};
\node at (0,-1.5+\Ea)[right] {$-t$};
\node [thick,circle,draw] (v) at (-0.8,2.8) {$v$};
\node [thick,circle,draw] (w) at (0.8,2.8) {$w$};
\draw [->-=.5] (v) [in=210,out=150,loop] to node[above=6]
      {$0$} (v);
\draw [->-=.5] (v) to [bend right] node[below]
      {$s$} (w);
\draw [->-=.5] (w) to [bend right] node[above]
      {$t$} (v);
\draw [->-=.5] (w) [in=30,out=330,loop] to node[above=6]
      {$0$} (w);
\end{tikzpicture}
\caption{A band}
\label{fig:catc}
\end{subfigure}
\begin{subfigure}{\wfigg}
\centering
\begin{tikzpicture}
\useasboundingbox (-2,-2) rectangle (2,4.5);
\draw [my grid] (-2,-2) rectangle (2,2);
\draw [my grid,my coordarrow] (-1.6,0) -- (1.6,0) node[anchor=north,my cood] {$v$};
\draw [my grid,my coordarrow] (0,-1.6) -- (0,1.6) node[anchor=east,my cood] {$w$};
\draw [my frame,dashed] 
      (-1.5,-2) -- (-2,-2);
\draw [my frame]
      (-1.5,-2) -- (1.5,-2);
\draw [my frame,dashed]
      (1.5,-2) -- (2,-2); %
\draw [my frame,dashed] 
      (2,-1.5) -- (2,-2);
\draw [my frame]
      (2,-1.5) -- (2,1.5);
\draw [my frame,dashed]
      (2,1.5) -- (2,2); %
\filldraw [black] (-2,-2) circle (\csize)
          (2,2) circle (\csize)
          (2,-2) circle (\csize);
\node [thick,circle,draw] (v) at (-0.8,2.8) {$v$};
\node [thick,circle,draw] (w) at (0.8,2.8) {$w$};
\draw [->-=.5] (v) [in=210,out=150,loop] to node[above=6]
      {$0$} (v);
\draw [->-=.5] (v) to [bend right] node[below]
      {$-\infty$} (w);
\draw [->-=.5] (w) to [bend right] node[above]
      {$\infty$} (v);
\draw [->-=.5] (w) [in=30,out=330,loop] to node[above=6]
      {$0$} (w);
\end{tikzpicture}
\caption{Two orthogonal lines}
\label{fig:catd}
\end{subfigure}
\begin{subfigure}{\wfigg}
\centering
\begin{tikzpicture}
\useasboundingbox (-2,-2) rectangle (2,4.5);
\draw [my grid] (-2,-2) rectangle (2,2);
\draw [my grid,my coordarrow] (-1.6,0) -- (1.6,0) node[anchor=north,my cood] {$v$};
\draw [my grid,my coordarrow] (0,-1.6) -- (0,1.6) node[anchor=east,my cood] {$w$};
\draw [my frame,dashed] 
      (-1.5,-2) -- (-2,-2);
\draw [my frame]
      (-1.5,-2) -- (1.5,-2);
\draw [my frame,dashed]
      (1.5,-2) -- (2,-2); %
\draw [my frame,dashed] 
      (-1.5,2) -- (-2,2);
\draw [my frame]
      (-1.5,2) -- (1.5,2);
\draw [my frame,dashed]
      (1.5,2) -- (2,2); %
\filldraw [black] (-2,-2) circle (\csize)
          (2,2) circle (\csize)
          (2,-2) circle (\csize)
          (-2,2) circle (\csize);
\node [thick,circle,draw] (v) at (-0.8,2.8) {$v$};
\node [thick,circle,draw] (w) at (0.8,2.8) {$w$};
\draw [->-=.5] (v) [in=210,out=150,loop] to node[above=6]
      {$0$} (v);
\draw [->-=.5] (v) to [bend right] node[below]
      {$\infty$} (w);
\draw [->-=.5] (w) to [bend right] node[above]
      {$\infty$} (v);
\draw [->-=.5] (w) [in=30,out=330,loop] to node[above=6]
      {$-\infty$} (w);
\end{tikzpicture}
\caption{Two parallel lines}
\label{fig:cate}
\end{subfigure}
\caption*{\phantom{$\overline{\mathbb{K}}$}}
\end{figure}

\begin{figure}[!p]
\ContinuedFloat 
\centering
\begin{subfigure}{\wfig}
\centering
\begin{tikzpicture}
\useasboundingbox (-2,-2) rectangle (2,3.7);
\draw [my grid] (-2,-2) rectangle (2,2);
\draw [my grid,my coordarrow] (-1.6,0) -- (1.6,0) node[anchor=north,my cood] {$v$};
\draw [my grid,my coordarrow] (0,-1.6) -- (0,1.6) node[anchor=east,my cood] {$w$};
\draw [my frame,dashed] 
      (-1.5,-2) -- (-2,-2);
\draw [my frame]
      (-1.5,-2) -- (1.5,-2);
\draw [my frame,dashed]
      (1.5,-2) -- (2,-2); %
\filldraw [black] (-2,-2) circle (\csize)
          (2,2) circle (\csize)
          (2,-2) circle (\csize);
\node [thick,circle,draw] (v) at (-0.8,2.8) {$v$};
\node [thick,circle,draw] (w) at (0.8,2.8) {$w$};
\draw [->-=.5] (v) [in=210,out=150,loop] to node[above=6]
      {$0$} (v);
\draw [->-=.5] (v) to [bend right] node[below]
      {$-\infty$} (w);
\draw [->-=.5] (w) to [bend right] node[above]
      {$\infty$} (v);
\draw [->-=.5] (w) [in=30,out=330,loop] to node[above=6]
      {$-\infty$} (w);
\end{tikzpicture}
\caption{A line and a point}
\label{fig:catf}
\end{subfigure}
\begin{subfigure}{\wfig}
\centering
\begin{tikzpicture}
\useasboundingbox (-2,-2) rectangle (2,3.7);
\draw [my grid] (-2,-2) rectangle (2,2);
\draw [my grid,my coordarrow] (-1.6,0) -- (1.6,0) node[anchor=north,my cood] {$v$};
\draw [my grid,my coordarrow] (0,-1.6) -- (0,1.6) node[anchor=east,my cood] {$w$};
\draw [my frame,dashed] 
      (2,-1.5) -- (2,-2);
\draw [my frame]
      (2,-1.5) -- (2,1.5);
\draw [my frame,dashed]
      (2,1.5) -- (2,2); %
\filldraw [black] (-2,-2) circle (\csize)
          (2,2) circle (\csize)
          (2,-2) circle (\csize);
\node [thick,circle,draw] (v) at (-0.8,2.8) {$v$};
\node [thick,circle,draw] (w) at (0.8,2.8) {$w$};
\draw [->-=.5] (v) [in=210,out=150,loop] to node[above=6]
      {$0$} (v);
\draw [->-=.5] (v) to [bend right] node[below]
      {$\infty$} (w);
\draw [->-=.5] (w) to [bend right] node[above]
      {$-\infty$} (v);
\draw [->-=.5] (w) [in=30,out=330,loop] to node[above=6]
      {$-\infty$} (w);
\end{tikzpicture}
\caption{A line and a point}
\label{fig:catg}
\end{subfigure}
\begin{subfigure}{\wfig}
\centering
\begin{tikzpicture}
\useasboundingbox (-2,-2) rectangle (2,3.7);
\draw [my grid] (-2,-2) rectangle (2,2);
\draw [my grid,my coordarrow] (-1.6,0) -- (1.6,0) node[anchor=north,my cood] {$v$};
\draw [my grid,my coordarrow] (0,-1.6) -- (0,1.6) node[anchor=east,my cood] {$w$};
\filldraw [black] (-2,-2) circle (\csize)
          (2,2) circle (\csize)
          (2,-2) circle (\csize)
          (-2,2) circle (\csize);
\node [thick,circle,draw] (v) at (-0.8,2.8) {$v$};
\node [thick,circle,draw] (w) at (0.8,2.8) {$w$};
\draw [->-=.5] (v) [in=210,out=150,loop] to node[above=6]
      {$-\infty$} (v);
\draw [->-=.5] (v) to [bend right] node[below]
      {$\infty$} (w);
\draw [->-=.5] (w) to [bend right] node[above]
      {$\infty$} (v);
\draw [->-=.5] (w) [in=30,out=330,loop] to node[above=6]
      {$-\infty$} (w);
\end{tikzpicture}
\caption{Four points}
\label{fig:cath}
\end{subfigure}
\begin{subfigure}{\wfigg}
\centering
\begin{tikzpicture}
\useasboundingbox (-2,-2) rectangle (2,4.5);
\draw [my grid] (-2,-2) rectangle (2,2);
\draw [my grid,my coordarrow] (-1.6,0) -- (1.6,0) node[anchor=north,my cood] {$v$};
\draw [my grid,my coordarrow] (0,-1.6) -- (0,1.6) node[anchor=east,my cood] {$w$};
\filldraw [black] (-2,-2) circle (\csize)
          (2,2) circle (\csize)
          (2,-2) circle (\csize);
\node [thick,circle,draw] (v) at (-0.8,2.8) {$v$};
\node [thick,circle,draw] (w) at (0.8,2.8) {$w$};
\draw [->-=.5] (v) [in=210,out=150,loop] to node[above=6]
      {$-\infty$} (v);
\draw [->-=.5] (v) to [bend right] node[below]
      {$-\infty$} (w);
\draw [->-=.5] (w) to [bend right] node[above]
      {$\infty$} (v);
\draw [->-=.5] (w) [in=30,out=330,loop] to node[above=6]
      {$-\infty$} (w);
\end{tikzpicture}
\caption{Three points}
\label{fig:cati}
\end{subfigure}
\begin{subfigure}{\wfigg}
\centering
\begin{tikzpicture}
\useasboundingbox (-2,-2) rectangle (2,4.5);
\draw [my grid] (-2,-2) rectangle (2,2);
\draw [my grid,my coordarrow] (-1.6,0) -- (1.6,0) node[anchor=north,my cood] {$v$};
\draw [my grid,my coordarrow] (0,-1.6) -- (0,1.6) node[anchor=east,my cood] {$w$};
\filldraw [black] (-2,-2) circle (\csize)
          (2,2) circle (\csize);
\node [thick,circle,draw] (v) at (-0.8,2.8) {$v$};
\node [thick,circle,draw] (w) at (0.8,2.8) {$w$};
\draw [->-=.5] (v) [in=210,out=150,loop] to node[above=6]
      {$-\infty$} (v);
\draw [->-=.5] (v) to [bend right] node[below]
      {$-\infty$} (w);
\draw [->-=.5] (w) to [bend right] node[above]
      {$-\infty$} (v);
\draw [->-=.5] (w) [in=30,out=330,loop] to node[above=6]
      {$-\infty$} (w);
\end{tikzpicture}
\caption{Two points}
\label{fig:catj}
\end{subfigure}
\caption{$\overline{\mathbb{K}}$-extended L-convex sets and $\overline{\mathbb{K}}$-categories ($s,t\in \mathbb{K},s+t\geq 0$)}
\label{fig:lcskcat}
\end{figure}
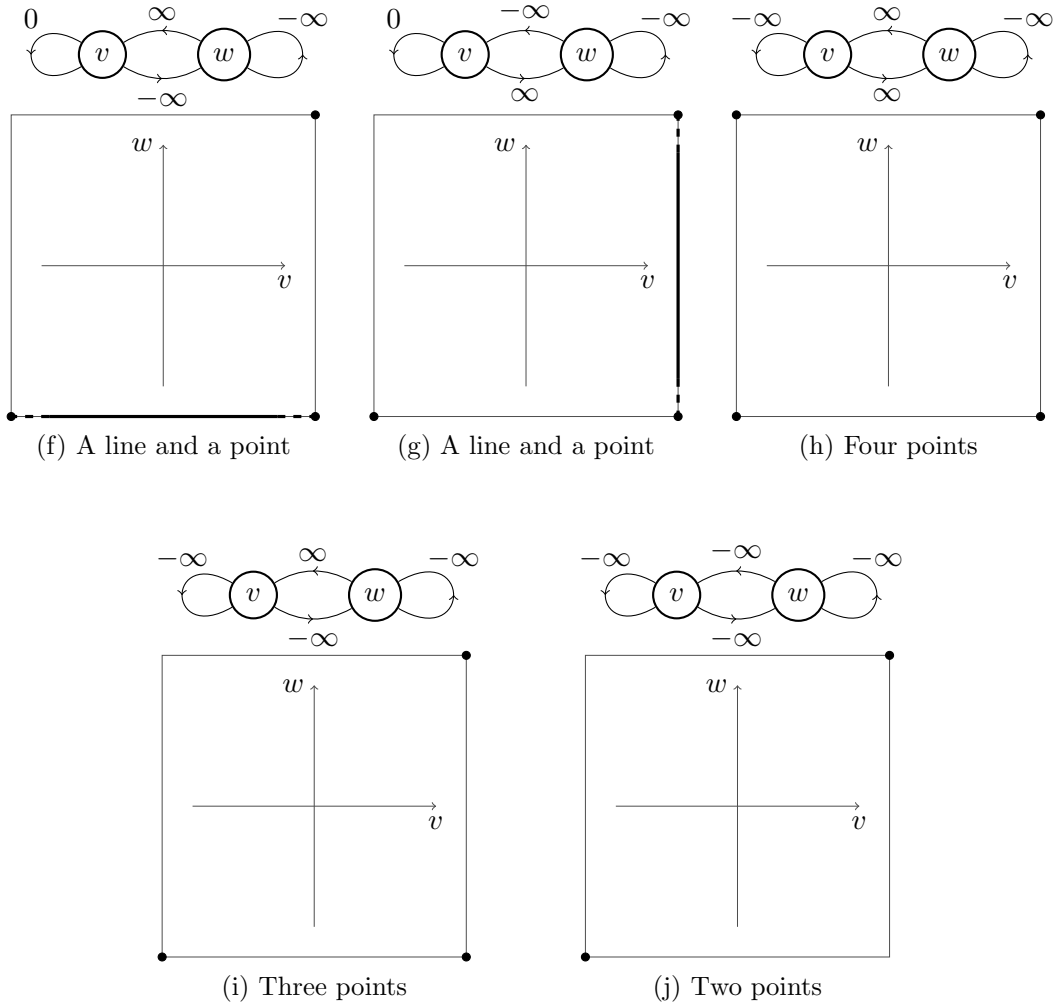
\afterpage{\clearpage}

\paragraph{}
We then show the duality for maps.
We denote the constructions for maps by the same
symbols
$[-,\overline{\mathbb{K}}]$
and 
$\db{-,\overline{\mathbb{K}}}$.
The theorem reads:
\begin{thm}\label{mainthm_map}
Let $\mathcal{A}=(\ob{\mathcal{A}},d_\mathcal{A})$ 
and $\mathcal{B}=(\ob{\mathcal{B}},d_\mathcal{B})$ 
be $\overline{\mathbb{K}}$-categories, and 
$\lcs{D}=(\ind{\lcs{D}},\lcs{D}_0)$ 
and $\lcs{E}=(\ind{\lcs{E}},\lcs{E}_0)$ 
be $\overline{\mathbb{K}}$-extended L-convex sets.
Then the following hold:
\begin{enumerate}
\item
The map 
$[-,\overline{\mathbb{K}}]\colon[\mathcal{A},\mathcal{B}]_0\longrightarrow\db{[\mathcal{B},\overline{\mathbb{K}}],[\mathcal{A},\overline{\mathbb{K}}]}_0$
is bijective.
\item
The map 
$\db{-,\overline{\mathbb{K}}}\colon\db{\lcs{D},\lcs{E}}_0\longrightarrow[\db{\lcs{E},\overline{\mathbb{K}}},\db{\lcs{D},\overline{\mathbb{K}}}]_0$
is bijective.
\end{enumerate}
\end{thm}
Note that the directions of maps reverse
when we apply the constructions.
The details of these constructions are 
given by the following lemmas:

\begin{lem}\label{lem:cat2lcs_hom}
Let $\mathcal{A}=(\ob{\mathcal{A}},d_\mathcal{A})$ 
and $\mathcal{B}=(\ob{\mathcal{B}},d_\mathcal{B})$ 
be $\overline{\mathbb{K}}$-categories and 
\[
F\colon \mathcal{A}\longrightarrow \mathcal{B}
\]
be a $\overline{\mathbb{K}}$-functor.
Then there is a homomorphism of 
$\overline{\mathbb{K}}$-extended L-convex sets 
\[
[F,\overline{\mathbb{K}}]\colon [\mathcal{B},\overline{\mathbb{K}}]\longrightarrow[\mathcal{A},\overline{\mathbb{K}}],
\]
defined as $\ind{[F,\overline{\mathbb{K}}]}=\ob{F}$.
\end{lem}
\begin{proof}
Since 
$[\mathcal{B},\overline{\mathbb{K}}]=(\ob{\mathcal{B}},[\mathcal{B},\overline{\mathbb{K}}]_0)$ and
$[\mathcal{A},\overline{\mathbb{K}}]=(\ob{\mathcal{A}},[\mathcal{A},\overline{\mathbb{K}}]_0)$,
a homomorphism from $[\mathcal{B},\overline{\mathbb{K}}]$
to $[\mathcal{A},\overline{\mathbb{K}}]$ is given 
as a map from
$\ind{[\mathcal{A},\overline{\mathbb{K}}]}=\ob{\mathcal{A}}$ to 
$\ind{[\mathcal{B},\overline{\mathbb{K}}]}=\ob{\mathcal{B}}$,
such as $\ob{F}$;
so at least the type matches.
What we have to check is that 
for all $p\in[\mathcal{B},\overline{\mathbb{K}}]_0$,
\begin{align}\label{nonexp_comp}
	p\circ \ob{F}\colon \ob{\mathcal{A}}\overset{\ob{F}}{\longrightarrow} \ob{\mathcal{B}}\overset{p}{\longrightarrow} \mathbb{K}\cup\{-\infty,\infty\}
\end{align}
is an element of 
$[\mathcal{A},\overline{\mathbb{K}}]_0$, i.e., 
nonexpansive.
But (\ref{nonexp_comp}) is a composition of 
nonexpansive maps and so nonexpansive as well.
\end{proof}

\begin{lem}\label{lem:lcs2cat_hom}
Let $\lcs{D}=(\ind{\lcs{D}},\lcs{D}_0)$ 
and $\lcs{E}=(\ind{\lcs{E}},\lcs{E}_0)$ 
be $\overline{\mathbb{K}}$-extended L-convex sets and 
\[
\Phi\colon \lcs{D}\longrightarrow \lcs{E}
\]
be a homomorphism.
Then there is a $\overline{\mathbb{K}}$-functor 
\[
\db{\Phi,\overline{\mathbb{K}}}\colon \db{\lcs{E},\overline{\mathbb{K}}}\longrightarrow\db{\lcs{D},\overline{\mathbb{K}}},
\]
defined as  
$\ob{\db{\Phi,\overline{\mathbb{K}}}}=\pi_\lcs{D}\circ \ind{\Phi}\circ \pi_\lcs{E}^{-1}$, 
where $\pi_\lcs{D}$ and $\pi_\lcs{E}$ are the bijections 
defined in the proof of Lemma \ref{lcs2kcat_ob}.
\end{lem}
\begin{proof}
Since $\db{\lcs{E},\overline{\mathbb{K}}}=
\big(\db{\lcs{E},\overline{\mathbb{K}}}_0,d_\db{\lcs{E},\overline{\mathbb{K}}}\big)$ 
and
$\db{\lcs{D},\overline{\mathbb{K}}}=
\big(\db{\lcs{D},\overline{\mathbb{K}}}_0,d_\db{\lcs{D},\overline{\mathbb{K}}}\big)$,
a $\overline{\mathbb{K}}$-functor from 
$\db{\lcs{E},\overline{\mathbb{K}}}$ to
$\db{\lcs{D},\overline{\mathbb{K}}}$ is given as a 
map from 
$\ob{\db{\lcs{E},\overline{\mathbb{K}}}}=\db{\lcs{E},\overline{\mathbb{K}}}_0$ to
$\ob{\db{\lcs{D},\overline{\mathbb{K}}}}=\db{\lcs{D},\overline{\mathbb{K}}}_0$, 
such as 
$\ob{\db{\Phi,\overline{\mathbb{K}}}}=\pi_\lcs{D}\circ \ind{\Phi}\circ \pi_\lcs{E}^{-1}$:
\begin{center}
\begin{tikzpicture}
\node (P1) at (3,2) {$\db{\lcs{E},\overline{\mathbb{K}}}_0$} ;
\node (P2) at (3,0) {$\db{\lcs{D},\overline{\mathbb{K}}}_0$};
\node (P3) at (0,2) {$\ind{\lcs{E}}$};
\node (P4) at (0,0) {$\ind{\lcs{D}}$};
\draw
(P1) edge[my arrow] node[right] {$\ob{\db{\Phi,\overline{\mathbb{K}}}}$} (P2)
(P3) edge[my arrow] node[above] {$\pi_\lcs{E}$} (P1)
(P4) edge[my arrow] node[below] {$\pi_\lcs{D}$} (P2)
(P3) edge[my arrow] node[left] {$\ind{\Phi}$} (P4);
\end{tikzpicture}
\end{center}
We adopt the notation
$\db{\lcs{E},\overline{\mathbb{K}}}_0=\{\pi_w\}_{w\in \ind{\lcs{E}}}$
and 
$\db{\lcs{D},\overline{\mathbb{K}}}_0=\{\pi_v\}_{v\in \ind{\lcs{D}}}$;
then 
\[
\ob{\db{\Phi,\overline{\mathbb{K}}}}(\pi_w)=\pi_{\ind{\Phi}(w)}\quad (\forall \pi_w\in \{\pi_w\}_{w\in \ind{\lcs{E}}}).
\]
We aim to show that 
\[
d_\db{\lcs{E},\overline{\mathbb{K}}}(\pi_w,\pi_{w^\prime})
\geq d_\db{\lcs{D},\overline{\mathbb{K}}}(\pi_{\ind{\Phi}(w)},\pi_{\ind{\Phi}(w^\prime)})\quad (\forall \pi_w,\pi_{w^\prime}\in \{\pi_w\}_{w\in \ind{\lcs{E}}}),
\]
or equivalently 
(by (\ref{d_dk}) in the proof of Lemma 
\ref{lcs2kcat_ob}),
\[
\sup_{q\in \lcs{E}_0}\{q(w^\prime)-q(w)\}
\geq \sup_{p\in \lcs{D}_0} \left\{p\circ\ind{\Phi}(w^\prime)-p\circ\ind{\Phi}(w)\right\}\quad (\forall w,w^\prime\in \ind{\lcs{E}}).
\]
Now recall that the condition for $\ind{\Phi}$ to
define a homomorphism is that for all $p\in \lcs{D}_0$,
$p\circ\ind{\Phi}\in \lcs{E}_0$; hence
\[
\{q(w^\prime)-q(w)\mid q\in \lcs{E}_0\}\supseteq 
\left\{p\circ\ind{\Phi}(w^\prime)-p\circ\ind{\Phi}(w)\mid p\in \lcs{D}_0\right\}\quad (\forall w,w^\prime\in \ind{\lcs{E}})
\]
and the proof is done.
\end{proof}

Theorem \ref{mainthm_map} is proved as follows:

\begin{proof}[Proof of Theorem~\ref{mainthm_map}]
\begin{description}[font=\normalfont]
\item[[(i)\!\!\!]]
Since for each $\overline{\mathbb{K}}$-functor 
$F\colon\mathcal{A}\longrightarrow\mathcal{B}$,
$[F,\overline{\mathbb{K}}]$ is defined as
\[
\ind{[F,\overline{\mathbb{K}}]}=\ob{F},
\]
$[-,\overline{\mathbb{K}}]$ is clearly injective.

To prove surjectivity of $[-,\overline{\mathbb{K}}]$,
take an arbitrary homomorphism
\[
\Phi\colon[\mathcal{B},\overline{\mathbb{K}}]\longrightarrow[\mathcal{A},\overline{\mathbb{K}}].
\]
Let 
$F_\Phi=\unit{\mathcal{B}}^{-1}\circ\db{\Phi,\overline{\mathbb{K}}}\circ\unit{\mathcal{A}}$,
where $\overline{\mathbb{K}}$-functors
$\unit{\mathcal{A}}$ and $\unit{\mathcal{B}}$ are 
the isomorphisms defined in the proof of 
Theorem \ref{mainthm_ob} (i):
\begin{center}
\begin{tikzpicture}
\node (P1) at (0,2) {$\mathcal{A}$} ;
\node (P2) at (3,2) {$\db{[\mathcal{A},\overline{\mathbb{K}}],\overline{\mathbb{K}}}$};
\node (P3) at (0,0) {$\mathcal{B}$};
\node (P4) at (3,0) {$\db{[\mathcal{B},\overline{\mathbb{K}}],\overline{\mathbb{K}}}$};
\draw
(P1) edge[my arrow] node[above] {$\unit{\mathcal{A}}$} (P2)
(P1) edge[my arrow] node[left] {$F_\Phi$} (P3)
(P2) edge[my arrow] node[right] {$\db{\Phi,\overline{\mathbb{K}}}$} (P4)
(P3) edge[my arrow] node[below] {$\unit{\mathcal{B}}$} (P4);
\end{tikzpicture}
\end{center}
We claim that $[F_\Phi,\overline{\mathbb{K}}]=\Phi$.
Recall
$\ob{\db{[\mathcal{A},\overline{\mathbb{K}}],\overline{\mathbb{K}}}}=\db{[\mathcal{A},\overline{\mathbb{K}}],\overline{\mathbb{K}}}_0=\{\pi_a\}_{a\in\ob{\mathcal{A}}}$
and
\[
\ob{\unit{\mathcal{A}}}=\pi_{[\mathcal{A},\overline{\mathbb{K}}]}=\lambda a\in\ob{\mathcal{A}}.\,\pi_a \colon\ob{\mathcal{A}}\longrightarrow\ob{\db{[\mathcal{A},\overline{\mathbb{K}}],\overline{\mathbb{K}}}}.
\]
Now we can conclude that
\begin{align*}
\ind{[F_\Phi,\overline{\mathbb{K}}]}&=\ob{F_\Phi}\\
&=\ob{\unit{\mathcal{B}}}^{-1}\circ\ob{\db{\Phi,\overline{\mathbb{K}}}}\circ\ob{\unit{\mathcal{A}}}\\
&=\ob{\unit{\mathcal{B}}}^{-1}\circ\pi_{[\mathcal{B},\overline{\mathbb{K}}]}\circ\ind{\Phi}\circ\pi_{[\mathcal{A},\overline{\mathbb{K}}]}^{-1}\circ\ob{\unit{\mathcal{A}}}\\
&=1_{\ob{\mathcal{B}}}\circ\ind{\Phi}\circ1_{\ob{\mathcal{A}}}\\
&=\ind{\Phi},
\end{align*}
hence $[F_\Phi,\overline{\mathbb{K}}]=\Phi$.
\item[[(ii)\!\!\!]]
Since for each homomorphism 
$\Phi\colon\lcs{D}\longrightarrow\lcs{E}$,
$\db{\Phi,\overline{\mathbb{K}}}$ is defined as
\[
\ob{\db{\Phi,\overline{\mathbb{K}}}}=\pi_\lcs{D}\circ\ind{\Phi}\circ\pi_\lcs{E}^{-1},
\]
where $\pi_\lcs{D}$ and $\pi_\lcs{E}$ are fixed 
bijections, 
$\db{-,\overline{\mathbb{K}}}$ is clearly injective.

To prove surjectivity of $\db{-,\overline{\mathbb{K}}}$,
take an arbitrary $\overline{\mathbb{K}}$-functor
\[
F\colon\db{\lcs{E},\overline{\mathbb{K}}}\longrightarrow\db{\lcs{D},\overline{\mathbb{K}}}.
\]
Let 
$\Phi_F=\counit{\lcs{E}}\circ[F,\overline{\mathbb{K}}]\circ\counit{\lcs{D}}^{-1}$,
or equivalently,
$\ind{\Phi_F}=\ind{\counit{\lcs{D}}}^{-1}\circ\ind{[F,\overline{\mathbb{K}}]}\circ\ind{\counit{\lcs{E}}}$,
where homomorphisms
$\counit{\lcs{D}}$ and $\counit{\lcs{E}}$ are 
the isomorphisms defined in the proof of 
Theorem \ref{mainthm_ob} (ii):
\begin{center}
\begin{tikzpicture}
\node (P1) at (0,2) {$[\db{\lcs{D},\overline{\mathbb{K}}},\overline{\mathbb{K}}]$} ;
\node (P2) at (3,2) {$\lcs{D}$};
\node (P3) at (0,0) {$[\db{\lcs{E},\overline{\mathbb{K}}},\overline{\mathbb{K}}]$};
\node (P4) at (3,0) {$\lcs{E}$};
\draw
(P1) edge[my arrow] node[above] {$\counit{\lcs{D}}$} (P2)
(P1) edge[my arrow] node[left] {$[F,\overline{\mathbb{K}}]$} (P3)
(P2) edge[my arrow] node[right] {$\Phi_F$} (P4)
(P3) edge[my arrow] node[below] {$\counit{\lcs{E}}$} (P4);
\end{tikzpicture}
\end{center}
We claim that $\db{\Phi_F,\overline{\mathbb{K}}}=F$.
Recall 
$\ind{[\db{\lcs{D},\overline{\mathbb{K}}},\overline{\mathbb{K}}]}=\ob{\db{\lcs{D},\overline{\mathbb{K}}}}=\db{\lcs{D},\overline{\mathbb{K}}}_0=\{\pi_v\}_{v\in\ind{\lcs{D}}}$
and
\begin{align*}
\ind{\counit{\lcs{D}}}=\pi_\lcs{D}=\lambda v\in\ind{\lcs{D}}.\,\pi_v\colon \ind{\lcs{D}}\longrightarrow \ind{[\db{\lcs{D},\overline{\mathbb{K}}},\overline{\mathbb{K}}]}.
\end{align*}
Now we can conclude that
\begin{align*}
\ob{\db{\Phi_F,\overline{\mathbb{K}}}}&=
\pi_\lcs{D}\circ\ind{\Phi_F}\circ\pi_\lcs{E}^{-1}\\
&=\pi_\lcs{D}\circ\ind{\counit{\lcs{D}}}^{-1}\circ\ind{[F,\overline{\mathbb{K}}]}\circ\ind{\counit{\lcs{E}}}\circ\pi_\lcs{E}^{-1}\\
&=1_{\ind{[\db{\lcs{D},\overline{\mathbb{K}}},\overline{\mathbb{K}}]}}\circ\ind{[F,\overline{\mathbb{K}}]}\circ 1_{\ind{[\db{\lcs{E},\overline{\mathbb{K}}},\overline{\mathbb{K}}]}}\\
&=\ind{[F,\overline{\mathbb{K}}]}\\
&=\ob{F},
\end{align*}
hence $\db{\Phi_F,\overline{\mathbb{K}}}=F$.
\end{description}
\end{proof}

Finally, the duality for canonical orderings:

\begin{thm}\label{mainthm_canord}
Let $\mathcal{A}=(\ob{\mathcal{A}},d_\mathcal{A})$ and
$\mathcal{B}=(\ob{\mathcal{B}},d_\mathcal{B})$ be 
$\overline{\mathbb{K}}$-categories,
$\lcs{D}=(\ind{\lcs{D}},\lcs{D}_0)$ and 
$\lcs{E}=(\ind{\lcs{E}},\lcs{E}_0)$ be 
$\overline{\mathbb{K}}$-extended L-convex sets,
$F,G\colon\mathcal{A}\longrightarrow\mathcal{B}$ be
$\overline{\mathbb{K}}$-functors, and
$\Phi,\Psi\colon\lcs{D}\longrightarrow\lcs{E}$ be
homomorphisms.
Then the following hold:
\begin{enumerate}
\item
$F\Rightarrow G$ if and only if
$[F,\overline{\mathbb{K}}]\Rightarrow[G,\overline{\mathbb{K}}]$.
\item
$\Phi\Rightarrow\Psi$ if and only if 
$\db{\Phi,\overline{\mathbb{K}}}\Rightarrow\db{\Psi,\overline{\mathbb{K}}}$.
\end{enumerate}
\end{thm}
\begin{proof}
\begin{description}[font=\normalfont]
\item[[(i), the ``only if'' part\!\!\!]]
Since 
$[F,\overline{\mathbb{K}}],[G,\overline{\mathbb{K}}]\colon[\mathcal{B},\overline{\mathbb{K}}]\longrightarrow[\mathcal{A},\overline{\mathbb{K}}]$,
we aim to prove that for all 
$q\in[\mathcal{B},\overline{\mathbb{K}}]_0$,
\[
q\circ \ob{F}(a)\geq q\circ \ob{G}(a) \quad(\forall a\in\ob{\mathcal{A}}).
\]
In fact, when we view $q$ as a
$\overline{\mathbb{K}}$-functor, it is equivalent to
$q\circ F\Rightarrow q\circ G$,
which follows from the fact that $\Rightarrow$ on
$\overline{\mathbb{K}}$-functors is reflexive and
preserved under composition.
\item[[(ii), the ``only if'' part\!\!\!]]
Since 
$\db{\Phi,\overline{\mathbb{K}}},\db{\Psi,\overline{\mathbb{K}}}\colon\db{\lcs{E},\overline{\mathbb{K}}}\longrightarrow\db{\lcs{D},\overline{\mathbb{K}}}$,
we aim to prove that for all
$\pi_w\in\db{\lcs{E},\overline{\mathbb{K}}}$,
\[
d_\db{\lcs{D},\overline{\mathbb{K}}}(\pi_{\ind{\Phi}(w)},\pi_{\ind{\Psi}(w)})=\sup_{p\in\lcs{D}_0}\{\Psi_0(p)(w)-\Phi_0(p)(w)\}\leq0.
\]
Because 
$0\geq\Psi_0(p)(w)-\Phi_0(p)(w)\iff\Phi_0(p)(w)\geq\Psi_0(p)(w)$ 
by
\begin{prooftree}
	\def\fCenter{\ \geq\ }
	\alwaysDoubleLine
	\Axiom$0\fCenter\Psi_0(p)(w)-\Phi_0(p)(w)$
	\UnaryInf$\Phi_0(p)(w)\fCenter \Psi_0(p)(w)$
\end{prooftree}
it suffices to show that
\[
\Phi_0(p)\geq\Psi_0(p)
\]
holds in $\lcs{E}_0$ for all $p\in\lcs{D}_0$, but this is 
nothing but the condition for $\Phi\Rightarrow\Psi$.
\item[[(i), the ``if'' part\!\!\!]]
The ``only if'' part of (ii) implies that 
$\db{[F,\overline{\mathbb{K}}],\overline{\mathbb{K}}}\Rightarrow\db{[G,\overline{\mathbb{K}}],\overline{\mathbb{K}}}$
holds.
By the proof of Theorem \ref{mainthm_map} (i),
we have the following commutative diagram:
\begin{center}
\begin{tikzpicture}
\node (P1) at (0,2) {$\mathcal{A}$} ;
\node (P2) at (3,2) {$\db{[\mathcal{A},\overline{\mathbb{K}}],\overline{\mathbb{K}}}$};
\node (P3) at (0,0) {$\mathcal{B}$};
\node (P4) at (3,0) {$\db{[\mathcal{B},\overline{\mathbb{K}}],\overline{\mathbb{K}}}$};
\draw
(P1) edge[my arrow] node[above] {$\unit{\mathcal{A}}$} (P2)
(P1) edge[my arrow] node[left] {$F$} (P3)
(P2) edge[my arrow] node[right] {$\db{[F,\overline{\mathbb{K}}],\overline{\mathbb{K}}}$} (P4)
(P3) edge[my arrow] node[below] {$\unit{\mathcal{B}}$} (P4);
\end{tikzpicture}
\end{center}
and similarly for $G$. 
In equations, 
\begin{align*}
F&=\unit{\mathcal{B}}^{-1}\circ\db{[F,\overline{\mathbb{K}}],\overline{\mathbb{K}}}\circ\unit{\mathcal{A}},\\
G&=\unit{\mathcal{B}}^{-1}\circ\db{[G,\overline{\mathbb{K}}],\overline{\mathbb{K}}}\circ\unit{\mathcal{A}}.
\end{align*}
Since $\unit{\mathcal{A}}$, $\unit{\mathcal{B}}$ are
isomorphisms, we have
\begin{align*}
d_\mathcal{B}(F(a),G(a))&=
d_\mathcal{B}(\unit{\mathcal{B}}^{-1}\circ\db{[F,\overline{\mathbb{K}}],\overline{\mathbb{K}}}\circ\unit{\mathcal{A}}(a),\unit{\mathcal{B}}^{-1}\circ\db{[G,\overline{\mathbb{K}}],\overline{\mathbb{K}}}\circ\unit{\mathcal{A}}(a))\\
&=d_\mathcal{B}(\unit{\mathcal{B}}^{-1}\circ\db{[F,\overline{\mathbb{K}}],\overline{\mathbb{K}}}(\pi_a),\unit{\mathcal{B}}^{-1}\circ\db{[G,\overline{\mathbb{K}}],\overline{\mathbb{K}}}(\pi_a))\\
&=d_\db{[\mathcal{B},\overline{\mathbb{K}}],\overline{\mathbb{K}}}(\db{[F,\overline{\mathbb{K}}],\overline{\mathbb{K}}}(\pi_a),\db{[G,\overline{\mathbb{K}}],\overline{\mathbb{K}}}(\pi_a))\\
&\leq 0
\end{align*}
for all $a\in \ob{\mathcal{A}}$. 
\item[[(ii), the ``if'' part\!\!\!]]
The ``only if'' part of (i) implies that 
$[\db{\Phi,\overline{\mathbb{K}}},\overline{\mathbb{K}}]\Rightarrow[\db{\Psi,\overline{\mathbb{K}}},\overline{\mathbb{K}}]$
holds.
By the proof of Theorem \ref{mainthm_map} (ii),
we have the following commutative diagram:
\begin{center}
\begin{tikzpicture}
\node (P1) at (0,2) {$[\db{\lcs{D},\overline{\mathbb{K}}},\overline{\mathbb{K}}]$} ;
\node (P2) at (3,2) {$\lcs{D}$};
\node (P3) at (0,0) {$[\db{\lcs{E},\overline{\mathbb{K}}},\overline{\mathbb{K}}]$};
\node (P4) at (3,0) {$\lcs{E}$};
\draw
(P1) edge[my arrow] node[above] {$\counit{\lcs{D}}$} (P2)
(P1) edge[my arrow] node[left] {$[\db{\Phi,\overline{\mathbb{K}}},\overline{\mathbb{K}}]$} (P3)
(P2) edge[my arrow] node[right] {$\Phi$} (P4)
(P3) edge[my arrow] node[below] {$\counit{\lcs{E}}$} (P4);
\end{tikzpicture}
\end{center}
and similarly for $\Psi$. 
In equations,
\begin{align*}
\Phi&=\counit{\lcs{E}}\circ[\db{\Phi,\overline{\mathbb{K}}},\overline{\mathbb{K}}]\circ\counit{\lcs{D}}^{-1},\\
\Psi&=\counit{\lcs{E}}\circ[\db{\Psi,\overline{\mathbb{K}}},\overline{\mathbb{K}}]\circ\counit{\lcs{D}}^{-1}.
\end{align*}
We have
\begin{align*}
\Phi_0(p)(w)&=(\counit{\lcs{E}})_0\circ[\db{\Phi,\overline{\mathbb{K}}},\overline{\mathbb{K}}]_0\circ(\counit{\lcs{D}}^{-1})_0(p)(w)\\
&=(\counit{\lcs{E}})_0\big([\db{\Phi,\overline{\mathbb{K}}},\overline{\mathbb{K}}]_0\circ(\counit{\lcs{D}}^{-1})_0(p)\big)(w)\\
&=\big(\lambda w^\prime\in\ind{\lcs{E}}.\,[\db{\Phi,\overline{\mathbb{K}}},\overline{\mathbb{K}}]_0\circ(\counit{\lcs{D}}^{-1})_0(p)(\pi_{w^\prime})\big)(w)\\
&=[\db{\Phi,\overline{\mathbb{K}}},\overline{\mathbb{K}}]_0\circ(\counit{\lcs{D}}^{-1})_0(p)(\pi_w)\\
&=[\db{\Phi,\overline{\mathbb{K}}},\overline{\mathbb{K}}]_0\big((\counit{\lcs{D}}^{-1})_0(p)\big)(\pi_w)
\end{align*}
and 
\[
\Psi_0(p)(w)=[\db{\Psi,\overline{\mathbb{K}}},\overline{\mathbb{K}}]_0\big((\counit{\lcs{D}}^{-1})_0(p)\big)(\pi_w)
\]
for all $p\in\lcs{D}_0$ and $w\in\ind{\lcs{E}}$.
Since 
$(\counit{\lcs{D}}^{-1})_0(p)\in[\db{\lcs{D},\overline{\mathbb{K}}},\overline{\mathbb{K}}]_0$,
we have 
\[
[\db{\Phi,\overline{\mathbb{K}}},\overline{\mathbb{K}}]_0\big((\counit{\lcs{D}}^{-1})_0(p)\big)(\pi_w)\geq
[\db{\Psi,\overline{\mathbb{K}}},\overline{\mathbb{K}}]_0\big((\counit{\lcs{D}}^{-1})_0(p)\big)(\pi_w).
\]
Therefore we obtain
$\Phi(p)(w)\geq \Psi(p)(w)$
for all $p\in\lcs{D}_0$ and $w\in\ind{\lcs{E}}$;
so $\Phi(p)\geq\Psi(p)$ for all $p\in\lcs{D}_0$;
so $\Phi\Rightarrow\Psi$.
\end{description}
\end{proof}

\chapter{Conclusion}
\section{Summary and Concluding Remarks}
We introduced
$\overline{\mathbb{K}}$-extended L-convex sets, 
homomorphisms 
and canonical orderings
(between homomorphisms),
and established the correspondence of them to 
entities of enriched-categorical origin;
$\overline{\mathbb{K}}$-categories,
$\overline{\mathbb{K}}$-functors 
and 
canonical orderings (between $\overline{\mathbb{K}}$-functors).
The whole correspondence between 
the theory of $\overline{\mathbb{K}}$-extended L-convex 
sets and that of $\overline{\mathbb{K}}$-categories
is so harmonious that
one may dare to say that these two 
theories are in fact the \textit{identical} 
one in different guises.

The formulation of a duality by the function space
(or ``hom'') with a fixed codomain  
living in two worlds (such an object is called
e.g., a \textit{Janusian object}, after Janus, a god in
Roman mythology usually depicted as having
two faces),
as in ours, is a common one.
A classical example is the 
\textit{Stone duality}; see~\cite{Joh86} for
details.
\cite{PT91} is an attempt at a general theory of such dualities.
One can also interpret Birkhoff's representation
theorem stating correspondence between
finite distributive lattices and finite posets,
or its infinite generalization~\cite{Win83},
in a similar fashion;
in fact, it is possible to regard them as the \textbf{2}-version 
of our result (concerning $\overline{\mathbb{K}}$).
In particular, we note that from an enriched-categorical 
viewpoint, finite distributive lattices are nothing
but finite \textbf{2}\textit{-presheaf categories}.

\section{Further Work}
Finally, we conclude the thesis by mentioning  
possible directions of further study.
An obvious problem is generalization of our
duality theorem (on $\overline{\mathbb{K}}$)
to other enriching posets
(or categories) $\mathcal{V}$.
It may also be fruitful to study on another particular 
enriching poset $\mathcal{V}$
and on $\mathcal{V}$-categories;
$\overline{\mathbb{K}}_+^\text{Cart}$ is
an attractive candidate.
Although some results, e.g., Corollary~\ref{cor}, seem to suggest
a strong link between enriched category theory and 
discrete convex analysis, we could not find any 
application which utilizes the link effectively.
A full-scale approach to discrete convex analysis
from a categorical viewpoint is hoped for.

\backmatter
\chapter{Acknowledgments}
First I would like to express my gratitude to my
supervisor Hiroshi Hirai, for allowing me to 
conduct the research in my favorite field 
(category theory),
for suggesting initial directions of study,
and for his occasional comments and encouragement.
In particular, I note that it was his 
suggestion on possible connections between 
distributive lattices and L-convex sets that lead me to 
the series of results.

As for the categorical side, I received many valuable
comments from Ichiro Hasuo, Kazuyuki Asada, and 
the members of the Hasuo Laboratory
at the Department of Computer Science.
I thank them for their kind assistance.

Life in the Mathematical Informatics 2nd Laboratory 
was comfortable and enjoyable;
I am thankful to the members of the laboratory, 
who were always helpful.

Last, but not least, I sincerely thank my family and 
grandparents for their warm and generous support.

\appendix
\end{document}